\documentclass[a4paper,10pt]{article}

\usepackage[english]{babel}
\usepackage[T1]{fontenc}
\usepackage[utf8]{inputenc}
\usepackage{amsmath,amssymb,mathrsfs,tensor}
\usepackage{amscd}
\usepackage{tikz}

\usepackage[colorlinks=true,linkcolor= blue, citecolor=blue]{hyperref}
\usepackage{a4wide}
\numberwithin{equation}{section}
\newtheorem{ass}{Assumption}[section]
\newtheorem{theo}{Theorem}[section]
\newtheorem{cor}{Corollary}[section]

\newtheorem{prop}{Proposition}[section]
\newtheorem{lem}{Lemma}[section]
\newtheorem{de}{Definition}[section]
\newtheorem{rem}{Remark}[section]
\numberwithin{figure}{section}

\date{}

\newcommand{\spp}{{\textbf{\textup{S}}_p}}

\newcommand{\su}{{\textbf{\textup{S}}_1}}
\newcommand{\sd}{{\textbf{\textup{S}}_2}}
\newcommand{\sinf}{{\textbf{\textup{S}}_\infty}}

\newcommand{\lao}{{H_0}}
\newcommand{\la}{{H_V}}

\newcommand{\im}{\operatorname{Im}}
\newcommand{\re}{\operatorname{Re}}

\newenvironment{proof}
	{\textit{Proof.}}
	{\hfill $\square$\vskip 8pt}

\newcommand{\R}{\mathbb{R}}

\newcommand{\C}{\mathbb{C}}

\newcommand{\e}{\mathrm{e}}

\newcommand{\bra}{\langle}
\newcommand{\ket}{\rangle}

\newcommand{\style}{\displaystyle}

\title{On the spectral properties of non-selfadjoint discrete Schr\"odinger operators}
\author{Olivier Bourget, Diomba Sambou and Amal Taarabt}

\begin{document}

\maketitle
\noindent
Facultad de Matem\'aticas, Pontificia Universidad Cat\'olica
 de Chile, Vicu\~na Mackenna 4860, Santiago, Chile 
\\
E-mails: {\it bourget@mat.uc.cl, disambou@mat.uc.cl, amtaarabt@mat.uc.cl}

\begin{abstract}
Let $H_0$ be a purely absolutely continuous selfadjoint operator acting on some separable infinite-dimensional Hilbert space and $V$ be a compact perturbation. We relate the regularity properties of $V$ to various spectral properties of the perturbed operator $H_0+V$. The structures of the discrete spectrum and the embedded eigenvalues are analysed jointly with the existence of limiting absorption principles in a unified framework. Our results are based on a suitable combination of complex scaling techniques, resonance theory and positive commutators methods. Various results scattered throughout the literature are recovered and extended. For illustrative purposes, the case of the one-dimensional discrete Laplacian is emphasized.
\end{abstract}
\bigskip


\bigskip


\tableofcontents

\section{Introduction}\label{intro}

The spectral theory of non-selfadjoint perturbations of selfadjoint operators has made significant advances in the last decade, partly due to its impact and applications in the physical sciences. For an historical panorama on the relationhips between non-selfadjoint operators and quantum mechanics, we refer the reader to e.g. \cite{bgsz} and references therein. The existence of Limiting Absorption Principles (LAP) and the distributional properties of the discrete spectrum have been two of the main issues considered in this field, among the many results obtained so far.

LAP have been evidenced in various contexts, mostly based on the existence of some positive commutators. \cite{royer1,royer2} and \cite{gb} have developed respectively non-selfadjoint versions of the regular Mourre theory and the weak Mourre theory. The analysis of the discrete spectrum has been carried out from a qualitative point of view in \cite{pav1, pav2, ego1, ego2, bog}, where the main focus is set on the existence and structure of the set of limit points. Quantitative approaches based on Lieb-Thirring inequalities and on the identification of the distribution law for the discrete spectrum around the limit points, have also been developed, see e.g. \cite{bru1, chr, goku, bor, dem, fr, gol, ghk, lap, wan} and references therein. These results have been established mostly for specific models like Schr\"odinger operators and Jacobi matrices. Extensions to non-selfadjoint perturbations of magnetic Hamiltonians, Dirac and fractional Schr\"odinger operators have also been studied in \cite{dio14, dio17} and \cite{cue} respectively. Specific developments concerning non-selfadjoint rank one perturbations are considered in \cite{arltse, dudvdo, kis}. In general, the techniques used depend strongly on the model considered. 

The present paper is an attempt to analyze more systematically the issue of the spectral properties of non-selfadjoint compact perturbations of Hamiltonians exhibiting some absolutely continuous spectrum. In particular, we show various relationships between the regularity of the perturbation and, first the properties of the discrete spectrum, second the existence of some LAP. With a view towards discretized models of non-selfadjoint operators \cite{bgsz}, we have focused our discussion on non-selfadjoint compact perturbations of the discrete Schr\"odinger operator in dimension one. For a study of non-selfadjoint discrete Schr\"odinger operator in random regimes, we refer the reader to \cite{dav, dav02}.

We have articulated this paper on the two following axes.

Firstly, for highly regular compact perturbations, we adapt various complex scaling arguments to our non-selfadjoint setting and show that the limit points of the discrete spectrum are necessarily contained in the set of thresholds of the unperturbed Hamiltonian (Theorem \ref{t:tda}). We also exhibit some LAP on some suitable subintervals of the essential spectrum away from these thresholds (Theorem \ref{t:tda+}). Since for perturbations displaying some exponential decay, the method of characteristic values allows to establish the absence of resonance in some neighborhood of the thresholds (Theorem \ref{t1}), we conclude on the finiteness of the discrete spectrum in this case (Theorem \ref{t2}). Let us mention that a variation of this result was previously obtained within the more restricted framework of Jacobi matrices with slightly weaker decay assumptions (see \cite[Theorem 1]{ego1}). Our approach extends that result in two ways: it applies to perturbations exhibiting a full off-diagonal structure and also proves the existence of a LAP, which seems to be new.

Secondly, for mildly regular compact perturbations, we show that the set of embedded eigenvalues away from the thresholds is finite and exhibits more restricted versions of LAP (Theorems \ref{c11sign} and \ref{c11signweak}). It is actually an application of a non-selfadjoint version of the Mourre theory developed under optimal regularity condition (Theorem \ref{mourrensa} and Corollary \ref{mourrensa1}).

The paper is structured as follows. The main concepts and main results are introduced in Section \ref{results} jointly with the model. The complex scaling arguments for non-selfadjoint operators are developed in Section \ref{dilatation} and culminate with the proof of Theorems \ref{t:tda} and \ref{t:tda+}. Section \ref{resonances} is focused on the method of characteristic values and the proof of Theorem \ref{t1}. The development of a Mourre theory for non-selfadjoint operators under optimal regularity conditions is exposed in Section \ref{mourre}, ending with the proofs of Theorem \ref{mourrensa}, Corollary \ref{mourrensa1}, Theorems \ref{c11sign} and \ref{c11signweak}. It is self-contained and can be read independently. Finally, various results related to the concepts of regularity used throughout the text are summarized and illustrated in Section \ref{regclass}. 

\medskip

\noindent
{\bf Notations:} Throughout this paper, ${\mathbb Z}$, ${\mathbb Z}_+$ and ${\mathbb N}$ denote the sets of integral numbers, non-negative and positive 
integral numbers respectively. For $\delta \geq 0$, we define the weighted Hilbert spaces
$$
\ell_{\pm \delta}^{2}({\mathbb Z}) := \big\lbrace x \in {\mathbb C}^{\mathbb Z} : \sum_{n \in {\mathbb Z}} {\rm e}^{\pm \delta \vert n \vert} |x(n)|^2 < \infty \big\rbrace ,
$$
and observe the following inclusions: $\ell_{\delta}^{2}({\mathbb Z}) \subset \ell^{2}({\mathbb Z}) 
\subset \ell_{- \delta}^{2}({\mathbb Z})$. In particular, $\ell^{2}({\mathbb Z}) = \ell_{0}^{2}({\mathbb Z})$. For $\delta > 0$, we define the multiplication operators $W_\delta: \ell_\delta^2({\mathbb Z}) \rightarrow 
\ell^{2}({\mathbb Z})$ by $\left( W_\delta x \right)(n) := {\rm e}^{(\delta/2) \vert n \vert} x(n)$, and $W_{- \delta}: \ell^2({\mathbb Z}) 
\rightarrow \ell_{- \delta}^2({\mathbb Z}) $ by $\left( W_{- \delta} x \right)(n) := {\rm e}^{- (\delta/2) \vert n \vert} x(n)$. We denote by 
$(e_n)_{n\in {\mathbb Z}}$ the canonical orthonormal basis of $\ell^{2}({\mathbb Z})$.

${\mathscr H}$ will denote a separable Hilbert space, $\mathcal{B}({\mathscr H})$ and ${\rm GL}({\mathscr H})$ the algebras of bounded linear operators 
and boundedly invertible linear operators acting on ${\mathscr H}$. $\sinf({\mathscr H})$ and $\spp({\mathscr H})$, $p\geq 1$, stand for the ideal of 
compact operators and the Schatten classes. In particular, $\sd({\mathscr H})$ is the ideal of Hilbert-Schmidt operators acting on ${\mathscr H}$. For any 
operator $H \in \mathcal{B}({\mathscr H})$, we denote its numerical range by 
${\mathcal N}(H) := \big \lbrace \bra H \psi,\psi \ket ; \psi \in {\mathscr H} , \| \psi \| = 1 \big \rbrace$, its spectrum by $\sigma (H)$, its resolvent set by $\rho (H)$, 
the set of its eigenvalues by ${\mathcal E}_{\rm p} (H)$. We also write
\begin{equation*}
\re(H) = \frac{1}{2}(H+H^*)\quad, \quad \im(H) := \frac{1}{2i} (H-H^*).
\end{equation*}
We define its point spectrum as the closure of the set of its eigenvalues and write it $\sigma_{\rm pp} (H)=\overline{{\mathcal E}_{\rm p} (H)}$. Finally, if $A$ 
is a selfadjoint operator acting on ${\mathscr H}$, we write $\bra A \ket := \sqrt{A^2+1}$.

For two  subsets $\Delta_1$ and $\Delta_2$ of ${\mathbb R}$, we denote as a subset of ${\mathbb C}$, $\Delta_1 + i \Delta_2 := \big\lbrace z \in {\mathbb C} : \re(z) \in \Delta_1, \im(z) \in \Delta_2 \big\rbrace$. For $R > 0$ and ${\zeta_0} \in {\mathbb C}$, we set $D_R({\zeta_0}) := \big\lbrace z \in \mathbb C : |z - {\zeta_0}| < R \big\rbrace$ and $D_R^\ast({\zeta_0}) := D_R({\zeta_0}) \setminus \lbrace {\zeta_0} \rbrace$. In particular, ${\mathbb D}=D_1(0)$ denotes the open unit disk of the complex plane. For $\Omega \subseteq {\mathbb C}$ an open domain and ${\mathbb B}$ a Banach space, ${\rm Hol}(\Omega,\mathbb{B})$ denotes the set of holomorphic functions from $\Omega$ with values in ${\mathbb B}$. We will adopt the following principal determination of the complex square root:
$\sqrt{\cdot} : {\mathbb C} \setminus (-\infty,0] \longrightarrow \big\lbrace z \in {\mathbb C} : \im(z) \ge 0 \big \rbrace$ and we set ${\mathbb C}^+ := \big\lbrace z \in {\mathbb C} : \im(z) > 0 \big \rbrace$. By $0 < \vert k \vert <\!\!<1$, we mean that $k \in {\mathbb C}\setminus \{0\}$ is sufficiently close to $0$.

The discrete Fourier transform ${\mathcal F} : \ell^2({\mathbb Z}) \rightarrow {\rm L}^2({\mathbb T})$, where ${\mathbb T} := {\mathbb R} /2\pi {\mathbb Z}$, 
is defined for any $x \in \ell^2({\mathbb Z})$ and $f \in {\rm L}^2({\mathbb T})$ by
\begin{equation}\label{eq1,3}
({\mathcal F} x)(\alpha) := (2\pi)^{-\frac{1}{2}} \sum_{n \, \in \, {\mathbb Z}} e^{-in\alpha} x(n), \quad
( {\mathcal F}^{-1} f )(n) := (2\pi)^{-\frac{1}{2}} \int_{{\mathbb T}} e^{in\alpha} f(\alpha) d\alpha.
\end{equation}
The operator ${\mathcal F}$ is unitary. For any bounded (resp. selfadjoint) operator $L$ acting on $\ell^2({\mathbb Z})$, we define the bounded (resp. selfadjoint) 
operator $\widehat L$ acting on ${\rm L}^2({\mathbb T})$ by
\begin{equation}\label{eq:mr2}
\widehat L := {\mathcal F} L {\mathcal F}^{-1}.
\end{equation}


\section{Model and Main results}\label{results}

\subsection{The model}

\paragraph{The unperturbed Hamiltonian.}
We denote by $\lao$ the one-dimensional Schr\"odinger operator defined on $\ell^2({\mathbb Z})$ by
\begin{equation}\label{eq1,1}
(\lao x)(n) := 2x(n) - x(n + 1) - x(n - 1).
\end{equation}
$\lao$ is a bounded selfadjoint operator.
We deduce that $\widehat \lao = {\mathcal F}\lao {\mathcal F}^{-1}$ is the multiplication operator on ${\rm L}^2({\mathbb T})$ by the function $f$ where
\begin{equation}\label{eq6,2}
f(\alpha) := 2 - 2\cos \alpha = 4 \sin^2 \frac{\alpha}{2}, \quad \alpha \in [-\pi,\pi].
\end{equation}
It follows that $\sigma (\lao) = \sigma_{\textup{ac}} (\lao) = [0,4]$, where $\lbrace 0,4 \rbrace$ are the thresholds. For $z\in {\mathbb C}\setminus [0,4]$, we write $R_0(z) = (\lao -z)^{-1}$. One deduces that for any $x\in \ell^2({\mathbb Z})$
\begin{equation*}
({\mathcal F} R_0(z) x)(\alpha) = \frac{({\mathcal F} x)(\alpha)}{f(\alpha ) - z}, \qquad z \in {\mathbb C} \setminus [0,4].
\end{equation*}
For $z \in {\mathbb C} \setminus [0,4]$ small enough, we can introduce the change of variables 
\begin{equation*}
z = 4 \sin^2 \frac{\phi}{2}, \qquad \im (\phi) > 0.
\end{equation*}
In this case, the resolvent $R_0(z)$ is represented by the convolution with the function 
\begin{equation}\label{R_0}
R_0(z,n) = \frac{i{\rm e}^{i\phi \vert n \vert}}{2 \sin \phi} = \frac{i{\rm e}^{i \vert n \vert 2\arcsin \frac{\sqrt{z}}{2}}}{\sqrt{z} \sqrt{4 - z}}, \qquad \arcsin \frac{\sqrt{z}}{2} 
\underset{ z = 0}{\sim} \frac{\sqrt{z}}{2}.
\end{equation}

\paragraph{The perturbation.}
For any bounded operator $V$ acting on $\ell^2({\mathbb Z})$, we define the perturbed operator
\begin{equation}\label{eq1,5}
\la := \lao + V.
\end{equation}
If $\big\lbrace V(n,m) \big\rbrace_{(n,m) \in {\mathbb Z}^2}$ denotes the matrix representation of the operator $V$ in the canonical orthonormal basis of $\ell^2({\mathbb Z})$, 
for $x \in \ell^2({\mathbb Z})$ the sequence $Vx \in \ell^2({\mathbb Z})$ is given by
\begin{equation}\label{eq1,6}
(Vx)(n) = \sum_{m \in {\mathbb Z}} V(n,m)x(m) \quad \text{for any } n\in {\mathbb Z}. 
\end{equation}
If $V$ is represented by a diagonal matrix (i.e. $V(n,m)=V(n,m) \delta_{nm}$), we write $V(n) := V(n,n)$ and $V$ is just 
the multiplicative operator defined by $(Vx)(n) = V(n)x(n)$, $x \in \ell^2({\mathbb Z})$, $n\in {\mathbb Z}$.

For further use, let us conclude this paragraph with the following observation (see Subsection \ref{ss:res} for more details). Let 
$J : \ell^2({\mathbb Z}) \rightarrow \ell^2({\mathbb Z})$ be the unitary operator defined by $(J x)(n) := (-1)^n x(n)$, $x\in \ell^2({\mathbb Z})$, and define
\begin{equation}\label{eq1,10}
V_J := J V J^{-1}.
\end{equation}
Then, $J^2 = I$ and the matrix representation of the operator $V_{J}$ in the canonical orthonormal basis of $\ell^2({\mathbb Z})$ satisfies
$$
V_J (n,m) = (-1)^{n+m} V(n,m), \quad (n,m) \in {\mathbb Z}^2.
$$
In particular, if $V$ is a diagonal matrix, then $V_J = V$. Calculations yield (see e.g. \cite[Eq. (A.1)]{ito}): $J \la J^{-1} = -H_{-V_J} + 4$ so that
\begin{equation}\label{eq1,111}
J ( \la - z )^{-1}J^{-1}= -( H_{-V_J} - (4 - z) )^{-1}.
\end{equation}

\paragraph{The conjugate operator.}
Finally, let us define the auxiliary operator $A_0$, conjugate to $\lao$ in the sense of the Mourre theory and acting on $\ell^2({\mathbb Z})$ by
\begin{equation}\label{A0}
A_0 := {\mathcal F}^{-1} \widehat A_0 {\mathcal F},
\end{equation}
where the operator $\widehat A_0$ is (abusing notation) the unique selfadjoint extension of the symmetric operator $$\sin \alpha (-i\partial_\alpha) + (-i\partial_\alpha) \sin \alpha$$ defined on $C^{\infty}({\mathbb T})$.

\begin{rem} Define the position operator $X$ and shift operator $S$ on span $\{e_n : n\in {\mathbb Z}\}$ by $(Xx)(n)=n x(n)$ and $(Sx)(n)=x(n+1)$. $A_0$ is also the unique selfadjoint extension of the symmetric operator $\im (S) X+ X\im (S)$ defined on span $\{e_n : n\in {\mathbb Z}\}$.
\end{rem}

\subsection{Essential and discrete spectra}

Let $L$ be a closed operator acting on a Hilbert space ${\mathscr H}$. If $z$ is an isolated point of $\sigma(L)$, let 
$\gamma$ be a small positively oriented circle centered at $z$ and separating $z$ from the other components of $\sigma(L)$. The point $z$ is said 
to be a discrete eigenvalue of $L$ if its algebraic multiplicity
\begin{equation}\label{eq1,00}
\textup{m}(z) := \text{rank} \left( \frac{1}{2i\pi} \int_{\gamma} (L - \zeta)^{-1} d\zeta \right)
\end{equation}
is finite. Note that $\mathrm{m} \, (z) \geq \mathrm{dim} \, ( \text{Ker} (L - z) )$, the geometric multiplicity of $z$. Equality holds if $L$ is 
normal (see e.g. \cite{kato}). We define the discrete spectrum of $L$ by
\begin{equation}\label{eq1,01}
\sigma_{\text{disc}}(L) := \big\lbrace z \in \sigma(L) : z \hspace*{0.1cm} \textup{is a discrete eigenvalue of $L$} \big\rbrace.
\end{equation}
\noindent
We recall that a closed linear operator is a Fredholm operator if its range is closed, and both its kernel and cokernel are finite-dimensional. We define 
the essential spectrum of $L$ by
\begin{equation}\label{eq1,02}
\sigma_{\text{ess}}(L) := \big\lbrace z \in {\mathbb C} : 
\textup{$L - z$ \textup{is not a Fredholm operator}} \big\rbrace.
\end{equation} 
It is a closed subset of $\sigma(L)$.

\begin{rem}
If $L$ is selfadjoint, $\sigma(L)$ can be decomposed always as a disjoint union:
$\sigma(L) = \sigma_{{\rm ess}}(L) \bigsqcup \sigma_{{\rm disc}}(L)$.
If $L$ is not selfadjoint, this property is not necessarily true. Indeed, consider for instance the shift operator 
$S: \ell^2({\mathbb Z}_+) \rightarrow \ell^2({\mathbb Z}_+)$ defined by $(S x)(n) := x(n + 1)$. We have 
\begin{equation*}
\sigma(S) = \big\lbrace z \in {\mathbb C} : \vert z \vert \leq 1 \big\rbrace, \qquad \sigma_{{\rm ess}}(S) = \big\lbrace z \in {\mathbb C} : \vert z \vert = 1 \big\rbrace,
\qquad \sigma_{{\rm disc}}(S) = \emptyset.
\end{equation*}
\end{rem}

However, in the case of the operator $H_V$, we have the following result:

\begin{theo} 
Let $V$ belongs to $\sinf(\ell^2({\mathbb Z}))$. Then, $\sigma(\la) = \sigma_{{\rm ess}}(\la) \bigsqcup \sigma_{{\rm disc}}(\la)$, where
$\sigma_{{\rm ess}} (\la) = \sigma_{{\rm ess}} (\lao) = [0,4]$.
The possible limit points of $\sigma_{{\rm disc}}(\la)$ are contained in $\sigma_{{\rm ess}} (\la)$.
\end{theo}
\begin{proof}
It follows from Weyl's criterion on the invariance of the essential spectrum under compact perturbations and from \cite[Theorem 2.1, p. 373]{goh}. See also \cite[Corollary 2, p.113]{RS4}.
\end{proof}

The reader will note that if $V$ is compact and selfadjoint, then $\sigma_{{\rm disc}}(\la) \subset (-\infty, 0)\cup (4, \infty)$ and the set of limit points of 
$\sigma_{{\rm disc}}(\la)$ is necessarily contained in $\{0,4\}$. If $V$ is non-selfadjoint, then $\sigma_{{\rm disc}}(\la)$ may contain non-real numbers and the 
set of limit points may be considerably bigger, see e.g. \cite{bog} for the case of Laplace operators. However, we show in Theorem \ref{t:tda} that this cannot 
be the case if $V$ satisfies some additional regularity conditions. 

\begin{de}\label{dcl} 
Let ${\mathscr H}$ be a Hilbert space, $R > 0$ and $A$ be a selfadjoint operator defined on ${\mathscr H}$. An 
operator $B \in \mathcal{B}({\mathscr H})$ belongs to the class ${\mathcal A}_R(A)$ if the map $\theta \mapsto {\rm e}^{i \theta A} B {\rm e}^{-i \theta A}$, 
defined for $\theta \in \R$ extends holomorphically on $D_R(0)$ so that the corresponding extension belongs to  
${\rm Hol} ( D_R(0),\mathcal{B}({\mathscr H}) )$. In this case, we write $B \in {\mathcal A}_R(A)$ and ${\mathcal A}(A) := \cup_{R>0} {\mathcal A}_R(A)$ is the collection
of bounded operators for which a complex scaling w.r.t. $A$ can be performed. 
\end{de}

\begin{rem} 
\begin{itemize}
\item[(a)] If $B \in \sinf({\mathscr H}) \cap {\mathcal A}_R(A)$ for some selfadjoint operator $A$ and some $R>0$, then the holomorphic extension of the map $\theta \mapsto {\rm e}^{i \theta A} B {\rm e}^{-i \theta A}$ belongs actually to ${\rm Hol} ( D_R(0),\sinf({\mathscr H}) )$, see e.g. \cite[Lemma 5, Section XIII.5]{RS4}.
\item[(b)] The main properties of the classes ${\mathcal A}_R(A)$ are recalled in Section \ref{regclass}. In particular, in our case, we have $V\in {\mathcal A}_R(A_0)$ if 
and only if $\widehat V\in {\mathcal A}_R (\widehat A_0)$.
\end{itemize}
\end{rem}

We have:
\begin{theo}\label{t:tda} 
If $V \in \sinf ( \ell^2({\mathbb Z}) ) \cap {\mathcal A}(A_0)$, then the possible limit points of $\sigma_{\textup{disc}} (\la)$ belong to $\lbrace 0,4 \rbrace$.
\end{theo}

Our next result, Theorem \ref{t:tda+}, provides additional informations on the essential spectrum $\sigma_{\textup{ess}} (\la) = [0,4]$. We recall that:
\begin{de} Let $A$ be a self-adjoint operator defined on the Hilbert space ${\mathscr H}$. A vector $\varphi \in {\mathscr H}$ is analytic w.r.t. $A$ if $\varphi \in \cap_{k\in {\mathbb N}} {\mathcal D} (A^k)$ and for some $R>0$, the power series 
$$
\sum_{k = 0}^\infty \frac{\vert \theta \vert^k}{k!} \big\Vert A^k \varphi \big\Vert 
$$
converges for any $\theta \in D_R(0)$.
\end{de}

Later on, we will distinguish among the vectors which are analytic w.r.t. $A_0$, the vectors of the linear subspace span $\{e_n; n\in {\mathbb Z}\}$ (see e.g. Lemma \ref{le2}) and more generally, the vectors belonging to $\ell_{\delta}^{2}({\mathbb Z})$ for some $\delta >0$ (see Corollary \ref{analyticv}).

\begin{theo}\label{t:tda+} 
Let $V \in \sinf ( \ell^2({\mathbb Z}) ) \cap {\mathcal A}(A_0)$. Then, there exists a discrete subset ${\mathcal D} \subset (0,4)$ 
whose only possible limits points belong to $\{0,4\}$ and for which the following holds: given any relatively compact interval $\Delta_0$, 
$\overline{\Delta_0} \subset (0,4)\setminus {\mathcal D}$, there exists $\delta_0 >0$ such that for any analytic vectors $\varphi$ and $\psi$ w.r.t. $A_0$,
\begin{gather*}
\sup_{z\in \Delta_0 + i(-\delta_0,0)} |\bra \varphi, (z-H_V)^{-1} \psi \ket | < \infty , \\
\sup_{z\in \Delta_0 + i(0,\delta_0)} |\bra \varphi, (z-H_V)^{-1} \psi \ket | < \infty.
\end{gather*}
\end{theo}

\begin{rem}\begin{itemize}
\item[(a)] For any subset $\Delta$ such that $\overline{\Delta} \subset (0,4)$, ${\mathcal D}\cap \Delta$ is finite.
\item[(b)] If $V$ is selfadjoint, then $\la = \la^\ast$ and ${\mathcal D}$ coincides with the set of eigenvalues of $\la$ embedded in $(0,4)$ i.e. ${\mathcal D} = {\mathcal E}_p(H_V)\cap (0,4)$. In the non-selfadjoint case $\la \neq \la^\ast$, we expect that ${\mathcal E}_p(H_V)\cap (0,4) \subset {\mathcal D}$.
\end{itemize}
\end{rem}

The proofs of Theorems \ref{t:tda} and \ref{t:tda+} are postponed to Section \ref{dilatation}. Among perturbations $V$ which satisfy the hypotheses of Theorem \ref{t:tda} and \ref{t:tda+}, we may consider:
\begin{itemize}
\item those which satisfy Assumption \ref{ade} below,
\item $V=|\psi \ket \bra \varphi |$, where $\varphi$ and $\psi$ are analytic vectors for $A_0$. We refer to Subsection \ref{ssec} for more details and further examples. 
\end{itemize}

For $V \in \sinf ( \ell^2({\mathbb Z}) ) \cap {\mathcal A}(A_0)$, Theorem \ref{t:tda} states that the points of $(0,4)$ cannot be the limit points of any 
sequence of discrete eigenvalues. In the next section, we consider more specifically the case of the thresholds $\{0,4\}$.

\subsection{Resonances}\label{ss:res}

In this section, we show how to control the distribution of resonances, in particular the eigenvalues of $\la$ around the thresholds $\{0,4\}$. We perform this analysis by means of the characteristic values method. First, we formulate a new assumption and recall some basic facts on resonances.

\begin{ass}\label{ade}
There exists $\delta > 0$ such that
\begin{equation*}\label{eq1,7}
\sup_{(n,m) \in {\mathbb Z}^2} {\rm e}^{\delta (\vert n \vert + \vert m \vert)} \big\vert V(n,m) \big\vert < \infty ,
\end{equation*}
where $( V(n,m) )_{(n,m) \, \in \, {\mathbb Z}^2}$ is the matrix representation of the operator $V$ in the canonical orthonormal basis of $\ell^{2}({\mathbb Z})$.
\end{ass}

\begin{rem}\label{rp} If $V$ satisfies Assumption \ref{ade}, then
\begin{itemize}
\item[(a)] $V\in \su ( \ell^2({\mathbb Z}) ) \subset \sinf ( \ell^2({\mathbb Z}) )$,
\item[(b)] $V\in {\mathcal A}(A_0)$ (see Proposition \ref{pe3}).
\end{itemize}
\end{rem}

We also note that Assumption \ref{ade} holds for $V$ if and only if it holds for $V_{J}$ defined by \eqref{eq1,10}.

As a preliminary, we have the following result, whose proof is contained in Section \ref{resonances}.
\begin{prop}\label{p3,1} 
Let Assumption \ref{ade} holds. Set $z(k)=k^2$. Then, there exists $0 < \varepsilon_0 \leq \frac{\delta}{8}$, small enough, such that for ${\bf V} \in \{ V, -V_J\}$, the operator-valued function with values in ${\mathcal B}(\ell_{\delta}^{2}({\mathbb Z}) ; \ell_{-\delta}^{2}({\mathbb Z}))$, 
\begin{equation*}
k \longmapsto ( H_{\textup{\bf V}} - z(k) )^{-1} ,
\end{equation*} 
admits a meromorphic extension from $D_{\varepsilon_0}^\ast(0) \cap {\mathbb C}^+$ to $D_{\varepsilon_0}^\ast(0)$. This extension will be denoted $R_{\textup{\bf V}} ( z )$.
\end{prop}

Now, we define the resonances of the operator $\la$ near $z = 0$ and $z = 4$. We follow essentially the terminology and characterization of the resonances used in \cite[Sect. 2 and 6]{bo} and refer to Section \ref{resonances} for more details. In particular, the quantity $Ind_{\gamma}(\cdot)$, which denotes the index w.r.t. a contour $\gamma$, is defined by \eqref{eqa,2} and the weighted resolvents $\mathcal{T}_{\bf V}(\cdot)$ are defined by \eqref{eq3,aa}. 

\begin{de}\label{d3,1} 
The resonances of the operator $\la$ near $0$ are the points $z$, which are the poles of the meromorphic extension of the resolvent $R_V(z)$, as introduced in Proposition \ref{p3,1}. 
\end{de}

For $k \in D_{\varepsilon_0}^\ast(0)$, set $z_0(k) := k^2$. Given $k_1 \in D_{\varepsilon_0}^\ast(0)$, let $z_1= z_0 (k_1)$. By Proposition \ref{p3,2}, $z_1$ is a resonance of $H_V$ near $0$ if and only if $k_1$ is a characteristic value of $I + \mathcal{T}_V (z_0 (\cdot ))$. We define the multiplicity of such a resonance as follows:
\begin{de}\label{d3,1+} 
The multiplicity of a resonance $z_1 = z_0(k_1)$ is defined as the multiplicity of the characteristic value $k_1$, namely:
\begin{equation}\label{eq3,14}
\textup{mult}(z_{1}) := Ind_{\gamma} \, ( I + \mathcal{T}_V ( z_0(\cdot) ) ),
\end{equation}
where $\gamma$ is a positively oriented circle centered at $k_1$, chosen sufficiently small so that $k_1$ is the only characteristic value enclosed in $\gamma$.
\end{de}

There exists a simple way to define the resonances of $\la$ near the threshold $z = 4$. Indeed, relation \eqref{eq1,111} implies that
\begin{equation}\label{eq1,11}
J W_{-\delta} (\la - z)^{-1}W_{-\delta} J^{-1}= -W_{-\delta} ( H_{-V_J} - (4 - z) )^{-1} W_{-\delta}.
\end{equation}
Therefore, the analysis of the resonances of $\la$ near the second threshold $4$ is reduced to that of the first one $0$ (up to a sign and 
a change of variable). Combining \eqref{eq1,11} and Definition \ref{d3,1}, we have:
\begin{de}\label{d3,2} 
The resonances of the operator $\la$ near $4$ are the points $z =  4 - u$ where $u$ are the poles of the meromorphic extension of the resolvent $R_{-V _J}(u)$, as introduced in Proposition \ref{p3,1}. 
\end{de}

For $k \in D_{\varepsilon_0}^\ast(0)$, set $z_4(k) := 4 - k^2$. Given $k_1 \in D_{\varepsilon_0}^\ast(0)$, let $z_1= z_4 (k_1)$. By Proposition \ref{p3,3}, $z_1$ is a resonance of $H_V$ near $4$ if and only if $k_1$ is a characteristic value of  $I + \mathcal{T}_{-V_J} ( 4 - z_4(\cdot) )$. As in Definition \ref{d3,1+}, one has:
\begin{de}\label{d3,2+} 
The multiplicity of a resonance $z_{1} = z_4(k_1)$ is defined as the multiplicity of the characteristic value $k_1$, namely:
\begin{equation}\label{eq3,15}
\textup{mult}(z_{1}) := Ind_{\gamma} \, ( I + \mathcal{T}_{-V_J} ( 4 - z_4(\cdot) ) ),
\end{equation}
where $\gamma$ is a positively oriented circle centered at $k_1$, chosen sufficiently small so that $k_1$ is the only characteristic value enclosed in $\gamma$.
\end{de}

We denote by ${\rm Res}_\mu \, (\la)$ the resonances set of $\la$ near the threshold $\mu \in \lbrace 0,4 \rbrace$. The discrete eigenvalues of the operator $\la$ near $0$ (resp. near $4$) are resonances. Moreover, the algebraic multiplicity \eqref{eq1,00} of a discrete eigenvalue coincides with its multiplicity as a resonance near $0$ (resp. near $4$), defined by \eqref{eq3,14} (resp. \eqref{eq3,15}). Let us justify it briefly in the case of the discrete eigenvalues near $0$ (the case concerning those near $4$ can be treated similarly). 
Let $z_{1} = z_0(k_1) \in {\mathbb C} \setminus [0,4]$ be a discrete eigenvalue of $\la$ near $0$. According to \cite[Chap. 9]{si} and since $V = W_{-\delta} \mathscr{V} W_{-\delta}$ is of trace class, $\mathscr{V}$ being defined by \eqref{eq3,100}, this is equivalent to the property $f(z_{1}) = 0$, where for $z \in {\mathbb C} \setminus [0,4]$, $f$ is the holomorphic function defined by
$$
f(z) := \det ( I + V(\lao - z)^{-1} ) = \det ( I + \mathscr{V} W_{-\delta} (\lao - z)^{-1} W_{-\delta} ).
$$
Moreover, the algebraic multiplicity \eqref{eq1,00} of $z_{1}$ is equal to its order as zero of the function $f$. Residue Theorem yields:
$$
\textup{m}(z_{1}) = ind_{\gamma'} f := \frac{1}{2i\pi} \int_{\gamma'} \frac{f'(z)}{f(z)} dz,
$$
$\gamma'$ being a small circle positively oriented containing $z_{1}$ as the only zero of $f$. The claim follows immediately from the equality 
$$
ind_{\gamma'} f = Ind_{\gamma} \, ( I + \mathcal{T}_{V} ( z_0(\cdot) ) ),
$$
see for instance \cite[Eq. (6)]{bo} for more details. 

\begin{rem} \begin{itemize}
\item[(a)] The resonances $z_\mu(k)$ are defined in some two-sheets Riemann surfaces $\mathcal{M}_\mu$, 
$\mu \in \lbrace 0,4 \rbrace$, respectively.
\item[(b)] The discrete spectrum and the embedded eigenvalues of $\la$ near $\mu$ belong to the set of resonances
\begin{equation*}
z_\mu(k) \in \mathcal{M}_\mu, \quad \im(k) \geq 0.
\end{equation*}
\end{itemize}
\end{rem}

The above considerations lead us to the following result:
\begin{theo}\label{t1} Let Assumption \ref{ade} hold. Let $\varepsilon_0$ be as defined in Proposition \ref{p3,1}. Then, we can choose $\varepsilon_0' \in (0, \varepsilon_0]$ in such a way that $\la$ has no resonance $z_\mu(k)$, for $\mu \in \lbrace 0,4 \rbrace$ and $k \in D_{\varepsilon_0'}^\ast (0)$.
\end{theo}

According to Theorem \ref{t1}, $\la$ has no resonance in a punctured neighborhood of $\mu$, in the two-sheets Riemann surface 
$\mathcal{M}_\mu$ (where the resonances are defined). Figure \ref{fig 1} illustrates this fact.

\begin{figure}[h]\label{fig 1}
\begin{center}
\tikzstyle{+grisEncadre}=[fill=gray!60]
\tikzstyle{blancEncadre}=[fill=white!100]
\tikzstyle{grisEncadre}=[fill=gray!20]
\begin{tikzpicture}[scale=1.2]

\draw [grisEncadre] (0,0) -- (180:1.58) arc (180:360:1.58) -- cycle;
\draw [blancEncadre] (0,0) -- (0:1.58) arc (0:180:1.58) -- cycle;
\draw [->] [thick] (-2.5,0) -- (2.3,0);
\draw (2.3,0) node[right] {\small{$\re(k)$}};
\draw [->] [thick] (0,-2.5) -- (0,2.7);
\draw (0,2.7) node[right] {\small{$\im(k)$}};
\draw (0,0) -- (1.31,0.9);
\draw (0.6,0.5) node[above] {\small{$\varepsilon_{0}'$}};
\node at (-1.4,-1.2) {\tiny{$\times$}};
\node at (0.9,-1.4) {\tiny{$\times$}};
\node at (1.8,-0.2) {\tiny{$\times$}};
\node at (1.4,-1.2) {\tiny{$\times$}};
\node at (0.9,1.4) {\tiny{$\times$}};
\node at (1.4,1.2) {\tiny{$\times$}};
\node at (-1.7,0.4) {\tiny{$\times$}};
\node at (-1.4,1.2) {\tiny{$\times$}};
\node at (-0.5,1.7) {\tiny{$\times$}};
\node at (-0.2,1.7) {\tiny{$\times$}};
\node at (2,2.3) {\small{$\textup{\textbf{Absence of resonances}}$}};
\draw [->] [dotted] (1.8,2.15) -- (-0.5,1);
\draw [->] [dotted] (1.8,2.15) -- (1,-0.5);
\node at (-4.5,0.5) {\small{$\textup{The \textbf{physical plane}}$}};
\draw [->] [dotted] (-2.1,0.5) -- (-1,0.5);
\node at (-4.5,-0.5) {\small{$\textup{The \textbf{non physical plane}}$}};
\node at (-4.5,-0.8) {\small{$\textup{(the second sheet of the Riemann surface $\mathcal{M}_{\mu}$)}$}};
\draw [->] [dotted] (-1.8,-0.5) -- (-1,-0.5);
\end{tikzpicture}
\caption{\textup{Resonances near the threshold $\mu \in \lbrace 0,4 \rbrace$ in variable $k$.}}\label{fig 1}
\end{center}
\end{figure}

Since the discrete eigenvalues of $\la$ near $\mu$ are part of the set of resonances $z_\mu(k) \in \mathcal{M}_\mu$, $\im(k) \geq 0$, we conclude from Theorems \ref{t:tda}, \ref{t1} and Proposition \ref{pe3} that:
\begin{theo}\label{t2}
Let Assumption \ref{ade} holds. Then, $\sigma_{\textup{disc}} (\la)$ has no limit points in $[0,4]$, hence is finite. There exists also a finite subset ${\mathcal D}' \subset (0,4)$ for which the following holds: given any relatively compact interval $\Delta_0$, $\overline{\Delta_0} \subset (0,4)\setminus {\mathcal D}'$, there exists $\delta_0 >0$ such that for any vectors $\varphi$ and $\psi$ in $\ell_{\delta}^{2}({\mathbb Z}),$
\begin{gather*}
\sup_{z\in \Delta_0 + i(-\delta_0,0)} |\bra \varphi, (z-H_V)^{-1} \psi \ket | < \infty, \\
\sup_{z\in \Delta_0 + i(0,\delta_0)} |\bra \varphi, (z-H_V)^{-1} \psi \ket | < \infty .
\end{gather*}
\end{theo}

As mentioned in the introduction, the first conclusion of Theorem \ref{t2} can be produced under weaker decay conditions along the diagonal provided that the perturbed operator $H_V$ is still a Jacobi operator, see e.g.  \cite[Theorem 1]{ego1}. By contrast, Theorem \ref{t2} allows to handle perturbations displaying a full off-diagonal structure and exhibits some LAP.

If $V$ is selfadjoint and satisfies Assumption \ref{ade}, it follows from Proposition \ref{ade}, Theorem \ref{t2} and the usual complex scaling arguments (see e.g. \cite{sig}) that:
\begin{cor}\label{c1}
Let Assumption \ref{ade} hold. If the perturbation $V$ is selfadjoint, then:
\begin{itemize}
\item $\sigma_{\rm {ess}} (H_V)= [0,4]$ and $\sigma_{{\rm disc}} (H_V)$ is finite.
\item There is at most a finite number of eigenvalues embedded in $[0,4]$. Each eigenvalue embedded in $(0,4)$ has finite multiplicity.
\item The singular continuous spectrum $\sigma_{{\rm sc}} (H_V) =\emptyset$ and the following LAP holds: given any relatively compact interval $\Delta_0$, $\overline{\Delta_0} \subset (0,4)\setminus {\mathcal E}_{\rm p}(H_V)$, there exists $\delta_0 >0$ such that for any vectors $\varphi$ and $\psi$ in $\ell_{\delta}^{2}({\mathbb Z}),$
\begin{gather*}
\sup_{z\in \Delta_0 + i(0,\delta_0)} |\bra \varphi, (z-H_V)^{-1} \psi \ket | < \infty, \\
\sup_{z\in \Delta_0 + i(-\delta_0,0)} |\bra \varphi, (z-H_V)^{-1} \psi \ket | < \infty \, .
\end{gather*}
\end{itemize}
\end{cor}

\subsection{Embedded eigenvalues and Limiting Absorption Principles}\label{sslap}

For less regular perturbation $V$, we can still take advantage of the existence of some positive commutation relations to control some spectral properties of $H_V$. Let us define the regularity conditions involved in the statement of Theorem \ref{c11sign} below.

\begin{de}\label{ck}
Let ${\mathscr H}$ be a Hilbert space and $A$ be a selfadjoint operator defined on ${\mathscr H}$. Let $k\in {\mathbb N}$. An operator 
$B \in \mathcal{B}({\mathscr H})$ belongs to the class $C^k (A)$, if the map ${\mathcal W}_A: \theta \mapsto {\rm e}^{i \theta A} B {\rm e}^{-i \theta A}$ 
is $k$-times strongly continuously differentiable on $\R$. We also denote $C^{\infty} (A) = \cap_{k\in {\mathbb N}}C^k (A)$.
\end{de}

\begin{rem}
A bounded operator $B$ belongs to $C^1(A)$ if and only if the sesquilinear form defined on ${\cal D}(A)\times {\cal D}(A)$ by 
$(\varphi,\psi) \mapsto \langle A \varphi,B \psi \rangle - \langle \varphi,BA \psi \rangle$, extends continuously to a bounded form on 
${\mathscr H} \times {\mathscr H}$. The (unique) bounded linear operator associated to the extension is denoted by $\mathrm{ad}_A B = \mathrm{ad}_A (B) =[A,B]$ and 
we have $(\partial_{\theta} {\mathcal W}_A)(0)= i\mathrm{ad}_A (B)$.
\end{rem}

Following \cite{abmg}, we also consider fractional order regularities:

\begin{de}\label{c11}
Let ${\mathscr H}$ be a Hilbert space and $A$ be a selfadjoint operator defined on ${\mathscr H}$. An operator $B \in \mathcal{B}({\mathscr H})$ 
belongs to ${\cal C}^{1,1}(A)$ if
\begin{equation*}
\int_0^1 \big\| e^{iA\theta}B {\rm e}^{-iA\theta}+e^{-iA\theta}B {\rm e}^{iA\theta} -2B \big\| \, \frac{d\theta}{\theta^2} < \infty \, .
\end{equation*}
\end{de}
Note that $B\in {\cal C}^{1,1}(A)$ if and only if $B$ can be suitably approximated by a family of operators in $C^2(A)$ ; we refer to Section \ref{linktoc11} for more details. Actually, ${\cal C}^{1,1}(A)$ is a linear subspace of ${\cal B}({\mathscr H})$, stable under adjunction $*$. It is also known that $C^2(A)\subset {\mathcal C}^{1,1}(A)\subset C^1(A)$, see e.g. inclusions (5.2.19) in \cite{abmg}.

\begin{rem} Examples of operators belonging to the class ${\cal C}^{1,1}(A_0)$ are given in Section \ref{ecsmoo} below.
\end{rem}

We recall that the operator $A_0$ is defined by (\ref{A0}). If we assume that $\im(V)$ has a sign, we obtain Theorem \ref{c11sign} below. Mind that a statement involving the symbol $\pm$ has to be understood as two independent statements.

\begin{theo}\label{c11sign} 
Let $V\in \sinf ( \ell^2({\mathbb Z}) )$ and assume $\pm \im(V) \geq 0$. Fix any open interval $\Delta$ such that $\overline{\Delta} \subset (0,4)$.
\begin{enumerate}
\item Assume $V\in C^1(A_0)$ and $\mathrm{ad}_{A_0} ( \re(V) )$ also belongs to $\sinf ( \ell^2({\mathbb Z}) )$. Then, $${\mathcal E}_{\rm p} (H_V) \cap \overline{\Delta} \subset \sigma_{\rm pp} ( \re (H_V) ) \cap \overline{\Delta}\, .$$ The set ${\mathcal E}_{\rm p} (H_V) \cap \overline{\Delta}$ is finite and its eigenvalues have finite geometric multiplicity.
\item Assume that $V \in {\mathcal C}^{1,1}(A_0)$. Given any open interval $\Delta_0$ such that 
$\overline{\Delta_0} \subset \Delta \setminus \sigma_{\rm pp} ( \re(H) )$ and any $s>1/2$, the following LAP holds:
\begin{equation*}
\sup_{ \mp \im(z) > 0, \re(z) \in \Delta_0} \|\bra A_0 \ket^{-s}(z-H_V)^{-1}\bra A_0 \ket^{-s} \| < \infty.
\end{equation*}
\end{enumerate}
\end{theo}

\begin{rem} 
\begin{itemize}
\item[(a)] If $V\in \sinf({\mathscr H})$, then $V^*$, $\re(V)$ and $\im(V)$ also belong to $\sinf({\mathscr H})$.
\item[(b)] If $V$ belongs to $C^1(A)$ (resp. ${\mathcal C}^{1,1}(A)$), then, $\re(V)$ and $\im(V)$ also belong to $C^1(A)$ 
(resp. ${\mathcal C}^{1,1}(A)$). In particular, if $V$ belongs to ${\mathcal C}^{1,1}(A)$ and $\re(V) \in \sinf({\mathscr H})$, then 
$\mathrm{ad}_A ( \re(V) ) \in \sinf({\mathscr H})$, see Remark (ii) in the proof of Theorem 7.2.9 in \cite{abmg}.
\end{itemize}
\end{rem}

Finally, we also have:
\begin{theo}\label{c11signweak} 
Let $V\in \sinf ( \ell^2({\mathbb Z}) )$ and assume that $V \in {\mathcal C}^{1,1}(A_0)$.
\begin{enumerate}
\item If $\im (V) >0$ and $i\mathrm{ad}_{A_0} (\re (V)) +\beta_- \im (V) \geq 0$ for some $\beta_- \geq 0$, then for any open interval $\Delta$ such that 
$\overline{\Delta} \subset (0,4)$ and any $s>1/2$, the following LAP holds:
\begin{equation*}
\sup_{- \im(z) > 0, \re(z) \in \Delta} \|\bra A_0 \ket^{-s}(z-H_V)^{-1}\bra A_0 \ket^{-s} \| < \infty.
\end{equation*}
\item If $\im (V) <0$ and $i\mathrm{ad}_{A_0} (\re (V)) -\beta_+ \im (V) \geq 0$ for some $\beta_+ \geq 0$, then for any open interval $\Delta$ such that 
$\overline{\Delta} \subset (0,4)$ and any $s>1/2$, the following LAP holds:
\begin{equation*}
\sup_{+ \im(z) > 0, \re(z) \in \Delta} \|\bra A_0 \ket^{-s}(z-H_V)^{-1}\bra A_0 \ket^{-s} \| < \infty.
\end{equation*}
\end{enumerate}
\end{theo}

The proofs of Theorems \ref{c11sign} and \ref{c11signweak} are direct applications of the abstract Mourre theory developed in Section \ref{mourre}. See Section \ref{proofthmc11sign} for the details.


\section{Complex scaling}\label{dilatation}

In this section, we use a complex scaling approach to study $\sigma (\la)$ for compact perturbations $V\in {\mathcal A}(A_0)$. Since the spectral properties of $\la$ and $\widehat \la$ coincide, we reduce our analysis to those of $\widehat \la$. 

\subsection{Before perturbation}

Following the general principles exposed in \cite{sig}, we first describe the scaling process for the unperturbed operator $\widehat \lao$, which is the multiplication operator by the function $f$ (see \eqref{eq6,2}). In what follows, we have summarized the main results. When no confusion can arise, the operator $\widehat \lao$ is identified with the function $f$.

We consider the unitary group $( \e^{i\theta\widehat A_0} )_{\theta\in\R}$, so that for $\psi\in\mathrm{L}^2(\mathbb{T})$, one has
\begin{equation}\label{eq6,4}
 ( \e^{i\theta\widehat A_0}\psi ) (\alpha) = \psi ( \varphi_\theta(\alpha) ) \sqrt{J(\varphi_\theta)(\alpha)},
\end{equation}
where
\begin{itemize}
\item $(\varphi_\theta)_{\theta\in\R}$ is the flow solution of the equation
$$
\begin{cases} 
\partial_{\theta} \varphi_\theta (\alpha) = 2\sin ( \varphi_\theta(\alpha) ), \\ 
\varphi_0(\alpha) = \mathrm{id}_\mathbb{T}(\alpha) = \alpha \: \: \textup{for each} \: \: \alpha \in \mathbb{T},
\end{cases}
$$
\item $J(\varphi_\theta)(\alpha)$ denotes the Jacobian of the transformation $\alpha \mapsto \varphi_\theta (\alpha)$.
\end{itemize}
Existence and uniqueness of the solution follow from standard ODE results. Explicitly,
\begin{equation*}
\varphi_\theta(\alpha) = \pm \arccos \left( \frac{-{\rm th}(2\theta) + \cos \alpha}{1 - {\rm th}(2\theta) \cos \alpha} \right) \: \: {\rm for} \: \: \pm \alpha \in \mathbb{T}.
\end{equation*}
Using \eqref{eq6,4} and the fact that $\varphi_{\theta_1} \circ \varphi_{\theta_2} = \varphi_{\theta_1 + \theta_2}$ for all $(\theta_1, \theta_2) \in \R^2$, one has for all 
$\theta \in \R$
\begin{equation*}
 \left( \e^{i\theta\widehat A_0} \widehat \lao \e^{-i\theta\widehat A_0}\psi\right)(\alpha)
 = f ( \varphi_\theta(\alpha) ) \psi(\alpha).
\end{equation*}
Let $T:\C \rightarrow \C$, $T(z) := 2(1 - z)$. Note that the map $T$ is bijective with
$$
T^{-1}(z) = 1- \frac{z}{2},
$$
and maps $[-1,1]$ onto $[0,4]$. The points $T(-1)=4$ and $T(1)=0$ are the thresholds of $\lao$ and $\widehat{\lao}$. Note also that $f= T\circ \cos$. Consider for $\theta \in \R$, the function $G_\theta$ defined on $[0,4]$ by $G_\theta := T \circ F_{\theta} \circ T^{-1}$ with
\begin{equation}\label{eq6,7}
F_\theta(\lambda) := \frac{\lambda - {\rm th}(2\theta)}{1 - \lambda {\rm th}(2\theta)}, \qquad \lambda \in [-1,1].
\end{equation}
Then, for all $\theta \in \R,$
\begin{equation}\label{eq6,6}
 \e^{i\theta\widehat A_0} \widehat \lao \e^{-i\theta\widehat A_0}\psi = G_{\theta} (\widehat \lao ) \psi .
\end{equation}
In other words, for any $\theta \in \R$, $\e^{i\theta\widehat A_0} \widehat \lao \e^{-i\theta\widehat A_0}$ is the multiplication operator by the function $(G_{\theta} \circ f)$, where $G_{\theta} \circ f = (T \circ F_{\theta} \circ \cos) = (G_{\theta} \circ T \circ \cos)$.

\begin{rem} In \eqref{eq6,7}, the denominator does not vanish since $\big\vert {\rm th}(2\theta) \lambda \big\vert < 1$ for $\lambda \in [-1,1]$.
\end{rem}

We summarise the following properties, whose verification is left to the reader.
\begin{prop}\label{p6,1} Let $\mathbb{D} := D_1(0)$ denote the open unit disk of the complex plane $\C$. Then,
\begin{enumerate}
\item[(a)] For any $\lambda \in [-1,1]$, the map $\theta \mapsto F_\theta(\lambda)$ is holomorphic in $D_\frac{\pi}{4}(0)$.
\item[(b)] For $\theta \in {\mathbb C}$ such that $\vert \theta \vert < \frac{\pi}{8}$, the map $\lambda \mapsto F_\theta(\lambda)$ is a homographic transformation with 
$F_\theta^{-1} = F_{-\theta}$. In particular, for $\theta \in \R$, $F_\theta(\mathbb{D}) = \mathbb{D}$ and $F_\theta ( [-1,1] ) = [-1,1]$.
\item[(c)] For $\theta \in {\mathbb C}$ such that $0 < \vert \theta \vert < \frac{\pi}{8}$, the unique fixed points of $F_\theta$ are $\pm 1$.
\item[(d)] For $\theta_1$, $\theta_2 \in {\mathbb C}$ with $\vert \theta_1 \vert$, $\vert \theta_2 \vert < \frac{\pi}{8}$, we have that: $F_{\theta_1} \circ F_{\theta_2} 
= F_{\theta_1 + \theta_2}$.
\end{enumerate}
\end{prop}

From Proposition \ref{p6,1}, statements (a) and (b), we deduce:
\begin{prop}\label{p6,2}
The bounded operator valued-function
$$
\theta \mapsto \e^{i\theta\widehat A_0} \widehat \lao \e^{-i\theta\widehat A_0} \in \mathcal{B} ( {\rm L}^2({\mathbb T}) ),
$$
admits a holomorphic extension from $( -\frac{\pi}{8},\frac{\pi}{8} )$ to $D_\frac{\pi}{8}(0)$, with extension operator given for $\theta \in D_\frac{\pi}{8}(0)$ by $G_{\theta} (\widehat \lao)$, which is the multiplication operator by the function $G_{\theta} \circ f = (T \circ F_{\theta} \circ \cos) = (G_{\theta} \circ T \circ \cos)$. In the sequel, this extension will be denoted $\widehat \lao(\theta)$.
\end{prop}

Combining the continuous functional calculus, Proposition \ref{p6,2} and unitary equivalence properties, we get for $\theta \in D_\frac{\pi}{8}(0),$
\begin{equation*}
\sigma ( \widehat \lao(\theta) ) = \sigma ( G_\theta (\widehat \lao) ) = G_\theta ( \sigma (\widehat \lao) ) 
= G_\theta ( \sigma (\lao) ) = G_\theta ( [0,4] ).
\end{equation*}
Thus, for $\theta \in D_\frac{\pi}{8}(0)$, $\sigma ( \widehat \lao(\theta) )$ is a smooth parametrized curve given by
\begin{equation}\label{eq6,11}
\sigma ( \widehat \lao(\theta) ) = \Big\lbrace T\circ F_{\theta}(\lambda) = G_{\theta}\circ T (\lambda) : \lambda \in [-1,1] \Big\rbrace.
\end{equation}
More precisely, we have:
\begin{prop}\label{p6,3} Consider the family of bounded operators $(\widehat \lao(\theta))_{\theta \in D_\frac{\pi}{8}(0)}$ defined in Proposition \ref{p6,2}. Then, it holds:
\begin{itemize}
\item[(a)] For $\theta_1$, $\theta_2 \in D_\frac{\pi}{8}(0)$ such that $\im(\theta_1) = \im(\theta_2)$, we have 
\begin{equation}\label{eq6,12}
\sigma ( \widehat \lao(\theta_1) ) = \sigma ( \widehat \lao(\theta_2) ).
\end{equation}
The curve $\sigma ( \widehat \lao(\theta) )$ does not depend on the choice of $\re(\theta)$.
\item[(b)] For $\pm \im(\theta) > 0$, $\theta \in D_\frac{\pi}{8}(0)$, the curve $\sigma ( \widehat \lao(\theta) )$ lies in ${\mathbb C}_{\pm}$.
\item[(c)] Let $\theta \in D_\frac{\pi}{8}(0)$. If $\im(\theta) \neq 0$, the curve $\sigma ( \widehat \lao(\theta) )$ is an arc of a circle, which contains the points $0$ and $4$. If $\im(\theta) = 0$, $\sigma ( \widehat \lao(\theta) ) = [0,4]$.
\end{itemize}
\end{prop}
We refer to Figure \ref{figd} below for a graphic illustration.

\medskip

\noindent
\begin{proof} Statement (a) can be derived from Proposition \ref{oconnor-0} (with $\widehat \lao(\cdot)$ in the role of the map $B ( \cdot )$). We give a direct proof here. Let $\theta_1$, $\theta_2 \in D_\frac{\pi}{8}(0)$ with $\im(\theta_1) = \im(\theta_2)$. Thanks to \eqref{eq6,11}, it is enough to show that
\begin{equation}\label{eq6,13}
\Big\lbrace F_{\theta_1}(\lambda) : \lambda \in [-1,1] \Big\rbrace = 
\Big\lbrace F_{\theta_2}(\lambda) : \lambda \in [-1,1] \Big\rbrace,
\end{equation}
to prove \eqref{eq6,12}. So, let $z = F_{\theta_1}(\lambda)$, for some $\lambda \in [-1,1]$. Let us show that there exists $\lambda' \in [-1,1]$ such that $z = F_{\theta_2}(\lambda')$. 
By Proposition \ref{p6,1}, statements (b) and (d), we can write
\begin{align*}
z = F_{\theta_2} ( F_{\theta_2}^{-1} ( F_{\theta_1}(\lambda) ) ) 
= F_{\theta_2} ( F_{-\theta_2} ( F_{\theta_1}(\lambda) ) ) 
= F_{\theta_2} ( F_{\theta_1 - \theta_2}(\lambda) ) .
\end{align*}
If $\im(\theta_1) = \im(\theta_2)$, then $\theta_1 - \theta_2 = \re(\theta_1 - \theta_2)$ and it follows that:
$$
z= F_{\theta_2} ( F_{\re(\theta_1 - \theta_2)}(\lambda) ).
$$
Since ${\rm th} ( 2 \re(\theta_1 - \theta_2) ) < 1$, then one has $\big\vert F_{\re(\theta_1 - \theta_2)}(\lambda) \big\vert \le 1$ for $\lambda \in [-1,1]$. 
By setting $\lambda' = F_{\re(\theta_1 - \theta_2)}(\lambda)$, one gets $z = F_{\theta_2}(\lambda')$ with $\lambda' \in [-1,1]$. This proves the inclusion
\begin{equation*}
\Big\lbrace F_{\theta_1}(\lambda) : \lambda \in [-1,1] \Big\rbrace \subset \Big\lbrace F_{\theta_2}(\lambda) : \lambda \in [-1,1] \Big\rbrace.
\end{equation*}
The opposite inclusion can be justified similarly once permuted the roles of $\theta_1$ and $\theta_2$. This proves the first claim. Statement (b) follows by direct calculations. Let us prove the last part. According to statement (a), $\theta$ can be chosen with $\re(\theta) = 0$, say $\theta = iy$ with $y \in {\mathbb R}$. The case $\im(\theta) = 0$ (i.e. $y=0$) is immediate since $\sigma ( \widehat \lao(\theta) ) = \sigma ( \widehat \lao(0) ) = \sigma ( \widehat \lao ) = \sigma ( \lao ) = [0,4]$. Now, suppose that $y \neq 0$. Observe that the map $T$ is a composition of a translation and a homothety. To prove that $\sigma ( \widehat \lao(\theta) )$ is an arc of a circle, it is enough to observe that $\{F_{iy}(\lambda): \lambda \in [-1,1]\}$ is a continuous parametrised curve and that for $\lambda \in [-1,1]$, the real and imaginary parts of $F_{iy}(\lambda)$ 
satisfy a circle equation. Indeed, denoting ${\rm th}(2\theta) = i \tan (2y) =: it$, one has
$$
F_{iy}(\lambda) = \frac{\lambda (1 + t^2)}{1 + \lambda^2 t^2} + i \frac{(\lambda^2 -1)t}{1 + \lambda^2 t^2} =: X + iY,
$$
and
\begin{equation}\label{circleequ}
X^2 + \Bigg( Y - \frac{1 - t^2}{2t} \Bigg)^2 = \Bigg( \frac{1 + t^2}{2t} \Bigg)^2.
\end{equation}
\end{proof}

\begin{rem} For $\theta \in D_{\frac{\pi}{8}}(0)$, $\im(\theta) \neq 0$, Equation \eqref{circleequ} allows to recover the center $c_{\theta} \in {\mathbb C}$ and the radius $R_{\theta} >0$ of the circle supporting $\sigma ( \widehat \lao(\theta) )$.
\end{rem}

\subsection{After perturbation}

Now, we focus on the complex scaling of the perturbation $\widehat V$ together with the perturbed operator $\widehat \la = \widehat \lao + \widehat V$. For $\widehat V \in {\mathcal A}_R (\widehat A_0)$, $R > 0$, and for all $\theta \in D_R(0)$, we set
\begin{equation*}
\widehat V(\theta) := \e^{i\theta\widehat A_0} \widehat V \ \e^{-i\theta\widehat A_0}.
\end{equation*}
The following lemma holds:

\begin{lem}\label{ld1}
Let $\widehat V \in \sinf ( {\rm L}^2({\mathbb T}) ) \cap {\mathcal A}_R (\widehat A_0)$, $R > 0$. Then, for all $\theta\in D_R(0)$, $\widehat V(\theta)$ 
is compact.
\end{lem}

\noindent
\begin{proof}
This follows from \cite[Lemma 5, Section 5]{RS4}, since $\widehat V(\cdot)$ is the analytic continuation of a bounded operator-valued function with 
compact values on the real axis.
\end{proof}

Now, for all $\theta \in D_{2R'}(0)$ with $2R' := \min (R,\frac{\pi}{8} )$, we consider
\begin{equation}
\widehat \la(\theta) := \widehat \lao(\theta) + \widehat V(\theta), \quad \widehat V \in {\mathcal A}_R (\widehat A_0).
\end{equation}
We obtain:

\begin{prop}\label{pdr1}
Let $\widehat V \in {\mathcal A}_R (\widehat A_0)$, $R > 0$. Then,
\begin{enumerate}
\item[(a)] $\widehat \la(\theta)$ is a holomorphic family of bounded operators on $D_{2R'}(0)$.
\item[(b)] For any $\theta' \in \R$ such that $\vert \theta' \vert < R'$, we have
\begin{equation*}
\widehat \la(\theta+ \theta')=\e^{i \theta' \widehat A_0} \widehat \la(\theta)\ \e^{-i \theta'\widehat A_0},
\end{equation*}
for all $\theta\in D_{R'}(0)$.
 \end{enumerate}
\end{prop}

\noindent
\begin{proof}
Statement (a) follows from Proposition \ref{p6,2}. Now, we prove (b). We fix $\theta' \in (-R' , R')$ and observe that $D_{R'} (0) \subset D_R (0) \cap D_R (-\theta')$. The maps
\begin{align*}
\theta \longmapsto \e^{i\theta'\widehat A_0} \widehat \la(\theta)\ \e^{-i\theta'\widehat A_0} \quad {\rm and} \quad \theta \longmapsto \widehat \la(\theta+\theta'),
\end{align*}
are bounded operator-valued and holomorphic on $D_{R'}(0)$. Moreover, they coincide on $\R\, \cap D_{R'}(0) = (-R',R')$. Hence, they also coincide on $D_{R'}(0)$.
\end{proof}

The next proposition gives the key to the proofs of Theorems \ref{t:tda} and \ref{t:tda+}.

\begin{prop}\label{pdr2}
Let $R > 0$ and $\widehat V \in \sinf ( {\rm L}^2({\mathbb T}) ) \cap {\mathcal A}_R (\widehat A_0)$, and let $R' > 0$ such that $2R ' = \min(R,\frac\pi8)$. Then, for any $\theta\in D_{R'}(0)$, we have
\begin{enumerate}
\item[(a)] $\sigma ( \widehat \la(\theta) )$ depends only on $\im (\theta)$.
\item[(b)] It holds: $\sigma_\mathrm{ess} ( \widehat \la(\theta) ) = \sigma_\mathrm{ess} ( \widehat \lao(\theta) ) = \sigma ( \widehat \lao(\theta) )$ and 
\begin{equation*}
\sigma ( \widehat \la(\theta) ) = \sigma_\mathrm{disc} ( \widehat \la(\theta) ) \bigsqcup \sigma_\mathrm{ess} ( \widehat \lao(\theta) ),
\end{equation*}
where the possible limit points of $\sigma_\mathrm{disc} ( \widehat \la(\theta) )$ lie in $\sigma_\mathrm{ess} ( \widehat \lao(\theta) )$.
\end{enumerate}
\end{prop}

\noindent
\begin{proof}
Statement (a) is a consequence of the unitary equivalence established in Proposition \ref{pdr1} (b). Statement (b) follows from Lemma \ref{ld1}, the Weyl criterion on the invariance of the essential spectrum and \cite[Theorem 2.1, p. 373]{goh} (see also \cite[Corollary 2, p.113]{RS4}).
\end{proof}

\subsection{Proof of Theorem \ref{t:tda}}

The proof of Theorem \ref{t:tda} follows from Proposition \ref{t:discrete:sp} below as an adaptation of the usual complex scaling arguments to our non-selfadjoint setting (see e.g. \cite[Theorem XIII.36]{RS4}, \cite{hissig}).

For any $\theta \in D_{\frac{\pi}{8}}(0)$, $c_{\theta} \in {\mathbb C}$ and $R_{\theta} >0$ stand respectively for the center and the radius of the circle supporting $\sigma ( \widehat \lao(\theta) )$. For $\pm \im(\theta) \geq 0$, we define the open domains $S^\mp_{\theta} := {\mathbb C}\setminus A^\pm_{\theta}$ where
$$
A^\pm_{\theta} := \big\lbrace z\in {\mathbb C} : \re(z) \in [0,4], \pm \im(z) \geq 0, |z-c_{\theta} |\geq R_{\theta} \big\rbrace, \quad \pm \im(\theta) > 0,
$$
and $A_0^{\pm} = \big\lbrace z\in {\mathbb C} : \re(z) \in [0,4], \pm \im(z) \geq 0 \big\rbrace$. According to Proposition \ref{p6,3}, 
the domains $S^{\pm}_{\theta}$ depend only on $\im(\theta)$. In addition, if $0\leq \im(\theta')  < \im(\theta)$, $S^-_{\theta'} \subsetneq S^-_{\theta}$ and if 
$\im(\theta)< \im(\theta') \leq 0$, $S^+_{\theta'} \subsetneq S^+_{\theta}$.

\begin{prop}\label{t:discrete:sp}
Let $R > 0$, $\widehat V \in{\mathcal A}_R (\widehat A_0)$ and $R' > 0$ such that $2R '= \min(R,\frac\pi8)$. Let $(\theta, \theta') \in D_{R'}(0)\times D_{R'}(0)$. 
Then:
\begin{enumerate}
\item[(a)] If $0\leq \im(\theta')  < \im(\theta)$, we have $\sigma_\mathrm{disc} ( \widehat \la(\theta') )\cap S^-_{\theta'} = \sigma_\mathrm{disc} ( \widehat \la(\theta) )\cap S^-_{\theta'} \subset \sigma_\mathrm{disc} ( \widehat \la(\theta) )\cap S^-_{\theta}$. 
In particular,
$$
\sigma_\mathrm{disc} ( \widehat \la(\theta) )\cap S_0^-=\sigma_\mathrm{pp} ( \widehat \la ) \cap S_0^-.
$$ 
\item[(b)] If $\im (\theta) < \im(\theta') \leq 0$, we have $\sigma_\mathrm{disc} ( \widehat \la(\theta') )\cap S^+_{\theta'} = 
\sigma_\mathrm{disc} ( \widehat \la(\theta) )\cap S^+_{\theta'} \subset \sigma_\mathrm{disc} ( \widehat \la(\theta) )\cap S^+_{\theta}$. 
In particular,
$$
\sigma_\mathrm{disc} ( \widehat \la(\theta) )\cap S_0^+ = \sigma_\mathrm{pp} ( \widehat \la ) \cap S_0^+.
$$ 
\end{enumerate}
As a consequence, the discrete spectrum of $\widehat \la$ (and $\la$) can only accumulate at $0$ and $4$. 
\end{prop}

\begin{figure}[!h]
\begin{center}
\includegraphics[scale=0.75]{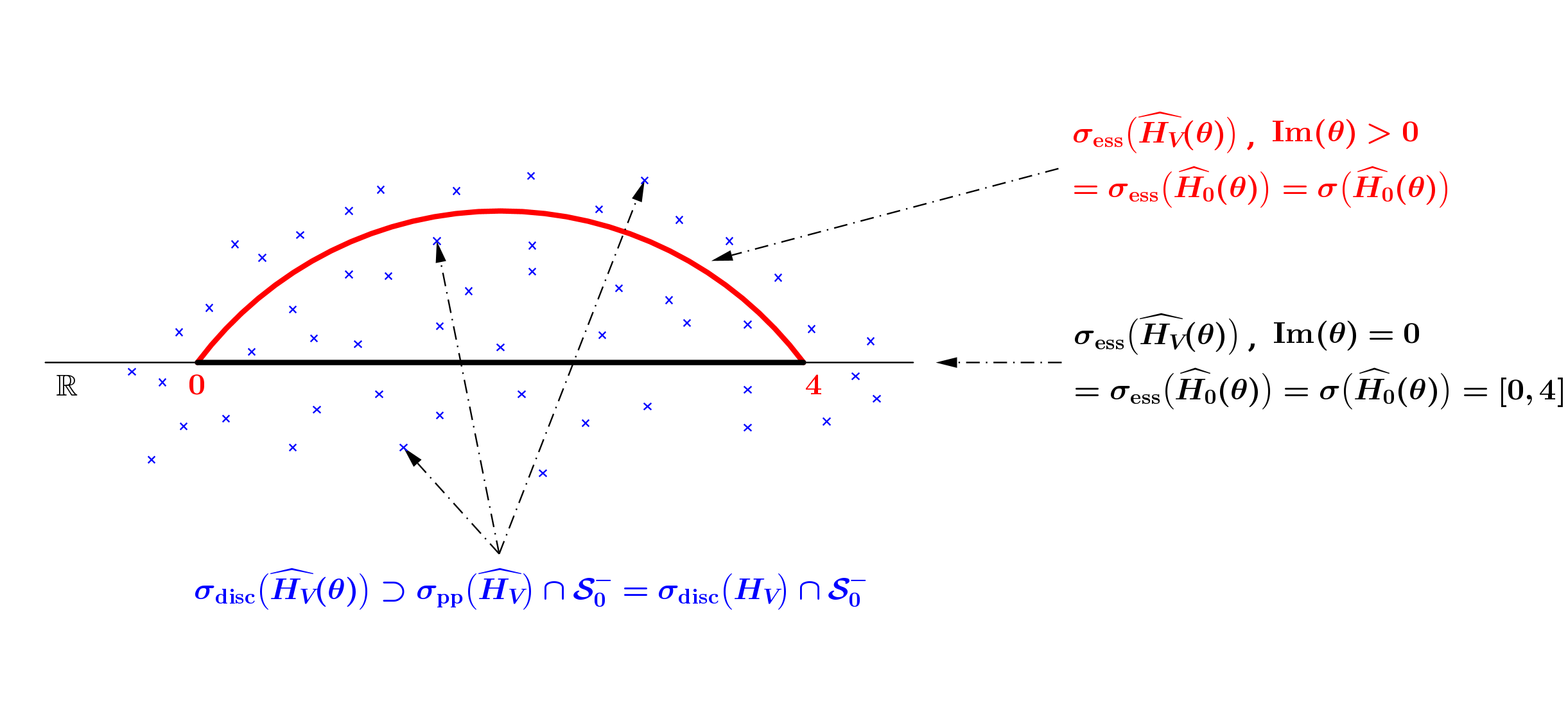}
\vspace*{-0.7cm}
\caption{Spectral structure of the operator $\widehat \la(\theta)$ for $\theta \in D_\frac{\pi}{8}(0)$ and $\im(\theta) \le 0$.} \label{figd}
\end{center}
\end{figure}

\noindent
\begin{proof} 
We focus our attention on case (a). For a moment, fix $\theta_0\in D_{R'}(0)$ such that $\im(\theta_0) >0$ and suppose that $\lambda\in \sigma_\mathrm{disc} ( \widehat \la(\theta_0))$. Since the map $\theta\longmapsto \widehat \la(\theta)$ is analytic, there exist open neighborhoods $\mathcal{V}_{\theta_0}$ and $\mathcal{W}_\lambda$ of $\theta_0$ and $\lambda$ respectively such that \cite{kato,RS4}:
\begin{itemize}
\item For all $\theta\in\mathcal{V}_{\theta_0}$, the operator  $\widehat \la(\theta) $ has a finite number of eigenvalues in $\mathcal{W}_\lambda$ denoted by $\lambda_j(\theta)\in\{1,\dots,n\}$, all of finite multiplicity.
\item These nearby eigenvalues are given by the branches of a finite number of holomorphic functions in $\mathcal{V}_{\theta_0}$ with at worst algebraic branch point near $\theta_0$.
\end{itemize}

\noindent
If $\varphi=\theta-\theta_0\in\R$ for  $\theta\in\mathcal{V}_{\theta_0}$, then $\widehat \la(\theta_0+\varphi)$ and $\widehat \la(\theta_0) $ are unitarily equivalent according to Proposition \ref{pdr1}. So, the only eigenvalue of $\widehat \la(\theta_0+\varphi)$ near $\lambda$ in $\mathcal{W}_\lambda$ is $\lambda$. Therefore, 
$\lambda_j(\theta)=\lambda$ for any $j\in\{1,\dots,n\}$ and for $\theta\in\mathcal{V}_{\theta_0}$ with $\theta-\theta_0\in\R$. By analyticity, we deduce that for all $\theta\in\mathcal{V}_{\theta_0}$ and all $j\in\{1,\dots,n\}$ one has $\lambda_j(\theta)=\lambda$. Finally, we have proved that given $\theta_0\in D_{R'}(0)$ and 
$\lambda\in \sigma_\mathrm{disc}( \widehat \la(\theta_0))$, there exists a neighborhood $\mathcal{V}_{\theta_0}$ of $\theta_0$, such that for all $\theta\in\mathcal{V}_{\theta_0}$, $\lambda\in\sigma_\mathrm{disc} ( \widehat \la(\theta) )$. 
Now, following \cite[Problem 76, Section XIII]{RS4}, if $\gamma\in\mathcal{C}^0 ( [0,1],D_{R'}(0) )$ is a continuous curve and $\lambda\in\sigma_\mathrm{disc} 
( \widehat \la(\gamma(0)))$, then either $\lambda\in\sigma_\mathrm{disc} ( \widehat \la(\gamma(1)) )$ or $\lambda\in\sigma_\mathrm{ess} 
( \widehat \la(\gamma(t)) )$ for some $t\in(0,1]$. We apply this observation twice keeping in mind Proposition \ref{pdr2}.

Fix $\theta\in D_{R'}(0)$ with $\im(\theta)>0$. Starting with $\lambda\in\sigma_\mathrm{disc} ( \widehat \la)\cap S_0^- = \sigma_\mathrm{pp}( \widehat \la)\cap S_0^-$, considering the continuous curve $\gamma^+: [0,1]\to D_{R'}(0)$
$$
\gamma^+ (t) = t\theta,
$$
and applying the above observation yields: $\sigma_\mathrm{disc} ( \widehat \la)\cap S_0^- \subset \sigma_\mathrm{disc} ( \widehat \la(\theta) )\cap S_0^-$. Now starting with $\lambda\in\sigma_\mathrm{disc} ( \widehat \la(\theta) )\cap S_0^-$, considering the opposite continuous curve $\gamma^-: [0,1]\to D_{R'}(0)$
$$
\gamma^- (t)= (1-t) \theta,
$$
and applying the above observation yields the opposite inclusion. So, we have proven that given $\theta\in D_{R'}(0)$ with $\im(\theta)>0$, $\lambda \in S_0^-$, then $\lambda\in\sigma_\mathrm{disc} ( \widehat \la(\theta) )\cap S_0^-$ if and only if $\lambda\in\sigma_\mathrm{disc} ( \widehat \la)\cap S_0^- = \sigma_\mathrm{pp}( \widehat \la)\cap S_0^-$.

The proof of the identity $\sigma_\mathrm{disc} ( \widehat \la(\theta') ) \cap S^-_{\theta'} = \sigma_\mathrm{disc} ( \widehat \la(\theta) )\cap 
S^-_{\theta'}$ for any $0\leq \im(\theta')  < \im(\theta)$ is similar. The inclusion $\sigma_\mathrm{disc} ( \widehat \la(\theta) ) \cap S^-_{\theta'} \subset \sigma_\mathrm{disc} ( \widehat \la(\theta) ) \cap S^-_{\theta}$ follows from the inclusion $S^-_{\theta'} \subsetneq S^-_{\theta}$, which allows us to conclude on case (a).

The proof of case (b) is analog.

Once proven Statements (a) and (b), we conclude as follows. From Proposition \ref{pdr2}, we know that the limit points of $\sigma_\mathrm{disc} (\widehat \la)$ belong 
necessarily to $\sigma_\mathrm{ess} (\widehat \la) = [0,4]$. Pick one of these points, say $\lambda_0$. It is necessarily the limit point of a subsequence of 
either $\sigma_\mathrm{disc} (\widehat \la) \cap S_0^-$ or $\sigma_\mathrm{disc} (\widehat \la) \cap S_0^+$. Without any loss of generality, assume 
there exists a subsequence of $\sigma_\mathrm{disc} (\widehat \la) \cap S_0^-$ which converges to $\lambda_0 \in [0,4]$. Since 
$\sigma_\mathrm{disc} (\widehat \la) \cap S_0^-= \sigma_\mathrm{disc} ( \widehat \la (\theta) ) \cap S_0^-$ for any $\im(\theta) > 0$, 
$\theta \in D_{R'} (0)$, $\lambda_0$ also belongs to $\sigma_\mathrm{ess} ( \widehat \la (\theta) )$. Since $\sigma_\mathrm{ess} ( \widehat \la 
(\theta) ) \cap \sigma_\mathrm{ess} (\widehat \la) =\{0,4\}$ for any $\im(\theta) > 0$, $\lambda_0\in \{0,4\}$ and the last part follows.
\end{proof}

\subsection{Proof of Theorem \ref{t:tda+}}

Let $R > 0$, $\widehat V \in{\mathcal A}_R (\widehat A_0)$ and $R' > 0$ such that $2R '= \min(R,\frac\pi8)$. Observe first that for any $\theta \in D_{R'}(0)$ with $\im (\theta) >0$, ${\mathcal D}_+:= \sigma_\mathrm{disc} ( \widehat \la(\theta) ) \cap (0,4)$ is discrete and its possible accumulation points belong to $\{0,4\}$. Let $\varphi$ and $\psi$ be two analytic vectors for $\widehat{A}_0$ with convergence radius $R_0 >0$ and denote by $\varphi: \theta \mapsto \varphi (\theta)$ and $\psi: \theta \mapsto \psi (\theta)$ their analytic extension on $D_{R_0}(0)$. Consider the function $F(z,\theta) = \bra \psi (\bar{\theta}), (H_V(\theta) -z)^{-1} \varphi (\theta) \ket$ whenever it exists. For $\theta \in D_{\min(R_0,R')}(0)$, $F(\cdot, \theta)$ is analytic in ${\mathbb C}\setminus \sigma ( \widehat \la(\theta) )$ and meromorphic in ${\mathbb C}\setminus \sigma_\mathrm{ess} ( \widehat \la(\theta) )$. Now, let us fix $z\in S_0^-\setminus \sigma_\mathrm{disc} (\widehat \la)$. Then $F(z,\cdot)$ is analytic on some region $D_{\min(R_0,R')}(0)\cap \{\theta \in {\mathbb C} : -\epsilon_z < \im(\theta) \}$, for some $\epsilon_z >0$. Since for any $\eta \in D_{\min(R_0,R')}(0) \cap {\mathbb R}$, direct calculations yield $F(z,\eta)=F(z,0)$, we conclude by analyticity that $F(z,\cdot)$ is constant in $D_{\min(R_0,R')}(0)\cap \{\theta \in {\mathbb C} :-\epsilon_z < \im (\theta) \}$. In particular, given $\theta_0 \in D_{\min(R_0,R')}(0)$ with $\im (\theta_0) >0$, $F(\cdot,\theta_0)$ provides an analytic continuation to the function $F(\cdot,0)$ from $S_0^-\setminus \sigma_\mathrm{disc} (\widehat \la)$ to $S^-_{\theta_0}\setminus \sigma_\mathrm{disc} (\widehat \la ( \theta_0) )$. In particular,  for any relatively compact interval $\Delta_0$, $\overline{\Delta_0} \subset (0,4)\setminus {\mathcal D}_+$, the map $\bra \psi, (H_V -z)^{-1} \varphi \ket$ extends continuously from some region $\Delta_0 - i(0,\delta_0^-]$, for some $\delta_0^- >0$, to $\Delta_0 - i[0,\delta_0^-]$. Note that $\delta_0^-$ can be chosen as:
$$
\delta_0^- =\text{dist} (\overline{\Delta_0}, \min \{\im (z); z\in \sigma_\mathrm{disc} (\widehat \la) \cap S_0^-\}) >0.
$$
This proves the first statement.

The proof of the other case is similar (with $\theta \in D_{R'}(0)$ and $\im (\theta) <0$). Once proven both cases, we can set ${\mathcal D}={\mathcal D}_+ \cup {\mathcal D}_-$ and the proposition follows.


\section{Resonances for exponentially decaying perturbations}\label{resonances}

Throughout this section, the perturbation $V$ is assumed to satisfy Assumption \ref{ade}. We aim at defining and characterizing the resonances of the 
operator $\la$ near the thresholds $0$ and $4$.

\subsection{Preliminaries}

The purpose of this first paragraph is to prove the following result:
\begin{lem}\label{l3,1}
Set $z(k) := k^2$. Then, there exists $0 < \varepsilon_0 \le \frac{\delta}{8}$ small enough such that the operator-valued function with values in $\sinf ( \ell^2({\mathbb Z}) )$
\begin{equation*}
k \mapsto W_{-\delta} ( \lao - z(k) )^{-1} W_{-\delta},
\end{equation*}
admits a holomorphic extension from $D_{\varepsilon_0}^\ast(0) \cap {\mathbb C}^+$ to $D_{\varepsilon_0}^\ast(0)$.
\end{lem}

\noindent
\begin{proof}
For $k \in D_{\varepsilon}^\ast(0) \cap {\mathbb C}^+$, 
$0 < \varepsilon < \frac{\delta}{4}$ small enough, and $x \in \ell^2({\mathbb Z})$, 
the operator $W_{-\delta} ( \lao - z(k) )^{-1} W_{-\delta}$ satisfies
\begin{equation}\label{eq3,60}
\left(W_{-\delta} ( \lao - z(k) )^{-1}W_{-\delta} x \right) (n) = 
\sum_{m \, \in \, {\mathbb Z}} {\rm e}^{-\frac{\delta}{2} \vert n \vert} 
R_0 ( z(k),n - m ) {\rm e}^{-\frac{\delta}{2} \vert m \vert}  x (m), 
\end{equation}
$R_0(z,\cdot)$ being the function defined by \eqref{R_0}. Thus, for any $(n,m) \in {\mathbb Z}^2$, we have
\begin{equation}\label{eq3,6}
{\rm e}^{-\frac{\delta}{2} \vert n \vert} R_0 ( z(k),n - m ) {\rm e}^{-\frac{\delta}{2} \vert m \vert}  = {\rm e}^{-\frac{\delta}{2} \vert n \vert} 
\frac{i {\rm e}^{ 2i \vert n - m \vert \arcsin \frac{k}{2}}}{k \sqrt{4 - k^2}}{\rm e}^{-\frac{\delta}{2} \vert m \vert}  =: f(k).
\end{equation}
Since for $0 < \vert k \vert \ll 1$ we have $2\arcsin \frac{k}{2} = k + o(\vert k \vert)$, then, there exists $0 < \varepsilon_0 \le \frac{\delta}{8}$ small enough such that 
for any $0 < \vert k \vert \le 
\varepsilon_0$, 
\begin{equation}\label{eq3,7}
\Big\vert {\rm e}^{-\frac{\delta}{2} \vert n \vert} R_0 ( z(k),n - m ) {\rm e}^{-\frac{\delta}{2} \vert m \vert}  \Big\vert \le {\rm e}^{-\frac{\delta}{2} \vert n \vert}
\frac{{\rm e}^{-\big( \im(k) - \frac{\delta}{4} \big) \vert n - m \vert}}{ \big\vert k \sqrt{4 - k^2} \big\vert} {\rm e}^{-\frac{\delta}{2} \vert m \vert} .
\end{equation}
It can be checked that the r.h.s. of \eqref{eq3,7} lies in 
$\ell^{2}({\mathbb Z}^{2},{\mathbb C})$ whenever
$0 \le \vert \im(k) \vert < \frac{\delta}{8}$.
In this case, the operator $W_{-\delta} ( \lao - z(k) )^{-1}W_{-\delta}$ belongs to $\sd ( \ell^2({\mathbb Z}) )$. Consequently, the operator-valued function 
$$
k \mapsto W_{-\delta} ( \lao - z(k) )^{-1} W_{-\delta} 
$$ 
can be extended via the kernel \eqref{eq3,6} from $D_{\varepsilon_0}^\ast(0) \cap {\mathbb C}^+$ to $D_{\varepsilon_0}^\ast(0)$, with values in 
$\sinf ( \ell^2({\mathbb Z}) )$. We shall denote this extension $A(k)$, $k \in D_{\varepsilon_0}^\ast(0)$. It remains to prove that $k \mapsto A(k)$ is 
holomorphic in $D_{\varepsilon_0}^\ast(0)$. 

\medskip

For $k \in D_{\varepsilon_0}^\ast(0)$, introduce $\mathscr D(k)$ the operator with kernel given by 
\begin{align*}
{\rm e}^{-\frac{\delta}{2} \vert n \vert} \partial_k R_0 ( z(k),n - m ) 
{\rm e}^{-\frac{\delta}{2} \vert m \vert}   
 = {\rm e}^{-\frac{\delta}{2} \vert n \vert} 
\frac{i{\rm e}^{ 2i \vert n - m \vert \arcsin \frac{k}{2}}}{k^2(4 - k^2)}
{\rm e}^{-\frac{\delta}{2} \vert m \vert}  \bigg( 2ik \vert n - m \vert - \frac{4-2k^2}{\sqrt{4-k^2}} \bigg).
\end{align*}
As above, it can be shown that $\mathscr D(k) \in \sd ( \ell^2({\mathbb Z}) )$. Then, for $k_0 \in D_{\varepsilon_0}^\ast(0)$, the kernel of the Hilbert-Schmidt operator 
$$
\frac{A(k) - A(k_0)}{k - k_0} - \mathscr D(k_0)$$
is:
\begin{equation}\label{eq3,71}
N(k,k_0,n,m) := {\rm e}^{-\frac{\delta}{2} \vert n \vert} \bigg( \frac{ R_0 ( z(k),n - m ) -  R_0 ( z(k_0),n - m )}{k - k_0} - \partial_k R_0 
( z(k_0),n - m ) \bigg) {\rm e}^{-\frac{\delta}{2} \vert m \vert} .
\end{equation}
To conclude the proof, it remains to prove that for any fixed $k_0 \in D_{\varepsilon_0}^\ast(0)$,
$$
\Big\Vert \frac{A(k) - A(k_0)}{k - k_0} - \mathscr D(k_0) \Big\Vert_{\sd(\ell^2({\mathbb Z}))}
$$
vanishes as $k$ tends to $k_0$. Since we have
\begin{equation}\label{eq3,70}
\bigg\Vert \frac{A(k) - A(k_0)}{k - k_0} - \mathscr D(k_0) \bigg\Vert_{\sd(\ell^2({\mathbb Z}))}^2 
\le \sum_{n,m} \big\vert N(k,k_0,n,m) \big\vert^2,
\end{equation}
it is actually enough to prove that the r.h.s. of \eqref{eq3,70} tends to zero
as $k$ tends to $k_0$. By applying the Taylor formula with integral remainder to the 
function 
$$
[0,1] \ni t \mapsto g(t) := f ( tk + (1-t)k_0 ) = R_0 ( z(tk + (1-t)k_0),n - m ),
$$
one gets
\begin{equation}\label{eq3,72}
\begin{split}
R_0 ( z(k),n - m ) & = R_0 ( z(k_0),n - m ) + (k - k_0) \partial_k R_0 ( z(k_0),n - m ) \\
& + (k - k_0)^2 \int_0^1 (1-s) \partial_k^{(2)} R_0 ( z(s k + (1-s)k_0),n - m ) \, ds.
\end{split}
\end{equation}
Then, it follows from \eqref{eq3,71} and \eqref{eq3,72} that the kernel $N(k,k_0,n,m)$ 
has the representation
\begin{equation}\label{eq3,73}
N(k,k_0,n,m) = (k - k_0) {\rm e}^{-\frac{\delta}{2} (\vert n \vert +\vert m \vert)} \int_0^1 (1-s) \partial_k^{(2)} R_0 ( z(s k + (1-s)k_0),n - m ) \, ds .
\end{equation}
By standard calculations, we can prove that for any $q \in {\mathbb N}$, there exists a family of functions $G_{j,q}$, 
$0 \le j \le q$, holomorphic on $D_{\varepsilon_0}^\ast(0)$ such that
$$
\partial_k^{(q)} R_0 ( z(k),n - m ) = i{\rm e}^{ 2i \vert n - m \vert \arcsin \frac{k}{2}}
\sum_{j=0}^q G_{j,q}(k) \vert n - m \vert^{q-j}.
$$
In particular, for $q = 2$, one has
$$
\partial_k^{(2)} R_0 ( z(k),n - m ) = i{\rm e}^{ 2i \vert n - m \vert \arcsin \frac{k}{2}}
( G_{0,2}(k) \vert n - m \vert^2 + G_{1,2}(k) \vert n - m \vert + G_{2,2}(k) ).
$$
Together with \eqref{eq3,73}, this implies that: $\big\vert N(k,k_0,n,m) \big\vert \le \sum_{j=0}^2 N_j(k,k_0,n,m)$, where
\begin{equation*}
\begin{split}
N_j(k,k_0,n,m):= \vert k - k_0 \vert \vert n - m \vert^{2-j} & {\rm e}^{-\frac{\delta}{2} (\vert n \vert+\vert m \vert)} \\
& \times \int_0^1 (1-s) \big\vert G_{j,2} ( s k + (1-s)k_0 ) \big\vert \Big\vert {\rm e}^{ 2i \vert n - m \vert \arcsin \frac{s k + (1-s)k_0}{2}} \Big\vert \, ds ,
\end{split}
\end{equation*}
for $j\in \{0,1,2\}$. Therefore,
\begin{equation}\label{eq3,74}
\sum_{n,m} \vert N(k,k_0,n,m) \vert^2 \le 3 \sum_{j=0}^2 \Bigg( \sum_{n,m} N_j (k,k_0,n,m)^2 \Bigg) .
\end{equation}
In the sequel, one shows that for each $j \in \lbrace 0,1,2 \rbrace$, $\sum_{n,m} N_j(k,k_0,n,m)^2$ vanishes as $k$ tends to $k_0$. Similarly to \eqref{eq3,7}, one
has for $k$, $k_0 \in D_{\varepsilon_0}^\ast(0)$
\begin{align*}
\Big\vert {\rm e}^{-\frac{\delta}{2} \vert n \vert} {\rm e}^{ 2i \vert n - m \vert \arcsin \frac{s k + (1-s)k_0}{2}} 
{\rm e}^{-\frac{\delta}{2} \vert m \vert}  \Big\vert & \le {\rm e}^{-\frac{\delta}{2} \vert n \vert} 
{\rm e}^{-( \im(s k + (1-s)k_0) - \frac{\delta}{4} ) \vert n - m \vert} {\rm e}^{-\frac{\delta}{2} \vert m \vert}  \\
& \le {\rm e}^{-\frac{\delta}{2} (\vert n \vert +\vert m \vert)} {\rm e}^{\frac{3\delta}{8} \vert n - m \vert} .
\end{align*}
Thus, for $j \in \lbrace 0,1,2 \rbrace$, it follows that
\begin{gather*}
N_j(k,k_0,n,m)^2 \le \vert k - k_0 \vert^2 \Big\vert \int_0^1 (1-s) G_{j,2}(s k + (1-s)k_0) \, ds \Big\vert^2 {\rm e}^{-\delta (\vert n \vert +\vert m \vert)} {\rm e}^{\frac{3\delta}{4} \vert n - m \vert} \vert n - m \vert^{2(2-j)}.
\end{gather*}
Since the sequence $( {\rm e}^{-\delta (\vert n \vert +\vert m \vert)} {\rm e}^{\frac{3\delta}{4} \vert n - m \vert} \vert n - m \vert^{2(2-j)} )_{n,m}$ 
belongs to $\ell^{2}({\mathbb Z}^{2},{\mathbb C})$, we deduce
$$
\sum_{n,m} N_j(k,k_0,n,m)^2 \le C \vert k - k_0 \vert^2 \Big\vert  \int_0^1 (1-s) G_{j,2}(s k + (1-s)k_0) \, ds \Big\vert^2 \underset{k \rightarrow k_0}{\longrightarrow} 0 ,
$$
where $C$ is a positive constant. Back to \eqref{eq3,70} and \eqref{eq3,74}, we have thus proved that the operator-valued extension $D_{\varepsilon_0}^\ast(0) \ni k \mapsto A(k)$ is holomorphic at $k_0$ with derivative $\partial_k A(k_0) = \mathscr D(k_0)$. Since $k_0$ has been chosen arbitrarily in $D_{\varepsilon_0}^\ast(0)$, the conclusion follows. \end{proof}

\subsection{Resonances as poles: proof of Proposition \ref{p3,1}}

Starting with the identity
$$
(H_{\textup{\bf V}} - z)^{-1} \left( I + {\bf V}(\lao - z)^{-1} \right) = (\lao - z)^{-1},
$$
we get: $W_{-\delta} (H_{\textup{\bf V}} - z)^{-1}W_{-\delta} = W_{-\delta} (\lao - z)^{-1}W_{-\delta} \left( I + W_{\delta} {\bf V}(\lao - z)^{-1} 
W_{-\delta} \right)^{-1}$.
Assumption \ref{ade} on $V$ implies that there exists a bounded operator $\mathscr{V}$ on $\ell^2({\mathbb Z})$ such that: $V = W_{-\delta} \mathscr{V} W_{-\delta}$ and $V_J = W_{-\delta} \mathscr{V}_J W_{-\delta}$ with $\mathscr{V}_J = J \mathscr{V} J$. We sum it up by writing:
\begin{equation}\label{eq3,100}
 {\bf V} = W_{-\delta} \mathfrak{V} W_{-\delta} \: \: \textup {or equivalently} \: \: W_\delta {\bf V} = \mathfrak{V} W_{-\delta} ,
\end{equation}
with $\mathfrak{V} = \mathscr{V}$ (resp. $-\mathscr{V}_J$) if ${\bf V} = V$ (resp. $-V_J$). Putting this together with Lemma \ref{l3,1}, it follows that the operator-valued 
function
\begin{equation}\label{eq3,aa}
k \longmapsto \mathcal{T}_{\textup{\bf V}} ( z(k) ) := W_\delta {\bf V} ( \lao - z(k) )^{-1}W_{-\delta} = \mathfrak{V} W_{-\delta} ( \lao - z(k) )^{-1}W_{-\delta}
\end{equation}
is holomorphic in $D_{\varepsilon_0}^\ast(0)$, with values in $\sinf ( \ell^2({\mathbb Z}) )$. Then, by the analytic Fredholm extension theorem, 
$$
k \longmapsto \left( I + W_\delta {\bf V} ( \lao - z(k) )^{-1}W_{-\delta} \right)^{-1}
$$
admits a meromorphic extension from $D_{\varepsilon_0}^\ast(0) \cap {\mathbb C}^+$ to $D_{\varepsilon_0}^\ast(0)$. So, we have proved Proposition \ref{p3,1}.
\newline

The compact operator $\mathcal{T}_{\textup{\bf V}} ( z(k) )$ defined by \eqref{eq3,aa} will play a fundamental role in the rest of our analysis.

\begin{rem}
Of course, the factorization ${\bf V} = W_{-\delta} \mathfrak{V} W_{-\delta}$ in \eqref{eq3,100} is not unique and other 
decompositions are sometimes more convenient than \eqref{eq3,100}. For instance, with respect to the polar 
decomposition of an operator, one has ${\bf V} = {\bf J} \vert {\bf V} \vert = {\bf J} \vert {\bf V} \vert^{1/2} \vert {\bf V} \vert^{1/2}$. 
\end{rem}

\subsection{Resonances as characteristic values}

Our next task is to provide a simple and useful characterization of the resonances of $\la$ near the thresholds $0$ and $4$. We recall some basic facts concerning the concept of characteristic values of a holomorphic operator-valued function. For more details on this subject, we refer for instance to \cite{go} and \cite[Section 4]{goh}. The content of the first part of this section follows \cite[Section 4]{goh}. 

\begin{de}
Let $\mathcal{U}$ be a neighborhood of a fixed point $w \in {\mathbb C}$, and $F : \mathcal{U} \setminus \lbrace w \rbrace \longrightarrow 
\mathcal{B}({\mathscr H})$ be a holomorphic operator-valued function. The function $F$ is said to be finite meromorphic at $w$ if its Laurent expansion 
at $w$ has the form
\begin{equation*}
F(z) = \sum_{n = m}^{+\infty} (z - w)^n F_n, \quad m > - \infty,
\end{equation*}
where (if $m < 0$) the operators $F_m, \ldots, F_{-1}$ are of finite rank. Moreover, if $F_0$ is a Fredholm operator, then, the function $F$ is said to be Fredholm 
at $w$. In this case, the Fredholm index of $F_0$ is called the Fredholm index of $F$ at $w$.
\end{de}

We have the following proposition:

\begin{prop}{\cite[Proposition 4.1.4]{goh}}\label{p,a1}
Let $\mathcal{D} \subseteq \mathbb{C}$ be a connected open set, $Z \subseteq \mathcal{D}$ be a closed and discrete subset of $\mathcal{D}$, and 
$F : \mathcal{D} \longrightarrow \mathcal{B}({\mathscr H})$ be a holomorphic operator-valued function in $\mathcal{D} \backslash Z$. Assume that:
\begin{itemize}
\item $F$ is finite meromorphic on $\mathcal{D}$ (i.e. it is finite meromorphic near each point of $Z$),
\item $F$ is Fredholm at each point of $\mathcal{D}$,
\item there exists $w_0 \in \mathcal{D} \backslash Z$ such that $F(w_0)$ is invertible. 
\end{itemize}
Then, there exists a closed and discrete subset $Z'$ of $\mathcal{D}$ such that:
\begin{enumerate}
\item $Z \subseteq Z'$,
\item $F(z)$ is invertible for each $z \in \mathcal{D} \backslash Z'$,
\item $F^{-1} : \mathcal{D} \backslash Z' \longrightarrow {\rm GL}({\mathscr H})$ is finite meromorphic and Fredholm at each point of $\mathcal{D}$.
\end{enumerate}
\end{prop}

In the setting of Proposition \ref{p,a1}, we define the characteristic values of $F$ and their multiplicities as follows:

\begin{de}\label{d,a1}
The points of $Z'$ where the function $F$ or $F^{-1}$ is not holomorphic are called the characteristic values of $F$. The multiplicity of a characteristic 
value $w_0$ is defined by
\begin{equation*}
{\rm mult}(w_0) := \frac{1}{2i\pi} \textup{Tr} \int_{\vert w - w_0 \vert = \rho} 
F'(z)F(z)^{-1} dz,
\end{equation*}
where $\rho > 0$ is chosen small enough so that $\big\lbrace w \in {\mathbb C} : \vert w - 
w_0 \vert \leq \rho \big\rbrace \cap Z' = \lbrace w_0 \rbrace$.
\end{de}

According to Definition \ref{d,a1}, if the function $F$ is holomorphic in $\mathcal{D}$, then, the characteristic values of $F$ are just the complex numbers 
$w$ where the operator $F(w)$ is not invertible. Then, results of \cite{go} and \cite[Section 4]{goh} imply that ${\rm mult}(w)$ is an integer.

Let $\Omega \subseteq \mathcal{D}$ be a connected domain with boundary $\partial \Omega$ 
not intersecting $Z'$. The sum of the multiplicities of the characteristic values of
the function $F$ lying in $\Omega$ is called {\it the index of $F$ with respect to the 
contour $\partial \Omega$} and is defined by 
\begin{equation}\label{eqa,2}
Ind_{\partial \Omega} \hspace{0.5mm} F := \frac{1}{2i\pi} \textup{Tr} 
\int_{\partial \Omega} F'(z)F(z)^{-1} dz = \frac{1}{2i\pi} \textup{Tr} 
\int_{\partial \Omega} F(z)^{-1} F'(z) dz.
\end{equation} 

Let us contextualize the previous discussion for our model. We reformulate our characterization of the resonances near $z = 0$ as follows:
\begin{prop}\label{p3,2}
For $k _1\in D_{\varepsilon_0}^\ast(0)$, the following assertions are 
equivalent:
\begin{itemize}
\item[(a)] $z_0(k_1) = k_1^{2} \in \mathcal{M}_0$ is a resonance of $\la$, 
\item[(b)] $z_{1} = z_0(k_1)$ is a pole of $R_V(z)$, 
\item[(c)] $-1$ is an eigenvalue of $\mathcal{T}_{V}(z_0(k_1)) := \mathscr{V} W_{-\delta} R_{0} ( z_0(k_1) )  W_{-\delta}$,
\item[(d)] $k_1$ is a characteristic value of  
$I + \mathcal{T}_{V} ( z_0(\cdot) )$. 

\noindent
Moreover, thanks to \eqref{eq3,14}, the multiplicity of the resonance $z_0(k_1)$ coincides with that of the characteristic
value $k_1$.
\end{itemize}
\end{prop}

\noindent
\begin{proof}
The equivalence between (a) and (b) is just Definition \ref{d3,1}. The equivalence between (b) and (c) is an immediate consequence of the identity
\begin{equation*}
( I +  \mathscr{V} W_{-\delta} R_{0}(z) W_{-\delta} ) ( I - \mathscr{V} W_{-\delta} R_V(z) W_{-\delta} ) = I,
\end{equation*}
which follows from the resolvent equation. Finally, the equivalence between (c) and (d) follows from Definition \ref{d,a1}.
\end{proof}

Similarly, we reformulate our characterization of the resonances near $z = 4$ as follows:
\begin{prop}\label{p3,3}
For $k_1 \in D_{\varepsilon_0}^\ast(0)$, the following assertions are equivalent:

\begin{itemize}
\item[(a)] $z_4(k_1) = 4 - k_1^{2} \in \mathcal{M}_4$ is a resonance of $\la$, 
\item[(b)] $4 - z_4(k_1) = k_1^2$ is a pole of $R_{-V_J}(u)$, 
\item[(c)] $-1$ is an eigenvalue of $\mathcal{T}_{-V_J}(4 - z_4(k_1)) := -\mathscr{V}_J W_{-\delta} R_{0} ( 4 - z_4(k_1) ) W_{-\delta}$,
\item[(d)] $k_1$ is a characteristic value of $I + \mathcal{T}_{-V_J}(4 - z_4( \cdot) )$.

\noindent
Moreover, thanks to \eqref{eq3,15}, the multiplicity of the resonance $z_4(k_1)$ coincides with that of the characteristic
value $k_1$.

\end{itemize}
\end{prop}

\subsection{Proof of Theorem \ref{t1}}\label{proofs}

Our first goal is to decompose the weighted resolvent $\mathcal{T}_{\bf V} ( z(k) )$, for $z(k) = k^2$ and ${\bf V} \in \{ V, -V_J\}$ near the spectral threshold $z = 0$. More precisely, we split it into a sum of a singular part at $k = 0$ and a holomorphic part in the whole open disk $D_{\varepsilon_0}(0)$, with values in $\sinf ( \ell^2({\mathbb Z}) )$. 

From \eqref{eq3,60}, we know that the kernel of $W_{-\delta} ( \lao - z(k) )^{-1} W_{-\delta}$ is given by
$$
{\rm e}^{-\frac{\delta}{2} \vert n \vert} R_0 ( z(k),n - m ) {\rm e}^{-\frac{\delta}{2} \vert m \vert},
$$
with
\begin{equation}\label{eq3,17}
\begin{split}
R_0 ( z(k),n - m ) & = \frac{i{\rm e}^{i \vert n - m \vert 2\arcsin \frac{k}{2}}}{k \sqrt{4 - k^2}} = \frac{i}{k \sqrt{4 - k^2}} +
\frac{i ( {\rm e}^{i \vert n - m \vert 2\arcsin \frac{k}{2}} - 1 )}{k \sqrt{4 - k^2}} \\
& = \frac{i}{2k} + \alpha(k) + \beta(k),
\end{split}
\end{equation}
where the functions $\alpha$ and $\beta$ are defined by
\begin{equation*}
\alpha(k) := i \left( \frac{1}{k \sqrt{4 - k^2}} - \frac{1}{2k} \right) \quad {\rm and} \quad \beta(k) := \frac{i ( e^{i \vert n - m \vert 
2\arcsin \frac{k}{2}} - 1 )} {k \sqrt{4 - k^2}}.
\end{equation*}
Moreover, it is not difficult to show that the functions $\alpha$ and $\beta$ can be extended to holomorphic functions in
$D_{\varepsilon_0}(0)$. By combining \eqref{eq3,60} and \eqref{eq3,17}, we get for $k \in D_{\varepsilon_0}^\ast(0)$
\begin{equation}\label{eq3,19}
( W_{-\delta} ( \lao - z(k) )^{-1} W_{-\delta} x ) (n) = \sum_{m \, \in \, {\mathbb Z}} \frac{i{\rm e}^{-\frac{\delta}{2} \vert n \vert} {\rm e}^{-\frac{\delta}{2} \vert m \vert}  
x(m)}{2k} + ( \mathcal{A} (k) x )(n),
\end{equation}
where $\mathcal{A}(k)$ is the operator defined by  
\begin{equation}\label{eq3,20}
( \mathcal{A} (k) x )(n) := \sum_{m \, \in \, {\mathbb Z}} {\rm e}^{-\frac{\delta}{2} \vert n \vert} \alpha(k) {\rm e}^{-\frac{\delta}{2} \vert m \vert}  
x(m) + \sum_{m \, \in \, {\mathbb Z}} {\rm e}^{-\frac{\delta}{2} \vert n \vert} \beta(k) {\rm e}^{-\frac{\delta}{2} \vert m \vert}  x(m).
\end{equation}
Let $\Theta : \ell^2({\mathbb Z}) \longrightarrow {\mathbb C}$ be the operator defined by $\Theta(x) := \big\langle x,{\rm e}^{-\frac{\delta}{2} \vert \cdot \vert} 
\big\rangle_{\ell^2({\mathbb Z})}$so that its adjoint $\Theta^\ast : {\mathbb C} \longrightarrow \ell^2({\mathbb Z})$ be given by $( \Theta^\ast(\lambda) )(n) := \lambda 
{\rm e}^{-\frac{\delta}{2} \vert n \vert} $. Then, this together with \eqref{eq3,19} yields
\begin{equation}\label{eq3,200}
( W_{-\delta} ( \lao - z(k) )^{-1} W_{-\delta} x ) (n) =  \frac{i ( \Theta^\ast \Theta x )(n)}{2k} + ( \mathcal{A} (k) x )(n).
\end{equation}
By combining \eqref{eq3,aa} and \eqref{eq3,200}, we finally obtain
\begin{equation*}
\mathcal{T}_{\bf V} ( z(k) ) = \mathfrak{V} \frac{i \Theta^\ast \Theta}{2k} + \mathfrak{V} \mathcal{A}(k).
\end{equation*}
We therefore have proved the following result:

\begin{prop}\label{p4,1} 
For ${\bf V} \in \{V, -V_J\}$ and $k \in D_{\varepsilon_0}^\ast(0)$, the operator $\mathcal{T}_{\bf V} ( z(k) )$ admits the decomposition 
\begin{equation*}
\mathcal{T}_{\bf V} ( z(k) ) = \frac{i \mathfrak{V}}{k} \mathscr{M}^\ast \mathscr{M} + \mathfrak{V} \mathcal{A}(k), \quad \mathscr{M} := \frac{1}{\sqrt{2}} \Theta,
\end{equation*}
with the operator $\mathcal{A}(k)$ given by \eqref{eq3,20}. Moreover, $k \mapsto \mathfrak{V} \mathcal{A}(k)$ is holomorphic in the open disk
$D_{\varepsilon_0}(0)$, with values in $\sinf ( \ell^2({\mathbb Z}) )$. 
\end{prop}

\begin{rem}\label{r4,1}
Notice that $\mathscr{M} \mathscr{M}^\ast : {\mathbb C} \longrightarrow {\mathbb C}$ is the rank-one operator given by
\begin{equation*}
\mathscr{M} \mathscr{M}^\ast (\lambda) = \frac{\lambda}{2} \big\langle {\rm e}^{-\frac{\delta}{2} \vert \cdot \vert},{\rm e}^{-\frac{\delta}{2} \vert \cdot \vert} 
\big\rangle_{\ell^2({\mathbb Z})} = \frac{\lambda}{2} {\rm Tr} ( {\rm e}^{-\delta \vert \cdot \vert} \,),
\end{equation*}
${\rm Tr} ( {\rm e}^{-\delta \vert \cdot \vert} \,)$ standing for the trace of the multiplication operator by the function ${\mathbb Z} \ni n \mapsto {\rm e}^{-\delta \vert n \vert}$.
Then, so is the operator $\mathscr{M}^\ast \mathscr{M} : \ell^2({\mathbb Z}) \longrightarrow \ell^2({\mathbb Z})$.
\end{rem}

The decomposition obtained in Proposition \ref{p4,1} allows us to finish the proof of Theorem \ref{t1}. 

Let $\mu = 0$. From Propositions \ref{p3,2} and \ref{p4,1}, it follows that $z_0 (k)$ is a resonance of $\la$ near $0$ if and only if $k$ is a characteristic value of the operator
\begin{equation*}
I + \mathcal{T}_V ( z(k) ) = I + \frac{i \mathscr{V}}{k} \mathscr{M}^\ast \mathscr{M} + \mathscr{V} \mathcal{A}(k).
\end{equation*}
The operator $\mathscr{V} \mathcal{A}(k)$ is holomorphic in the open disk $D_{\varepsilon_0}(0)$ with values in $\sinf ( \ell^2({\mathbb Z}) )$, while $i \mathscr{V} \mathscr{M}^\ast \mathscr{M}$ is a finite-rank operator.

Let $\mu =4$. From Propositions \ref{p3,3} and \ref{p4,1}, it follows that $z_4 (k)$ is a resonance of $\la$ near $4$ if and only if $k$ is a characteristic value of the operator
\begin{equation*}
I + \mathcal{T}_{- V_J} ( z(k) ) = I - \frac{i \mathscr{V}_J }{k} \mathscr{M}^\ast \mathscr{M} - \mathscr{V}_J \mathcal{A}(k).
\end{equation*}
The operator $\mathscr{V}_J \mathcal{A}(k)$ is holomorphic in the open disk $D_{\varepsilon_0}(0)$ with values in $\sinf ( \ell^2({\mathbb Z}) )$, while $i \mathscr{V}_J \mathscr{M}^\ast \mathscr{M}$ is a finite-rank operator.
\newline

So, Theorem \ref{t1} is obtained after applying Proposition \ref{p,a1} as follows:
\begin{itemize}
\item in the case $\mu=0$: take $\mathcal{D} = D_{\varepsilon_0}(0)$, $Z = \lbrace 0 \rbrace$ and $F = I + \mathcal{T}_V ( z(\cdot) )$,
\item in the case $\mu=4$: take $\mathcal{D} = D_{\varepsilon_0}(0)$, $Z = \lbrace 0 \rbrace$ and $F = I + \mathcal{T}_{-V_J} ( z(\cdot) )$.
\end{itemize}
This concludes the proof.

\subsection{Proof of Theorem \ref{t2}}

Since Assumption \ref{ade} holds, $V\in {\mathcal A}(A_0)$ (see Proposition \ref{pe3}). According to Theorem \ref{t:tda}, the only possible accumulation points of $\sigma_{\textup{disc}} (H_V)$ are $0$ and $4$. But, according to Theorem \ref{t1}, $\sigma_{\textup{disc}} (H_V) \cap (D_{r_0}(0) \cup D_{r_0}(4)) =\emptyset$ for some $r_0 >0$. As a subset of the compact rectangle $[-\|V \|, 4+\|V \|]+i[-\|V \|, \|V\|]$ with no accumulation point, $\sigma_{\textup{disc}} (H_V)$ is necessarily finite.

Combining Proposition \ref{p3,1} with Theorem \ref{t1}, the map $z\mapsto (z - H_V)^{-1}$ can be extended continuously in the operator norm topology of ${\mathcal B}(\ell_{\delta}^{2}({\mathbb Z}) ; \ell_{-\delta}^{2}({\mathbb Z}))$:
\begin{itemize}
\item from $\C^+ \setminus \sigma_{\textup{disc}} (H_V)$ to $( \C^+ \setminus \sigma_{\textup{disc}} (H_V) )\cup (0,r_0) \cup (4-r_0,4)$ and
\item from $\C^- \setminus \sigma_{\textup{disc}} (H_V)$ to $( \C^- \setminus \sigma_{\textup{disc}} (H_V) )\cup (0,r_0) \cup (4-r_0,4)$.
\end{itemize}
In particular, for any vectors $\varphi$, $\psi$ in $\ell_{\delta}^{2}({\mathbb Z})$, any interval $\Delta_0$ with $\overline{\Delta_0} \subset (0,r_0) \cup (4-r_0,4)$, it holds:
\begin{gather}\label{lap-thresholds}
\sup_{z\in \Delta_0 + i(-\delta_0,0)} |\bra \varphi, (z-H_V)^{-1} \psi \ket | < \infty  \quad \text{and} \quad \sup_{z\in \Delta_0 + i(0,\delta_0)} |\bra \varphi, (z-H_V)^{-1} \psi \ket | < \infty
\end{gather}
for some $\delta_0 >0$. Now, fix $0< r < r_0$ small enough. Due to Theorem \ref{t:tda+}, the set ${\mathcal D}' := {\mathcal D} \cap [r,4-r]$ is finite. Applying Theorem \ref{t:tda+} to vectors $\varphi$, $\psi$ in $\ell_{\delta}^{2}({\mathbb Z})$ on the interval $[r,4-r]$ and combining it with \eqref{lap-thresholds} allows to conclude.


\section{On positive commutators}\label{mourre}

In this section, we extend the validity of the Mourre theory in a non-selfadjoint setting under optimal regularity assumptions. In what follows, ${\mathscr H}$ is a fixed Hilbert space and $A$ is a selfadjoint operator densely defined on ${\mathscr H}$. We recall that the classes $C^k(A)$ and ${\mathcal C}^{1,1}(A)$, $k\in {\mathbb N}$, have been introduced in Definitions \ref{ck} and \ref{c11}. Throughout this section, $H$ belongs to $\cal B({\mathscr H} )$ and we denote $\sigma_+ : =\max \sigma ( \im(H) )$ and 
$\sigma_- :=\min \sigma ( \im(H) )$.

As a general rule, a statement involving the symbol $\pm$ has to be understood as two independent statements.

\subsection{Abstract results}

The abstract results we are about to introduce relate the existence of positive commutation relations with the control of some spectral properties of the operator $H$. We start by describing these positivity conditions. The spectral measure of the self-adjoint operator $\re (H)$ is denoted by $E_{\re (H)}$. Let $\Delta \subset {\mathbb R}$ be a Borel subset: \\

\noindent{\bf Assumptions}
\begin{itemize}
\item[{\bf (M$_+$)}]
$\re (H)\in C^1(A)$ and there exist $a_+ >0$, $b_+ >0$ and $\beta_+ \geq 0$ such that
\begin{equation*}
i\mathrm{ad}_A ( \re H) + \beta_+ (\sigma_+ - \im H)\geq a_+ E_{\re(H)}(\Delta) -b_+ E_{\re(H)}^{\perp}(\Delta).
\end{equation*}
\item[{\bf (M$_-$)}] 
$\re (H)\in C^1(A)$ and there exist $a_- >0$, $b_- >0$ and $\beta_- \geq 0$ such that
\begin{equation*}
i\mathrm{ad}_A ( \re H) + \beta_- (\im H - \sigma_-)\geq a_- E_{\re(H)}(\Delta) -b_- E_{\re(H)}^{\perp}(\Delta).
\end{equation*}
\item[{\bf (M)}\,\,\,\,] 
$\re (H)\in C^1(A)$ and there exist $c_{\Delta} >0$, $K\in \sinf(\mathscr H)$, such that
\begin{equation}\label{mourreK}
E_{\re(H)}(\Delta ) i\left( \mathrm{ad}_A ( \re H)\right) E_{\re(H)}(\Delta ) \geq c_{\Delta} E_{\re(H)}(\Delta ) + K .
\end{equation}
\end{itemize}

\begin{rem} Assumptions (M$_{\pm}$) stated with $\beta_{\pm}=0$ are equivalent to Assumption (M) stated with $K=0$. Indeed, if $B\in\mathcal{B}({\mathscr H})$ is symmetric and $E$ is an orthogonal projection acting on ${\mathscr H}$, then the following statements are equivalent:
\begin{enumerate}
\item[(a)] There exists $c>0$ such that $EBE\geq c E$.
\item[(b)] There exist $a>0$, $b>0$, such that $B\geq a E- b E^\bot$, with $E^\bot=1-E$.
\item[(c)] There exist $a>0$, $b>0$, such that $B\geq a - (a+b) E^\bot$, with $E^\bot=1-E$,
\end{enumerate}
where inequalities are understood in the sense of quadratic forms.
\end{rem}

We start by a reformulation of \cite[Proposition 2.4]{royer2} and refer to Section \ref{viriallike} for its proof.
\begin{lem}\label{virialnsa} Let $H \in \cal B (\mathscr H)$ and $\Delta \subset {\mathbb R}$ be a Borel subset.
\begin{itemize}
\item Assume (M) holds on $\Delta$. Then, ${\mathcal E}_{\rm p} (H) \cap ( \Delta +i\{\sigma_{\pm}\} ) \subset ( \sigma_{\rm pp} ( \re H) \cap \Delta ) +i\{\sigma_{\pm}\}$. In particular, the set ${\mathcal E}_{\rm p} (H) \cap ( \Delta +i\{\sigma_{\pm}\} )$ is finite and each of these eigenvalues has finite geometric multiplicity. 
If (M) holds with $K=0$, then ${\mathcal E}_{\rm p} (H) \cap ( \Delta +i\{\sigma_{\pm}\} )=\emptyset$.
\item Assume that $\re(H) \in C^1(A)$ and that $i\mathrm{ad}_A ( \re H) >0$ (i.e. positive and injective). Then, ${\mathcal E}_{\rm p} (H) \cap ( \R +i\{\sigma_{\pm}\} )=\emptyset$.
\end{itemize}
\end{lem}

\begin{rem}\label{c11ReIm} If $H\in C^1(A)$, then, $\re(H)$ and $\im(H)$ belong to $C^1(A)$ and,
\begin{equation*}
\begin{split}
\mathrm{ad}_A ( \re H) &= i \im ( \mathrm{ad}_A H),\\
\mathrm{ad}_A ( \im H) &= -i \re ( \mathrm{ad}_A H).
\end{split}
\end{equation*}
Indeed, $H^*\in C^1(A)$ and $\mathrm{ad}_A (H^*)=- ( \mathrm{ad}_A H)^*$, see e.g. \cite[Proposition 5.1.7]{abmg}.
\end{rem}
We recall that if $H\in {\mathcal C}^{1,1}(A)$, then $H\in C^1(A)$. 

\begin{theo}\label{mourrensa} 
Let $H\in {\mathcal C}^{1,1}(A)$, $\Delta \subset \R$ be an open interval and $s>1/2$.
\begin{itemize}
\item Assume (M$_+$) holds on $\Delta$. For any relatively compact interval $\Delta_0$, $\overline{\Delta_0} \subset \Delta$, we have
$$
\sup_{z\in \Delta_0 +i(\sigma_+ ,\infty)}\| \bra A\ket^{-s} (z-H)^{-1} \bra A\ket^{-s} \| < \infty.
$$
\item Assume (M$_-$) holds on $\Delta$. For any relatively compact interval $\Delta_0$, $\overline{\Delta_0} \subset \Delta$, we have
$$
\sup_{z\in \Delta_0 +i(-\infty, \sigma_-)}\| \bra A\ket^{-s} (z-H)^{-1} \bra A\ket^{-s} \| < \infty.
$$
\end{itemize}
\end{theo}

The proof of Theorem \ref{mourrensa} is carried out through Sections \ref{diffineq}--\ref{linktoc11}.

\begin{rem}\label{farfromnr} Once observed that for any $s>0$ and for any $z\in \Delta_0 +i ((-\infty ,\sigma_- -1] \cup [\sigma_+ +1, \infty))$, we have 
$$
\| (z-H)^{-1} \| \leq 1/{\rm dist} ( z, {\mathcal N}(H) ) \leq 1,
$$
the proof of Theorem \ref{mourrensa} reduces to bound the weighted resolvent uniformly for $z \in \Delta_0 + i(\sigma_+ ,\sigma_+ +1]$ and $z \in \Delta_0 + i[\sigma_- -1, \sigma_-)$ respectively.
\end{rem}

\begin{rem}\label{weights} \begin{itemize}
\item[(a)] Theorem \ref{mourrensa} can be formulated equivalently with weights of the form $h(A)$, $\bar{h}(A)$ instead of $\bra A\ket^{-s}$, where $h:{\mathbb R}\rightarrow {\mathbb C}$ is such that: $0< c_1\leq h(x) \bra x\ket^{s} \leq c_2$ for some $c_1 >0$, $c_2 >0$ and $s>1/2$. For technical reasons, the proof of Theorem \ref{mourrensa} is actually developed with weights of the form $h(A)$ where $h(x) = (|x| +1)^{-s}$.
\item[(b)] Interpolation spaces can also be introduced within this framework in order to obtain an optimal version on the Besov scales.
\end{itemize}
\end{rem}

\begin{rem} We expect that Theorem \ref{mourrensa} can be adapted to include the case of some unbounded non-selfadjoint operators, in the spirit of e.g. \cite[Section 7.4]{abmg}.
\end{rem}

Theorem \ref{mourrensa} can be reformulated as follows:
\begin{cor}\label{Hsmooth} 
Let $H\in {\mathcal C}^{1,1}(A)$, $\Delta \subset \R$ be an open interval and $s>1/2$. Let $\Delta_1$ be an open bounded interval such that $\overline{\Delta_0}\subset \Delta_1\subset \overline{\Delta_1} \subset \Delta$. Let $\chi \in C_0^{\infty}({\mathbb R})$ be supported on $\Delta_1$ and taking value $1$ on $\overline{\Delta_0}$.
\begin{itemize}
\item Assume (M$_+$) holds on $\Delta$. Then, we have
\begin{equation*}
\sup_{z\in {\mathbb R} +i(\sigma_+ ,\infty)} \big \|\bra A\ket^{-s} \chi ( \re H) (z-H)^{-1} \chi ( \re H) \bra A\ket^{-s} \big \| < \infty.
\end{equation*}
\item Assume (M$_-$) holds on $\Delta$. Then, we have
\begin{equation*}
\sup_{z\in {\mathbb R} +i(-\infty,\sigma_-)} \big \| \bra A\ket^{-s} \chi ( \re H) (z-H)^{-1} \chi ( \re H) \bra A\ket^{-s} \big \| < \infty.         
\end{equation*}
\end{itemize}
\end{cor}

Synthesizing Lemma \ref{virialnsa} and Theorem \ref{mourrensa}, we obtain:
\begin{cor}\label{mourrensa1}
Let $H\in {\mathcal C}^{1,1}(A)$. Assume (M) holds on some open interval $\Delta \subset \R$. Then,
\begin{itemize}
\item ${\mathcal E}_{\rm p} (H) \cap ( \Delta +i\{\sigma_{\pm}\} ) \subset ( \sigma_{\rm pp} ( \re H) \cap \Delta ) +i\{\sigma_{\pm}\}$. In particular, the set ${\mathcal E}_{\rm p} (H) \cap ( \Delta +i\{\sigma_{\pm}\} )$ is finite and each of these eigenvalues has finite geometric multiplicity. 
\item Let $s>1/2$. For any relatively compact interval $\Delta_0 \subset \Delta$ such that $\overline{\Delta_0} \subset \Delta \setminus {\mathcal E}_{\rm p} (\re (H)),$
\begin{gather*}
\sup_{z\in \Delta_0 +i(\sigma_+ ,\infty)}\| \bra A\ket^{-s} (z-H)^{-1} \bra A\ket^{-s} \| < \infty ,\\
\sup_{z\in \Delta_0 +i(-\infty, \sigma_-)}\| \bra A\ket^{-s} (z-H)^{-1} \bra A\ket^{-s} \| < \infty.
\end{gather*}
\end{itemize}
\end{cor}
The proofs of Corollaries \ref{Hsmooth} and \ref{mourrensa1} are postponed to Sections \ref{proofHsmooth} and \ref{proofmourrensa1} respectively. \\

We conclude this section by the following remark. If $H$ is selfadjoint in Theorem \ref{mourrensa}, that is $H=\re(H)$, $\im(H) = 0$ and $\sigma_{\pm}=0$, we recover \cite[Theorem 7.3.1]{abmg}. Namely,
\begin{theo}\label{mourresa} Let $H\in {\mathcal C}^{1,1}(A)$ be selfadjoint, $\Delta \subset \R$ be an open interval and $s>1/2$. Assume (M$_{\pm}$) hold on $\Delta$ with $\beta_{\pm}=0$, then for any relatively compact interval $\Delta_0$, $\overline{\Delta_0} \subset \Delta$, we have
\begin{equation}
\sup_{z\in \Delta_0 + i{\mathbb R}^*} \| \bra A\ket^{-s} (z-H)^{-1} \bra A\ket^{-s} \| < \infty.
\end{equation}
\end{theo}

\subsection{Proof of Lemma \ref{virialnsa}}\label{viriallike}

The next result takes its source in \cite[Lemma 1]{dav01}:
\begin{lem}\label{evfrontier} 
Let $H \in \cal B (\mathscr H)$. 
Assume that there exist $\varphi \in {\mathscr H}$ and $\lambda \in {\mathbb R}$ such that 
$H\varphi = (\lambda +i\sigma_{\pm}) \varphi$. Then, $\re(H)\varphi = \lambda \varphi$. In particular, for any subset $\Delta \subset {\mathbb R},$
\begin{equation}
{\mathcal E}_{\rm p} (H) \cap ( \Delta +i\{\sigma_{\pm}\} ) \subset \left( {\mathcal E}_{\rm p} (\re H) \cap \Delta \right) 
+i\{\sigma_{\pm}\}.
\end{equation}
\end{lem}
\begin{proof} Mind that the operators $\sigma_+ - \im(H)$ and $\im(H) - \sigma_-$ are bounded and non-negative. To fix ideas, let us consider the $+$ case. By hypothesis, we have 
\begin{equation}
\bra \varphi, \im(H) \varphi\ket = \im \bra \varphi, H \varphi\ket = \sigma_+ \|\varphi \|^2
\end{equation}
We deduce that $\big\| \sqrt{\sigma_+ - \im(H)} \varphi \big\|^2 = \bra \varphi, ( \sigma_+ - \im(H) ) \varphi \ket =0$, so 
$\sqrt{\sigma_+ - \im(H)} \varphi=0$. Hence, $\im(H)\varphi = \sigma_+ \varphi$ and $\re(H)\varphi = \lambda \varphi$.
\end{proof}

Let us prove now Lemma \ref{virialnsa}. Assume (M) holds. As a consequence of the Virial Theorem, see e.g. \cite[Corollary 7.2.11]{abmg}, $\sigma_{\rm pp} ( \re(H) ) \cap \Delta$ is necessarily finite with finite multiplicity. Moreover, it is empty if $K=0$. In view of Lemma \ref{evfrontier}, this proves the first statement of Lemma \ref{virialnsa}.

If there exists $\varphi \in {\mathscr H}$ such that $H\varphi = (\lambda +i\sigma_+) \varphi$ for some $\lambda \in \R$, then $\re (H) \varphi =\lambda \varphi$, due to Lemma \ref{evfrontier}. Assuming that $\re(H) \in C^1(A)$, we can apply the Virial Theorem (see e.g. \cite[Proposition 7.2.10]{abmg}) to deduce that $\bra \varphi, i\mathrm{ad}_A ( \re H) \varphi \ket =0$. If in addition $i\mathrm{ad}_A ( \re H) >0$, then necessarily $\varphi=0$. This means that ${\mathcal E}_{\rm p} (H) \cap ( \R +i\{\sigma_+\} )=\emptyset$. Similarly, we have that ${\mathcal E}_{\rm p} (H) \cap ( \R +i\{\sigma_-\} )=\emptyset$, which proves the second statement of Lemma \ref{virialnsa}.

\subsection{Deformed resolvents and first estimates}\label{diffineq}

The proof of Theorem \ref{mourrensa} is based on Mourre's differential inequality strategy. Our presentation interpolates between \cite{abmg} and \cite{royer1}. Throughout this section, $\Delta \subset {\mathbb R}$ denotes an open interval and $H\in C^1(A)$. In particular $\re (H) \in C^1(A)$ (see Remark \ref{c11ReIm}). For further use, we formulate various additional hypotheses, that will be related to the condition $H\in {\mathcal C}^{1,1}(A)$ in Section \ref{linktoc11}.

We assume there exist $0< \epsilon_0 \leq 1$ and two maps
\begin{eqnarray*}
S: (0,\epsilon_0) &\rightarrow & {\mathcal B}({\mathscr H}), \\
B: (0,\epsilon_0) &\rightarrow & {\mathcal B}({\mathscr H}),
\end{eqnarray*}
such that

\medskip

\noindent{\bf Assumptions}
\begin{itemize}
\item[{\bf (A1)}] there exists $C>0$ so that for any $\epsilon \in (0,\epsilon_0)$, $\| S(\epsilon) - H\| \leq C\epsilon$,
\item[{\bf (A2)}] $\style\lim_{\epsilon \rightarrow 0} \| B(\epsilon) - i\mathrm{ad}_A(H) \|=0$,
\item[{\bf (A3)}] $\style\lim_{\epsilon \rightarrow 0} \epsilon^{-1}\| S(\epsilon) - H\| =0$,
\item[{\bf (A4)}] the maps $S$ and $B$ are continuously differentiable on $(0,\epsilon_0)$  w.r.t. the operator norm topology,
\item[{\bf (A5)}] $\style\sup_{\epsilon \in (0,\epsilon_0)} \| B(\epsilon) \| <\infty$,
\item[{\bf (A6)}] for any $\epsilon \in (0,\epsilon_0)$, $S(\epsilon)$ and $B(\epsilon)$ belong to $C^1(A)$.
\end{itemize}

\begin{rem} \begin{itemize}
\item[(a)] Of course, Hypothesis (A3) implies (A1), while (A2) implies (A5). We have separated them in order to identify properly the ingredients required at each step of the proof.
\item[(b)] Note that Assumptions (A1)--(A4) allow to extend continuously the functions $S$, $B$ and $\partial_{\epsilon} S$ on $[0,\epsilon_0)$, 
by setting $S(0)=H$, $B(0)=i\mathrm{ad}_A(H)$, so that $(\partial_{\epsilon} S)(0) = 0$.
\end{itemize}
\end{rem}

For $\epsilon \in (0,\epsilon_0) \times \C$, we denote
\begin{equation*}
\begin{split}
Q_{\epsilon}^{\pm} & := \frac{S(\epsilon ) -H}{\epsilon} \mp i ( B(\epsilon ) -i\mathrm{ad}_A(H) ),\\
R_{\epsilon}^{\pm} & := \frac{S(\epsilon ) -H}{\epsilon} \mp i B(\epsilon ) = Q_{\epsilon}^{\pm} \pm \mathrm{ad}_A(H),
\end{split}
\end{equation*}
and set $Q_0^{\pm}=0$, $R_0^{\pm} := \pm \mathrm{ad}_A(H)$. We also define for $(\epsilon , z)\in [0,\epsilon_0) \times \C$,
\begin{equation}
\begin{split}
T_{\epsilon}^{\pm} (z) & = z- S(\epsilon) \pm i\epsilon B(\epsilon ) = T_0 (z) \pm i\epsilon ( i\mathrm{ad}_A(H) ) - \epsilon Q_{\epsilon}^{\pm} \\
&= T_0 (z) - \epsilon R_{\epsilon}^{\pm},
\end{split}
\end{equation}
where $T_0(z) = T_0^{\pm} (z) = z-H$.

Due to Assumptions (A1)--(A4), the maps $\epsilon \mapsto Q_{\epsilon}^{\pm}$ and $\epsilon \mapsto R_{\epsilon}^{\pm}$ are continuous on $[0,\epsilon_0)$ 
w.r.t. the operator norm topology. Using (A1), (A3) and (A5), we also have
\begin{equation}
C_R := \sup_{\epsilon \in [0,\epsilon_0)} \| R_{\epsilon}^{\pm} \| < \infty .
\end{equation}
Thus, it follows that:
\begin{lem}\label{Tepsilon-T0} 
Assume that (A1)--(A5) hold. Then, for any $z\in \C$, any $\epsilon \in [0,\epsilon_0)$, we have
\begin{equation}
\big\| T_{\epsilon}^{\pm} (z) - T_0 (z) \big\| \leq C_R \epsilon \quad \mbox{and} \quad \big\| T_{\epsilon}^{\pm} (z)^* - T_0 (z)^* \big\| \leq C_R \epsilon .
\end{equation}
\end{lem}

\begin{lem}\label{ImT0} 
Assume that (A1)--(A4) hold. Then, for any $z\in \C$ such that $\im(z) \in [\sigma_+, \sigma_+ +1]\cup [\sigma_- -1, \sigma_-]$, any 
$\epsilon \in (0,\epsilon_0]$, any $p>0$ and any $\psi \in {\mathscr H}$, we have
\begin{equation}
\begin{split}
\big| \bra \psi, \im ( T_0(z) ) \psi \big| & \leq \dfrac{1}{2p} \|T_{\epsilon}^{\pm}(z) \psi \|^2 + ( C_R \epsilon + \frac{p}{2} ) \|\psi \|^2, \\
\big| \bra \psi, \im ( T_0(z) ) \psi \big| & \leq \dfrac{1}{2p} \big\| T_{\epsilon}^{\pm}(z)^* \psi \big\|^2 +( C_R \epsilon + \frac{p}{2} ) \|\psi \|^2.
\end{split}
\end{equation}
In particular, for any $z\in \C$ such that $\im(z) \in [\sigma_+, \sigma_+ +1]\cup [\sigma_- -1, \sigma_-]$, any 
$\epsilon \in (0,\epsilon_0]$, any $p>0$ and any $\psi \in {\mathscr H}$, we also have
\begin{equation}
\begin{split}
\big\| \im ( T_0(z) ) \psi \big\|^2 & \leq \dfrac{\sigma_0}{2p} \|T_{\epsilon}^{\pm}(z) \psi \|^2 + 
\sigma_0 ( C_R \epsilon + \frac{p}{2} ) \|\psi \|^2, \\
\big\| \im ( T_0(z) ) \psi \big\|^2 & \leq \dfrac{\sigma_0}{2p} \big\| T_{\epsilon}^{\pm}(z)^* \psi \big\|^2 + 
\sigma_0 ( C_R \epsilon + \frac{p}{2} ) \|\psi \|^2,
\end{split}
\end{equation}
where $\sigma_0 := \sigma_+ - \sigma_- +1$. 
\end{lem}

\noindent
\begin{proof} Due to Lemma \ref{Tepsilon-T0}, for any $z\in \C$, any $\psi \in {\mathscr H}$ and any $\epsilon\in [0,\epsilon_0),$
\begin{equation*}
|\bra \psi, \im (T_0 (z)) \psi \ket | \leq |\bra \psi, T_0 (z) \psi \ket | \leq |\bra \psi, T_{\epsilon}^{\pm}(z) \psi \ket | + C_R \epsilon \|\psi \|^2 \leq \| \psi \| \| T_{\epsilon}^{\pm}(z) \psi \| 
+ C_R \epsilon \|\psi \|^2,
\end{equation*}
while for any $p>0,$
\begin{equation*}
\| \psi \| \| T_{\epsilon}^{\pm}(z) \psi \| \leq \frac{1}{2p} \| T_{\epsilon}^{\pm}(z) \psi \|^2 + \frac{p}{2} \|\psi \|^2.
\end{equation*}
Once observed that $\im (T_0 (z)^*) =- \im (T_0 (z))$, some analog inequalities hold with $T_0 (z)^*$ and $T_{\epsilon}^{\pm}(z)^*$ instead of $T_0 (z)$ and $T_{\epsilon}^{\pm}(z)$. This allows us to conclude on the first estimates. Now, note that $\im ( T_0(z) ) = (\im z - \sigma_{\pm}) +(\sigma_{\pm} -\im (H))$. In particular, $\im ( T_0(z) ) \geq 0$ if $\im(z) \in [\sigma_+, \sigma_+ +1]$ and $\im ( T_0(z) ) \leq 0$ if $\im(z) \in [\sigma_- -1, \sigma_-]$. So, for any $\psi \in {\mathscr H}$ and any $z\in \C$ such that $\im z \in [\sigma_+, \sigma_+ +1]\cup [\sigma_- -1, \sigma_-],$
\begin{equation}\label{signedqf}
\big\| \im ( T_0(z) ) \psi \big\|^2 \leq \big\| \im ( T_0(z) ) \big\| \big| \bra \psi, \im ( T_0(z) ) \psi \ket \big| \leq \sigma_0 \big| \bra \psi, \im ( T_0(z) ) \psi \ket \big|.
\end{equation}
The last estimates follow immediately from the first ones.
\end{proof}

In order to shorten some notations, we denote for any set $S \subset \R,$
\begin{equation*}
\begin{split}
\bar{\Omega}^+ (S) & := S +i[\sigma_+,\sigma_+ +1],\\
\bar{\Omega}^- (S) & := S +i[\sigma_- -1,\sigma_-],\\
\Omega^+ (S) & := S +i(\sigma_+,\sigma_+ +1],\\
\Omega^- (S) & := S +i[\sigma_- -1,\sigma_-) .
\end{split}
\end{equation*}
Until the end of the section, we fix $\Delta_0$ as a relatively compact subset of $\Delta$, such that $\overline{\Delta_0} \subset \Delta$. In particular, $\delta_0 :=$ dist$(\Delta_0, \Delta^c) >0$.

\begin{lem}\label{perp} 
Assume that (A1)--(A5) hold. For any $z\in \bar{\Omega}^{\pm} (\Delta_0)$, any $\epsilon \in [0,\epsilon_0)$, any $p>0$ and any $\varphi \in {\mathscr H}$, we have
\begin{equation}
\begin{split}
\| E_{\re(H)}^{\perp}(\Delta) \varphi \|^2 &\leq 3\delta_0^{-2} ( (1 + \dfrac{\sigma_0}{2p}) \|T_{\epsilon}^{\pm}(z) \varphi \|^2 + ( C_R^2 \epsilon^2 +\sigma_0 C_R \epsilon + \sigma_0 \frac{p}{2}) \|\varphi \|^2 ),\\
\| E_{\re(H)}^{\perp}(\Delta) \varphi \|^2 &\leq 3\delta_0^{-2} ( (1 + \dfrac{\sigma_0}{2p}) \|T_{\epsilon}^{\pm}(z)^* \varphi \|^2 + ( C_R^2 \epsilon^2 +\sigma_0 C_R \epsilon + \sigma_0 \frac{p}{2}) \|\varphi \|^2 ).
\end{split}
\end{equation}
\end{lem}

\noindent
\begin{proof} 
For $\re(z) \in \Delta_0$ we have $\delta_0 \big\| ( \re T_0(z) )^{-1} E_{\re(H)}^{\perp}(\Delta) \big\| \leq 1$. So, if $z \in \bar{\Omega}^{\pm} (\Delta_0)$, we can write 
\begin{equation}
\begin{split}
E_{\re(H)}^{\perp}(\Delta) &= ( \re T_0(z) )^{-1} E_{\re(H)}^{\perp}(\Delta) \left( T_{\epsilon}^{\pm}(z) - i\im ( T_0 (z) ) + \epsilon R_{\epsilon}^{\pm} \right), \\
E_{\re(H)}^{\perp}(\Delta) &= ( \re T_0(z) )^{-1} E_{\re(H)}^{\perp}(\Delta) \left( T_{\epsilon}^{\pm}(z)^* +  i\im ( T_0 (z) )+ \epsilon R_{\epsilon}^{\pm *} \right),
\end{split}
\end{equation}
and deduce from the triangular inequality that
\begin{equation}
\begin{split}
\| E_{\re(H)}^{\perp}(\Delta) \varphi \|^2 &\leq 3\delta_0^{-2} ( \| T_{\epsilon}^{\pm}(z) \varphi \|^2 + \big\| \im ( T_0 (z) ) \varphi \big\|^2 + 
C_R^2 \epsilon^2 \| \varphi \|^2 ), \\
\| E_{\re(H)}^{\perp}(\Delta) \varphi \|^2 &\leq 3\delta_0^{-2} ( \big\| T_{\epsilon}^{\pm}(z)^* \varphi \|^2 + \big\| \im ( T_0 (z) ) \varphi \big\|^2 + 
C_R^2 \epsilon^2 \| \varphi \|^2 ).
\end{split}
\end{equation}
The conclusion follows now from Lemma \ref{ImT0}.
\end{proof}

From now on, we fix
\begin{equation}\label{pchoice}
p_{\pm} := \frac{a_{\pm} \delta_0^2}{6(a_{\pm}+b_{\pm})\sigma_0} .
\end{equation}
In view of Hypotheses (A2), (A3) and (M), we can also assume without any restriction that $\epsilon_0 >0$ is chosen such that:

\medskip
\noindent {\bf Assumption (A7)} 
\begin{equation*}
\beta_{\pm} \epsilon_0 \leq 1 \quad \text{and} \quad \frac{3(a_{\pm}+b_{\pm})}{\delta_0^2} \left( (C_R \epsilon_0)^2 +\sigma_0 C_R \epsilon_0 \right) +
\sup_{\epsilon \in [0,\epsilon_0)} \big\| \im(Q_{\epsilon}^{\pm}) \big\|\leq \frac{a_{\pm}}{4}.
\end{equation*}
Assumption (A7) and (\ref{pchoice}) ensure that
\begin{equation*}
\frac{3(a_{\pm}+b_{\pm})}{\delta_0^2}\sup_{\epsilon \in [0,\epsilon_0)} \left( (C_R \epsilon_0)^2 +\sigma_0 C_R \epsilon_0 + \frac{\sigma_0 p_{\pm}}{2}\right) + \sup_{\epsilon \in [0,\epsilon_0)} \big\| \im(Q_{\epsilon}^{\pm}) \big\|\leq \frac{a_{\pm}}{2}.
\end{equation*}
For $z\in {\mathbb C}$, we have 
\begin{equation}
\im ( T_{\epsilon}^{\pm} (z) ) = \im(z) - \im ( S(\epsilon) ) \pm \epsilon \re ( B(\epsilon) ) = \im( T_0(z) ) \pm \epsilon i\mathrm{ad}_A ( \re(H) ) -\epsilon \im(Q_{\epsilon}^{\pm}) ,
\end{equation}
with $\im ( T_0(z) ) = (\im z - \sigma_{\pm}) +(\sigma_{\pm} -\im (H))$. Assuming (M$_+$), we get for $z\in \Delta + i[\sigma_+,\infty),$
\begin{align}\label{mou+}
\im ( T_{\epsilon}^+ (z) ) &= -\im ( T_{\epsilon}^+ (z)^* ) \\
&\geq ( \im z - \sigma_+ ) + a_+ \epsilon - (a_+ +b_+) \epsilon E_{\re(H)}^\perp (\Delta) +(1-\epsilon \beta_+) (\sigma_+ - \im (H)) - \epsilon \im(Q_{\epsilon}^+) \nonumber
\end{align}
Assuming (M$_-$), we get for $z\in \Delta + i(-\infty, \sigma_-],$
\begin{align}\label{mou-}
- \im ( T_{\epsilon}^-(z) ) &= \im ( T_{\epsilon}^- (z)^* ) \\
& \geq ( \sigma_- - \im z ) + a_- \epsilon - (a_- +b_-) \epsilon E_{\re(H)}^\bot(\Delta) +(1- \epsilon \beta_-) (\im (H) -\sigma_-) + \epsilon \im(Q_{\epsilon}^-) \nonumber
\end{align}
This leads us to:

\begin{prop}\label{+} 
Assume that (A1)--(A5) and (A7) hold.
\begin{itemize}
\item If (M$_+$) holds, then there exists $C_1^+ >0$ such that for any $\epsilon \in [0,\epsilon_0)$, any $\varphi \in {\mathscr H}$ and any $z\in \bar{\Omega}^+ (\Delta_0),$
\begin{gather}
\im \bra \varphi, T_{\epsilon}^+ (z) \varphi \ket + C_1^+ \epsilon \| T_{\epsilon}^+ (z) \varphi \|^2 \geq d_+ ( \im(z),\epsilon ) \|\varphi \|^2 \label{mou+1},\\
- \im \bra \varphi, T_{\epsilon}^{+} (z)^* \varphi \ket + C_1^+ \epsilon \| T_{\epsilon}^+ (z)^* \varphi \|^2 \geq d_+ ( \im(z),\epsilon ) \|\varphi \|^2 \label{mou+2},
\end{gather}
with $d_+ ( \im(z),\epsilon) := ( \im(z) - \sigma_+) + \epsilon \alpha_+$ and $\alpha_+ =a_+/2$.
\item If (M$_-$) holds, then there exists $C_1^- >0$ such that for any $\epsilon \in [0,\epsilon_0)$, any $\varphi \in {\mathscr H}$ and any $z\in \bar{\Omega}^- (\Delta_0)$,
\begin{gather}
- \im \bra \varphi,T_{\epsilon}^- (z) \varphi \ket + C_1^- \epsilon \| T_{\epsilon}^- (z) \varphi \|^2 \geq d_- ( \im(z),\epsilon ) \|\varphi \|^2 \label{mou-1},\\
\im \bra \varphi,T_{\epsilon}^- (z)^* \varphi \ket + C_1^- \epsilon \| T_{\epsilon}^- (z)^* \varphi \|^2 \geq d_- ( \im(z),\epsilon ) \|\varphi \|^2 \label{mou-2},
\end{gather}
with $d_- ( \im(z),\epsilon) := -( \im(z) - \sigma_-) + \epsilon \alpha_-$ and $\alpha_- =a_-/2$.
\end{itemize}
In particular, $T_{\epsilon}^{\pm} (z)$ is boundedly invertible as soon as $\epsilon \in [0,\epsilon_0)$, $z\in \bar{\Omega}^{\pm} (\Delta_0)$ and $d_{\pm} ( \im(z),\epsilon ) >0$.
\end{prop}

\begin{rem}\label{++} The constants $C_1^{\pm}$ can be explicited:
\begin{equation}\label{C1}
C_1^{\pm} := \frac{3(a_{\pm}+b_{\pm})}{\delta_0^2} \Bigg( 1+\frac{\sigma_0}{2p_{\pm}}\Bigg) .
\end{equation}
\end{rem}

\noindent
\begin{proof}
Due to Assumption (A7), $\pm (1-\epsilon \beta_{\pm}) (\sigma_{\pm} - \im (H)) \geq 0$, for any $\epsilon \in (0,\epsilon_0]$. Inequalities (\ref{mou+1}) and (\ref{mou+2}) follow from (\ref{mou+}), Lemmata \ref{ImT0} and \ref{perp}, while inequalities (\ref{mou-1}) and (\ref{mou-2}) follows from (\ref{mou-}), Lemmata \ref{ImT0} and \ref{perp}. 

Fix $\epsilon \in [0,\epsilon_0)$, $z\in \bar{\Omega}^+ (\Delta_0)$ with $d_+ ( \im(z),\epsilon ) >0$. Inequalities (\ref{mou+1}) and (\ref{mou+2}) show that
$T_{\epsilon}^+ (z)$ and $T_{\epsilon}^+ (z)^*$ are injective from ${\mathscr H}$ into itself and have closed ranges. Since
\begin{equation*}
\overline{\mathrm{Ran} \, T_{\epsilon}^+ (z) } = (\, \mathrm{Ker} \, T_{\epsilon}^+ (z)^* )^{\perp},
\end{equation*}
we deduce that $T_{\epsilon}^+ (z)$ and $(T_{\epsilon}^+ (z))^*$ are actually bijective and boundedly invertible. Similarly, fix $\epsilon \in [0,\epsilon_0)$, 
$z\in \bar{\Omega}^- (\Delta_0)$ with $d_- ( \im(z),\epsilon ) >0$. Inequalities (\ref{mou-1}) and (\ref{mou-2}) show $T_{\epsilon}^- (z)$ and $T_{\epsilon}^- (z)^*$ are 
injective from ${\mathscr H}$ into itself and have closed ranges. Since
\begin{equation*}
\overline{\mathrm{Ran} \, T_{\epsilon}^- (z) } = (\, \mathrm{Ker} \, T_{\epsilon}^- (z)^* )^{\perp},
\end{equation*}
we deduce that $T_{\epsilon}^- (z)$ and $(T_{\epsilon}^- (z))^*$ are actually bijective and boundedly invertible.
\end{proof}

In view of Proposition \ref{+}, we define for any $\epsilon \in [0,\epsilon_0)$, $z\in \bar{\Omega}^{\pm} (\Delta_0)$ with $d_{\pm} ( \im(z),\epsilon ) >0$,
\begin{equation}\label{G}
G_{\epsilon}^{\pm}(z) := ( T_{\epsilon}^{\pm} (z) )^{-1}.
\end{equation}
As a direct consequence of Proposition \ref{+}, we obtain

\begin{prop}\label{firstbounds} 
Assume that (A1)--(A5), (M$_{\pm}$) and (A7) hold. First, we have
\begin{equation}
C_0^{\pm}:= \sup_{(\epsilon,z) \in [0,\epsilon_0) \times \bar{\Omega}^{\pm} (\Delta_0), d_{\pm}(\im(z), \epsilon) >0} d_{\pm} ( \im(z), \epsilon ) 
\| G_{\epsilon}^{\pm}(z) \| < \infty.
\end{equation}
In addition, for $C_1^{\pm}$ defined in Proposition \ref{+} and Remark \ref{++}, we have:
\begin{itemize}
\item for any $\epsilon \in (0,\epsilon_0)$, $z\in \bar{\Omega}^+ (\Delta_0)$,
\begin{align}
\|G_{\epsilon}^+(z)\varphi \| &\leq \sqrt{\dfrac{2C_1^+}{a_+}} \| \varphi \|+ \sqrt{\frac{2 \big| \im \bra \varphi,G_{\epsilon}^{+}(z) \varphi \ket  \big|}{\epsilon a_+}},
\label{G-IMG-1}\\
\|G_{\epsilon}^{+} (z)^* \varphi \| &\leq \sqrt{\frac{2 C_1^+}{a_+}} \| \varphi \|+ \sqrt{\frac{2 \big| \im \bra \varphi, G_{\epsilon}^+(z)\varphi \ket \big|}{\epsilon a_+}},
\label{G-IMG-2}
\end{align}
\item for any $\epsilon \in (0,\epsilon_0)$, $z\in \bar{\Omega}^- (\Delta_0)$,
\begin{align}
\|G_{\epsilon}^-(z)\varphi \| &\leq \sqrt{\frac{2 C_1^-}{a_-}} \| \varphi \|+ \sqrt{\frac{2 \big| \im \bra \varphi,G_{\epsilon}^-(z) \varphi \ket \big|}{\epsilon a_-}}, 
\label{G-IMG-3}\\
\|G_{\epsilon}^{-} (z)^* \varphi \| &\leq \sqrt{\frac{2 C_1^-}{a_-}} \| \varphi \|+ \sqrt{\frac{2 \big| \im \bra \varphi, G_{\epsilon}^-(z)\varphi \ket \big|}{\epsilon a_-}}. 
\label{G-IMG-4}
\end{align}
\end{itemize}
\end{prop}

\noindent
\begin{proof} 
Observe first that for any $\epsilon \in [0,\epsilon_0)$, $z\in \bar{\Omega}^{\pm} (\Delta_0)$ with $d_{\pm} ( \im(z),\epsilon ) >0$,
\begin{equation}
G_{\epsilon}^{\pm} (z)^\ast \im ( T_{\epsilon}^{\pm} (z) ) G_{\epsilon}^{\pm} (z) =- \im ( G_{\epsilon}^{\pm} (z) ).
\end{equation}
Once replaced $\varphi$ by $G_{\epsilon}^+(z) \varphi$ in (\ref{mou+1}), resp. by $G_{\epsilon}^-(z) \varphi$ in (\ref{mou-1}), Proposition \ref{+} reads:
for any $\epsilon \in [0,\epsilon_0)$, $z\in \bar{\Omega}^{\pm} (\Delta_0)$, $d_{\pm} ( \im(z),\epsilon ) >0$,
\begin{equation}
\big| \im \bra G_{\epsilon}^{\pm}(z) \varphi, \varphi \ket \big| + C_1^{\pm} \epsilon \| \varphi \|^2 \geq d_{\pm} ( \im(z),\epsilon ) \| G_{\epsilon}^{\pm}(z) \varphi \|^2 \label{mou+10}.
\end{equation}
But, for all $\epsilon \in [0,\epsilon_0)$, $z\in \bar{\Omega}^{\pm} (\Delta_0)$ with $d_{\pm} ( \im(z),\epsilon ) >0$, we have
\begin{equation}\label{CS}
\big| \im \bra G_{\epsilon}^{\pm}(z) \varphi, \varphi \ket \big| \leq | \bra G_{\epsilon}^{\pm}(z) \varphi, \varphi \ket | \leq \|\varphi \| \| G_{\epsilon}^{\pm}(z) \varphi \|.
\end{equation}
Multiplying both sides by $d_{\pm} ( \im(z),\epsilon )$ and using the fact that $d_{\pm} ( \im(z),\epsilon ) \leq 1 + \epsilon_0 \alpha_{\pm}$ if $z\in \bar{\Omega}^{\pm} (\Delta_0)$, 
we get
\begin{equation*}
d_{\pm} ( \im(z),\epsilon ) \| G_{\epsilon}^{\pm}(z) \varphi \| \|\varphi \| + C_1^{\pm} (1 + \epsilon_0 \alpha_{\pm})\epsilon_0 \| \varphi \|^2 \geq d_{\pm} ( \im(z),\epsilon )^2 
\| G_{\epsilon}^{\pm}(z) \varphi \|^2.
\end{equation*}
Taking supremum over the vectors $\varphi$ of norm 1, we deduce that necessarily $C_0^{\pm} < \infty$.

For any $\epsilon \in (0,\epsilon_0)$, $z\in \bar{\Omega}^{\pm} (\Delta_0)$, inequality (\ref{mou+10}) implies that
\begin{equation*}
\big| \im \bra G_{\epsilon}^{\pm}(z) \varphi, \varphi \ket \big| + C_1^{\pm} \epsilon \| \varphi \|^2 \geq \epsilon \alpha_{\pm} \| G_{\epsilon}^{\pm}(z) \varphi \|^2,
\end{equation*}
which implies estimates (\ref{G-IMG-1}) and (\ref{G-IMG-3}).

Note that $\| G_{\epsilon}^{\pm}(z)^* \| = \| G_{\epsilon}^{\pm}(z) \|$ whenever it exists. Once replaced $\varphi$ by $G_{\epsilon}^+(z)^* \varphi$ in (\ref{mou+2}), 
resp. by $G_{\epsilon}^-(z)^* \varphi$ in (\ref{mou-2}), Proposition \ref{+} reads: for any $\epsilon \in [0,\epsilon_0)$, $z\in \bar{\Omega}^{\pm} (\Delta_0)$, $d_{\pm} ( \im(z),\epsilon ) >0$,
\begin{equation*}
\big| \im \bra G_{\epsilon}^{\pm}(z)^* \varphi, \varphi \ket \big| + C_1^{\pm} \epsilon \| \varphi \|^2 \geq \epsilon \alpha_{\pm} \| G_{\epsilon}^{\pm}(z)^* \varphi \|^2.
\end{equation*}
Estimates (\ref{G-IMG-2}) and (\ref{G-IMG-4}) follow at once. 
\end{proof}

We recall that for $z\in {\mathbb C}\setminus {\mathcal N}(H)$, we have $\| (z-H)^{-1} \| \leq 1/\text{dist} (z, {\mathcal N}(H))$. In particular, if 
$z\in \Omega^{\pm} (\Delta_0)$, $\| (z-H)^{-1} \| \leq | \im(z) -\sigma_{\pm} |^{-1}$. The following 
estimates follows directly from Proposition \ref{firstbounds} and Lemma \ref{Tepsilon-T0}:

\begin{cor}\label{cvto0} 
Assume that (A1)--(A5), (M$_{\pm}$) and (A7) hold. For any $\epsilon \in [0,\epsilon_0)$ and any $z\in \Omega^{\pm} (\Delta_0)$, we have 
\begin{equation*}
\| G_{\epsilon}^{\pm}(z) - (z-H)^{-1} \| \leq \min \Bigg( \frac{C_R C_0^{\pm} \epsilon}{\big| \im(z) -\sigma_{\pm} \big|^2},\frac{2 C_R C_0^{\pm}}{a \big| \im(z) -\sigma_{\pm} \big|} \Bigg).
\end{equation*}
\end{cor}

\noindent
\begin{proof} 
For any $\epsilon \in [0,\epsilon_0)$ and any $z\in \Omega^+ (\Delta_0)$, 
$G_{\epsilon}^+(z) - (z-H)^{-1} = G_{\epsilon}^+(z) ( T_0 (z) - T_{\epsilon}^+(z) ) (z-H)^{-1}$. We obtain the first estimate by using Lemma \ref{Tepsilon-T0} 
and Proposition \ref{firstbounds}. 

The proof of the second estimate is analogous.
\end{proof}

Let $1/2 < s \leq 1$. For any $\epsilon \in (0,\epsilon_0)$, $z\in \bar{\Omega}^{\pm} (\Delta_0)$, we define
\begin{equation*}
F_{s, \epsilon}^{\pm} (z) := W_s (\epsilon ) G_{\epsilon}^{\pm} (z) W_s (\epsilon ),
\end{equation*}
where $W_s (\epsilon ) := (|A|+1)^{-s} (\epsilon |A|+1)^{s -1}$. Note that for any $s\in (1/2,1]$, $(W_s (\epsilon ))_{\epsilon \in [0,\epsilon_0)}$ is a family of bounded selfadjoint operators.  In particular, 
$\sup_{\epsilon \in [0,1]} \|W_s (\epsilon ) \| \leq 1$. Proposition \ref{firstbounds} entails immediately:

\begin{cor}\label{secondbounds} 
Assume that (A1)--(A5), (M$_{\pm}$) and (A7) hold. We have
$$
\sup_{\epsilon \in (0,\epsilon_0), z\in \bar{\Omega}^{\pm} (\Delta_0)} \epsilon \| F_{s, \epsilon}^{\pm} (z) \| \leq \frac{2 C_0^{\pm}}{a_{\pm}} < \infty.
$$
In addition, for any $\epsilon \in (0,\epsilon_0)$, $z\in \bar{\Omega}^{\pm} (\Delta_0)$, we have
\begin{eqnarray*}
\|G_{\epsilon}^{\pm}(z)W_s (\epsilon ) \| &\leq \sqrt{\dfrac{2 C_1^{\pm}}{a_{\pm}}} + \sqrt{\dfrac{2 \| F_{s, \epsilon}^{\pm}(z)\|}{\epsilon a_{\pm}}}, \\
\|W_s (\epsilon) G_{\epsilon}^{\pm} (z) \|=\|(G_{\epsilon}^{\pm} (z))^* W_s (\epsilon ) \| &\leq \sqrt{\dfrac{2 C_1^{\pm}}{a_{\pm}}} +  \sqrt{\dfrac{2 \| F_{s, \epsilon}^{\pm}(z)\|}{\epsilon a_{\pm}}}.
\end{eqnarray*}
\end{cor}

\subsection{Differential Inequalities}

Next, we derive a system of differential inequalities for the weighted deformed resolvents $F_{s, \epsilon}^{\pm}(z)$. We recall the following basic facts and refer 
to \cite[Lemma 7.3.4]{abmg} for proofs.

\begin{lem} 
Assume that (A1)--(A5), (M$_{\pm}$) and (A7) hold. Then, for any fixed $z\in \bar{\Omega}^{\pm} (\Delta_0)$, the map $\epsilon\mapsto G_{\epsilon}^{\pm} (z)$ is continuous on the interval $[0,\epsilon_0)$ w.r.t. the operator norm topology, continuously differentiable on $(0, \epsilon_0)$ w.r.t. the operator norm topology, with
\begin{equation*}
\begin{split}
\partial_{\epsilon} G_{\epsilon}^+ (z) &= G_{\epsilon}^+(z) ( \partial_{\epsilon} S(\epsilon) -i \epsilon \partial_{\epsilon} B(\epsilon) -i B(\epsilon) ) G_{\epsilon}^+(z), \\
\partial_{\epsilon} G_{\epsilon}^- (z) &= G_{\epsilon}^-(z) ( \partial_{\epsilon} S(\epsilon) +i \epsilon \partial_{\epsilon} B(\epsilon) +i B(\epsilon) ) G_{\epsilon}^-(z) \, .
\end{split}
\end{equation*}
\end{lem}

\begin{lem} 
Assume that (A1)--(A7), (M$_{\pm}$) hold. Then, for any $\epsilon \in [0,\epsilon_0)$, $z\in \bar{\Omega}^{\pm} (\Delta_0)$, we have $G_{\epsilon}^{\pm} (z) \in C^1(A)$, with
\begin{equation*}
\begin{split}
\mathrm{ad}_{A} G_{\epsilon}^+(z) &= G_{\epsilon}^+(z) ( \mathrm{ad}_A ( S(\epsilon) ) - i\epsilon \mathrm{ad}_A ( B(\epsilon) ) ) G_{\epsilon}^+(z), \\
\mathrm{ad}_{A} G_{\epsilon}^-(z) &= G_{\epsilon}^-(z) ( \mathrm{ad}_A ( S(\epsilon) ) + i\epsilon \mathrm{ad}_A ( B(\epsilon) ) ) G_{\epsilon}^-(z). 
\end{split}
\end{equation*}
\end{lem}

Summing up, we obtain:

\begin{prop}\label{partial=ad} 
Assume that (A1)--(A7), (M$_{\pm}$) hold. Then, for any $\epsilon \in (0,\epsilon_0)$, $z\in \bar{\Omega}^{\pm} (\Delta_0)$, the map $\epsilon\mapsto G_{\epsilon}^{\pm} (z)$ is continuously 
differentiable on $(0, \epsilon_0)$ w.r.t. the operator norm topology with
\begin{align*}
\partial_{\epsilon} G_{\epsilon}^+ (z) &= \mathrm{ad}_A ( G_{\epsilon}^+ (z) ) + G_{\epsilon}^+ (z) {\mathcal Q}^+ (\epsilon) G_{\epsilon}^+ (z),  \\
\partial_{\epsilon} G_{\epsilon}^- (z) &= -\mathrm{ad}_A ( G_{\epsilon}^- (z) ) + G_{\epsilon}^- (z) {\mathcal Q}^- (\epsilon) G_{\epsilon}^- (z), 
\end{align*}
where
\begin{align*}
{\mathcal Q}^+ &= \partial_{\epsilon} S - i\epsilon \partial_{\epsilon} B - iB - \mathrm{ad}_A S + i\epsilon \mathrm{ad}_A B, \\
{\mathcal Q}^- &= \partial_{\epsilon} S + i\epsilon \partial_{\epsilon} B + iB + \mathrm{ad}_A S + i\epsilon \mathrm{ad}_A B.
\end{align*}
\end{prop}
This is the first key to state some differential inequalities. Now, for any fixed $1/2 < s < 1$, the map $\epsilon \mapsto W_s (\epsilon)$ is strongly continuous on $[0,\epsilon_0)$ and converges strongly to $(|A|+1)^{-s}$ as $\epsilon$ tends to zero. Let us introduce for any $\epsilon \in (0,\epsilon_0),$
\begin{equation}
q^{\pm} (\epsilon ) := \epsilon^{-1} \| {\mathcal Q}^{\pm} (\epsilon) \| \, .
\end{equation}

We have:

\begin{prop}\label{diffins<1} 
Assume that (A1)--(A7), (M$_{\pm}$) hold. Let $1/2 < s <1$. For any fixed $z\in \bar{\Omega}^{\pm} (\Delta_0)$, the map $\epsilon\mapsto F_{s,\epsilon}^{\pm}(z)$ is weakly continuously 
differentiable on $(0, \epsilon_0)$ and for any $\varphi \in {\mathscr H}$, any $\epsilon \in (0,\epsilon_0)$,
\begin{align}\label{diffin.s<1}
\left| \bra \varphi, \partial_{\epsilon} F_{s, \epsilon}^{\pm}(z) \varphi \ket \right | \leq h_{0,s}^{\pm} ( \epsilon ) \|\varphi \|^2 + h_{1,s}^{\pm} ( \epsilon ) \|\varphi \| \sqrt{| \bra \varphi, F_{s, \epsilon}^{\pm}(z)\varphi \ket |} + h_2^{\pm} ( \epsilon ) | \bra \varphi, F_{s, \epsilon}^{\pm}(z)\varphi \ket |.
\end{align}
where the functions $h_{0,s}^{\pm}$, $h_{1,s}^{\pm}$ and $h_2^{\pm}$ are respectively defined for $s\geq 0$ and $\epsilon \in (0,\infty)$ by:
\begin{align*}
h_{0,s}^{\pm} ( \epsilon ) &= 2(2-s) \sqrt{\frac{2C_1^{\pm}}{a_{\pm}}} \, \epsilon^{s-1} +4 \epsilon q^{\pm} ( \epsilon ) \frac{C_1^{\pm}}{a_{\pm}} \\
h_{1,s}^{\pm} ( \epsilon ) &= 2(2-s) \sqrt{\frac{2}{a_{\pm}}} \, \epsilon^{s-3/2} \quad \text{and} \quad h_2^{\pm} ( \epsilon ) = \frac{4q^{\pm} ( \epsilon )}{a_{\pm}} . 
\end{align*}
\end{prop}

\noindent
\begin{proof} For $1/2 < s < 1$, the map $\epsilon \mapsto W_s (\epsilon )$ is strongly continuously differentiable on the interval $(0,\epsilon_0)$. In addition, we have for any $\epsilon \in (0,\epsilon_0)$ and any $\varphi \in {\mathscr H},$
\begin{align}
\| \partial_{\epsilon} W_s (\epsilon ) \varphi \| &\leq (1-s) \epsilon^{s -1} \| \varphi \| \label{we} \\
\| AW_s (\epsilon ) \| &= \| W_s (\epsilon )A \| \leq \epsilon^{s-1} . \label{awe}
\end{align}
Fix $z\in \bar{\Omega}^{\pm} (\Delta_0)$. Due to Proposition \ref{partial=ad}, the map $\epsilon\mapsto F_{s,\epsilon}^{\pm} (z)$ is weakly 
continuously differentiable on $(0, \epsilon_0)$. We have that for any $\epsilon \in (0,\epsilon_0)$,
\begin{equation}\label{eqnum0}
\begin{split}
|\bra \varphi, \partial_{\epsilon} F_{s, \epsilon}^{\pm} (z) \varphi \ket | &\leq |\bra (\partial_{\epsilon} W_s (\epsilon ))\varphi, G_{\epsilon}^{\pm} (z) W_s (\epsilon) \varphi \ket 
+ \bra W_s (\epsilon) \varphi, G_{\epsilon}^{\pm} (z) (\partial_{\epsilon} W_s (\epsilon )) \varphi \ket | \\
&+ | \bra W_s (\epsilon) \varphi, (\partial_{\epsilon} G_{\epsilon}^{\pm} (z)) W_s (\epsilon) \varphi \ket |.
\end{split}
\end{equation}
Cauchy-Schwarz inequality followed by Proposition \ref{firstbounds} and (\ref{we}) allow to bound the first term on the r.h.s. of (\ref{eqnum0}) by:
\begin{align*}
\| (\partial_{\epsilon} W_s (\epsilon ))\varphi \| &\left( \| G_{\epsilon}^{\pm} (z) W_s (\epsilon) \varphi \| + \| G_{\epsilon}^{\pm} (z)^* W_s (\epsilon) \varphi \| \right) \\
& \leq 2(1-s) \epsilon^{s-1} \| \varphi \| \left( \sqrt{\frac{2 C_1^{\pm}}{a_{\pm}}} \| \varphi \|+ \sqrt{\frac{2 |\bra \varphi, F_{s, \epsilon}^{\pm} (z)\varphi \ket |}{\epsilon a_{\pm}}} \right) ,
\end{align*}
Due to Proposition \ref{partial=ad}, the second term on the r.h.s. of \eqref{eqnum0} is bounded by
$$
\big| \bra W_s (\epsilon) \varphi, \mathrm{ad}_A ( G_{\epsilon}^{\pm} (z) ) W_s (\epsilon) \varphi \ket \big| +  |\bra G_{\epsilon}^{\pm} (z)^* W_s (\epsilon) \varphi, {\mathcal Q}^{\pm} (\epsilon) G_{\epsilon}^{\pm} (z) W_s (\epsilon) \varphi \ket | .
$$
Cauchy Schwarz inequality, Proposition \ref{firstbounds} and \eqref{awe} yield:
\begin{align*}
\big| \bra W_s (\epsilon) \varphi, \mathrm{ad}_A ( G_{\epsilon}^{\pm} (z) ) W_s (\epsilon) \varphi \ket \big| & \leq \| A W_s(\epsilon ) \varphi \| \left( \| G_{\epsilon}^{\pm} (z) W_s (\epsilon) \varphi \| + \| G_{\epsilon}^{\pm} (z)^* W_s (\epsilon) \varphi \| \right) , \\
&\leq 2 \epsilon^{s-1} \| \varphi \| \left(\sqrt{\frac{2 C_1^{\pm}}{a_{\pm}}} \|\varphi \|+ \sqrt{\frac{2 | \bra \varphi, F_{s, \epsilon}^{\pm} (z) \varphi \ket |}{\epsilon a_{\pm}}} \right) ,
\end{align*}
while
\begin{align*}
|\bra G_{\epsilon}^{\pm} (z)^* W_s (\epsilon) \varphi, {\mathcal Q}^{\pm} (\epsilon) G_{\epsilon}^{\pm} (z) W_s (\epsilon) \varphi \ket | & \leq \epsilon q^{\pm} (\epsilon ) \| G_{\epsilon}^{\pm} (z)^* W_s (\epsilon) \varphi \| \| G_{\epsilon}^{\pm} (z) W_s (\epsilon) \varphi \| \\
&\leq 4 \epsilon q^{\pm} (\epsilon ) \left( \frac{C_1^{\pm}}{a_{\pm}} \|\varphi \|^2 + \frac{| \bra \varphi, F_{s, \epsilon}^{\pm} (z) \varphi \ket |}{\epsilon a_{\pm}} \right) .
\end{align*}
This completes the proof.
\end{proof}

Note that if $s=1$, stronger conclusions hold:

\begin{prop}\label{diffins=1} 
Assume that (A1)--(A7), (M$_{\pm}$) hold. For any fixed $z\in \bar{\Omega}^{\pm} (\Delta_0)$, the map $\epsilon\mapsto F_{1,\epsilon}^{\pm}(z)$ is continuously differentiable on $(0, \epsilon_0)$ 
w.r.t. the operator norm topology, and for any $\epsilon \in (0,\epsilon_0)$,
\begin{equation}\label{diffin.s=1}
\left\| \partial_{\epsilon} F_{1, \epsilon}^{\pm} (z) \right\| \leq h_{0,1}^{\pm} ( \epsilon ) + h_{1,1}^{\pm} ( \epsilon ) \sqrt{ \| F_{1, \epsilon}^{\pm} (z)\| } + h_2^{\pm} ( \epsilon ) \| F_{1, \epsilon}^{\pm} (z)\| ,
\end{equation}
where $h_{0,1}^{\pm}$, $h_{1,1}^{\pm}$ and $h_2^{\pm}$ are defined in Proposition \ref{diffins<1} (with $s=1$).
\end{prop}

\noindent
\begin{proof} 
Fix $z\in \bar{\Omega}^{\pm} (\Delta_0)$. Due to Proposition \ref{partial=ad}, the map $\epsilon\mapsto F_{1,\epsilon}^{\pm} (z)$ is continuously 
differentiable on $(0, \epsilon_0)$ w.r.t. the operator norm topology. We have that for any $\epsilon \in (0,\epsilon_0),$
\begin{equation*}
\| \partial_{\epsilon} F_{1, \epsilon}^{\pm} (z) \| = \| (|A|+1)^{-1} (\partial_{\epsilon} G_{\epsilon}^{\pm} (z)) (|A|+1)^{-1} \|,
\end{equation*}
and we deduce that:
\begin{equation*}
\| \partial_{\epsilon} F_{1, \epsilon}^{\pm} (z) \| \leq \| (|A|+1)^{-1} \mathrm{ad}_A ( G_{\epsilon}^{\pm} (z) ) (|A|+1)^{-1} \| + \| (|A|+1)^{-1} G_{\epsilon}^{\pm} (z) {\mathcal Q}^{\pm} (\epsilon) G_{\epsilon}^{\pm} (z) (|A|+1)^{-1} \| .
\end{equation*}
Note that $\| A(|A|+1)^{-1} \| = \| (|A|+1)^{-1}A \| \leq 1$. In view of Corollary \ref{secondbounds} (with $s=1$), we have
\begin{align*}
\| (|A|+1)^{-1} \mathrm{ad}_A ( G_{\epsilon}^{\pm} (z) ) (|A|+1)^{-1} \| &\leq \| G_{\epsilon}^{\pm} (z) (|A|+1)^{-1} \| + \| (|A|+1)^{-1} G_{\epsilon}^{\pm} (z) \| \\
&\leq 2 \left( \sqrt{\frac{2 C_1^{\pm}}{a_{\pm}}}+ \sqrt{\frac{2 \| F_{1, \epsilon}^{\pm} (z)\|}{\epsilon a_{\pm}}} \right)
\end{align*}
while
\begin{align*}
\| (|A|+1)^{-1} G_{\epsilon}^{\pm} (z) {\mathcal Q}^{\pm} (\epsilon) G_{\epsilon}^{\pm} (z) (|A|+1)^{-1} \| &\leq \epsilon q^{\pm} (\epsilon ) \| (|A|+1)^{-1} G_{\epsilon}^{\pm} (z) \| \| G_{\epsilon}^{\pm} (z) (|A|+1)^{-1} \| \\
& \leq 4 \epsilon q^{\pm} (\epsilon ) \left( \frac{C_1^{\pm}}{a_{\pm}}+ \frac{\| F_{1, \epsilon}^{\pm} (z)\|}{\epsilon a_{\pm}} \right).
\end{align*}
We deduce \eqref{diffin.s=1}.  
\end{proof}

The next result follows from Propositions \ref{diffins<1}, \ref{diffins=1} and Gronwall lemma.

\begin{prop}\label{unifbounds} 
Assume that (A1)--(A7), (M$_{\pm}$) hold. Fix $s\in (1/2,1]$. If the function $q^{\pm}$ belongs to $L^1((0,\epsilon_0))$, then for any $0 < \mu_0 < \epsilon_0,$
\begin{equation*}
\sup_{\epsilon \in (0,\mu_0), z\in \bar{\Omega}^{\pm} (\Delta_0)} \| F_{s,\epsilon}^{\pm} (z)\| < \infty. \label{ub}
\end{equation*}
\end{prop}

\noindent
\begin{proof} Assume $q^{\pm} \in L^1 ((0,\epsilon_0))$ and denote $C_q = \max \| q^{\pm} \|_{L^1 ((0,\epsilon_0))}$. Fix $\mu_0 \in (0,\epsilon_0)$. Let $s=1$. For any $z\in \bar{\Omega}^{\pm} (\Delta_0)$ and any $\epsilon \in (0,\mu_0)$,
$$
\| F_{1, \epsilon}^{\pm} (z) \| \leq \| F_{1, \mu_0}^{\pm} (z) \| + \int_{\epsilon}^{\mu_0} \| \partial_{\mu} F_{1, \mu}^{\pm} (z) \| \, d\mu ,
$$
which combined with Corollary \ref{secondbounds} and Proposition \ref{diffins=1} yields:
$$
\| F_{1, \epsilon}^{\pm} (z) \| \leq C_{2,1}^{\pm} + \int_{\epsilon}^{\mu_0} h_{1,1}^{\pm} ( \mu ) \sqrt{ \| F_{1, \mu_0}^{\pm} (z) \|} \, d\mu + \int_{\epsilon}^{\mu_0} h_{2}^{\pm} ( \mu ) \| F_{1, \mu}^{\pm} (z) \| \, d\mu
$$
where we define $C_{2,s}^{\pm}$ for $1/2 < s \leq 1$ by:
$$
C_{2,s}^{\pm} := \frac{2C_0^{\pm}}{\mu_0 a_{\pm}} + \int_0^1 h_{0,s}^{\pm} ( \mu )\, d\mu .
$$
Mind that the functions $h_{0,s}^{\pm}$ are integrable on $(0,\epsilon_0)$ by hypothesis. Using Gronwall Lemma as stated in e.g. \cite{coddlev} or \cite{abmg} Lemma 7.A.1, we deduce:
$$
\left\| F_{1, \epsilon}^{\pm} (z) \right\| \leq \left[ \sqrt{C_{2,1}^{\pm}} + \frac{1}{2} \int_{\epsilon}^{\mu_0} h_{1,1}^{\pm} ( \mu ) \exp \left( - \frac{1}{2} \int_{\mu}^{\mu_0} h_2^{\pm} (x)\, dx \right) \, d\mu \right]^2  \exp \left( \int_{\epsilon}^{\mu_0} h_2^{\pm} (\mu)\, d\mu \right) .
$$
Since the functions $h_{1,1}^{\pm}$ and $h_2^{\pm}$ are integrable on $(0,\epsilon_0)$, it follows that:
\begin{equation*}
\sup_{\epsilon \in (0,\mu_0), z\in \bar{\Omega}^{\pm} (\Delta_0)} \| F_{1,\epsilon}^{\pm} (z)\| < \infty.
\end{equation*}
If $1/2 < s < 1$, an extra step is required. For any $z\in \bar{\Omega}^{\pm} (\Delta_0)$, any $\epsilon \in (0,\mu_0)$ and any $\varphi \in {\mathcal H},$
$$
| \bra \varphi, F_{s, \epsilon}^{\pm}(z) \varphi \ket | \leq | \bra \varphi, F_{s, \mu_0}^{\pm}(z) \varphi \ket | + \int_{\epsilon}^{\mu_0} | \bra \varphi, \partial_{\mu} F_{s, \mu}^{\pm}(z) \varphi \ket |\, d\mu ,
$$
which combined with Corollary \ref{secondbounds} and Proposition \ref{diffins<1} yields:
\begin{align*}\label{int.s<1}
| \bra \varphi, F_{s, \epsilon}^{\pm}(z) \varphi \ket | \leq  C_{2,s}^{\pm} \|\varphi \|^2 + \|\varphi \| \int_{\epsilon}^{\mu_0} h_{1,s}^{\pm} ( \mu ) \sqrt{| \bra \varphi, F_{s, \mu}^{\pm}(z)\varphi \ket |}\, d\mu + \int_{\epsilon}^{\mu_0} h_2^{\pm} ( \mu ) | \bra \varphi, F_{s, \mu}^{\pm}(z)\varphi \ket |\, d\mu .
\end{align*}
Using again Gronwall Lemma, we deduce
\begin{align*}
| \bra \varphi, F_{s, \epsilon}^{\pm}(z) \varphi \ket | &\leq \left[ \sqrt{C_{2,s}^{\pm}} + \frac{1}{2} \int_{\epsilon}^{\mu_0} h_{1,s}^{\pm} (\mu) \exp \left( - \frac{1}{2} \int_{\mu}^{\mu_0} h_2^{\pm} (x) \, dx \right) \, d\mu \right]^2 \\
&\quad \times \exp \left( \int_{\epsilon}^{\mu_0} h_2^{\pm} ( \mu ) \, d\mu \right) \|\varphi \|^2 .
\end{align*}
The functions $h_{1,s}^{\pm}$ and $h_2^{\pm}$ are integrable on $(0,\epsilon_0)$, so
$$
\sup_{\| \varphi \| =1} \sup_{\epsilon \in (0,\mu_0), z\in \bar{\Omega}^{\pm} (\Delta_0)} |\bra \varphi, F_{s,\epsilon}^{\pm} (z)\varphi \ket | < \infty .
$$
Since $\| F_{s,\epsilon}^{\pm}(z) \| = \sup_{\| \varphi \| =1, \| \psi \| =1} |\bra \varphi, F_{s,\epsilon}^{\pm} (z)\psi \ket |$, Proposition \ref{unifbounds} follows by polarisation.
\end{proof}

\begin{prop}\label{mourrensa0} 
Assume that (A1)--(A7), (M$_{\pm}$) hold. Fix $s\in (1/2,1]$. If the function $q^{\pm}$ belongs to $L^1 ((0,\epsilon_0))$, then
\begin{equation*}
\sup_{z \in \Omega^{\pm} (\Delta_0)} \| (|A|+1)^{-s} (z-H)^{-1} (|A|+1)^{-s} \| < \infty.
\end{equation*}
\end{prop}

\noindent
\begin{proof} 
Assume first that $s=1$ and fix $z\in \Omega^{\pm} (\Delta_0)$. Due to Corollary \ref{cvto0}, for any 
$\epsilon \in (0,\epsilon_0)$,
\begin{equation*}
\| (|A|+1)^{-1} (z-H)^{-1} (|A|+1)^{-1} - F_{1,\epsilon}^{\pm}(z) \| \leq \frac{C_R C_0^{\pm} \epsilon}{| \im(z) -\sigma_{\pm} |^2}.
\end{equation*}
In other words, for any fixed $z\in \Omega^{\pm} (\Delta_0)$, $\lim_{\epsilon \rightarrow 0^+} 
F_{1,\epsilon}^{\pm} (z) = (|A|+1)^{-1} (z-H)^{-1} (|A|+1)^{-1}$ in the operator norm topology. By considering Proposition \ref{unifbounds} and taking the limit 
when $\epsilon$ tends to zero, we deduce that
\begin{equation*}
\sup_{z \in \Omega^{\pm} (\Delta_0)} \| (|A|+1)^{-1} (z-H)^{-1} (|A|+1)^{-1} \| < \infty.
\end{equation*}

Assume now that $1/2 < s <1$. Pick two vectors $\varphi$, $\psi$ in ${\mathscr H}$ and fix again $z\in \Omega^{\pm} (\Delta_0)$. For any $\epsilon \in (0,\epsilon_0)$,
\begin{align*}
\bra \varphi, (|A|+1)^{-s} (z-H)^{-1} (|A|+1)^{-s} \psi \ket - \bra \varphi, F_{s,\epsilon}^{\pm}(z) \psi \ket &= \bra W_s (\epsilon) \varphi, ( (z-H)^{-1} - G_{\epsilon}^{\pm} (z) ) W_s (\epsilon) \psi \ket \\
&+ \bra W_s (0) \varphi, (z-H)^{-1} ( W_s (0) -W_s (\epsilon) ) \psi \ket \\
&+ \bra ( W_s (0) - W_s (\epsilon) ) \varphi, (z-H)^{-1} W_s (\epsilon) \psi \ket.
\end{align*}
Using the facts that $\| W_s (\epsilon) \| \leq 1$ for any $\epsilon \in [0,1]$, $\| (z-H)^{-1} \| \leq 1/\text{dist} ( z, {\mathcal N}(H) ) \leq | \im(z) -\sigma_{\pm} |^{-1}$ 
and Corollary \ref{cvto0}, we deduce that
$$
| \bra \varphi, (|A|+1)^{-s} (z-H)^{-1} (|A|+1)^{-s} \psi \ket - \bra \varphi, F_{s,\epsilon}^{\pm} (z) \psi \ket |
$$
is bounded by
$$
\frac{C_R C_0^{\pm} \epsilon}{| \im(z) -\sigma_{\pm} |^2} \|\varphi \| \|\psi \| +\frac{1}{| \im(z) -\sigma_{\pm} |} \left( \| (W_s (0)-W_s ( \epsilon))\varphi \| \|\psi \| + \|\varphi \| \| (W_s (0)-W_s (\epsilon))\psi \| \right).
$$
In other words, for any fixed $z\in \Omega^{\pm} (\Delta_0)$, $w- \lim_{\epsilon \rightarrow 0^+} 
F_{s,\epsilon}^{\pm}(z) = (|A|+1)^{-s} (z-H)^{-1} (|A|+1)^{-s}$. Since $\| (|A|+1)^{-s} (z-H)^{-1} (|A|+1)^{-s} \| = \sup_{\| \varphi \| =1, \| \psi \| =1} |\bra \varphi, (|A|+1)^{-s} (z-H)^{-1} (|A|+1)^{-s} \psi \ket |$, we deduce from Proposition \ref{unifbounds} that:
\begin{equation*}
\sup_{z \in \Omega^{\pm} (\Delta_0)} \| (|A|+1)^{-s} (z-H)^{-1} (|A|+1)^{-s} \| < \infty.
\end{equation*}
\end{proof}

In order to conclude on the proof of Proposition \ref{mourrensa}, it remains to show that if $H\in {\mathcal C}^{1,1}(A)$, then it satisfies the hypotheses (A1)--(A6) (hence (A7)) 
and the integrability of the functions $q^{\pm}$. This is the purpose of the next section.

\subsection{Last step to the proof of Proposition \ref{mourrensa}}\label{linktoc11}

The next result is quoted from \cite[Lemma 7.3.6]{abmg}. 

\begin{lem}\label{c11approx} 
Let $S \in {\mathcal B}(\mathscr H)$ be a bounded selfadjoint operator of class ${\mathcal C}^{1,1}(A)$. Then, there exists a family $(S(\epsilon))_{\epsilon \in (0,1)}$ 
of bounded operators with the following properties:
\begin{itemize}
\item The map $\epsilon \mapsto S(\epsilon)$ with values in ${\mathcal B}(\mathscr H)$ is of class $C^{\infty}$, $\| S(\epsilon)-S \| \leq C\epsilon$ for some $C>0$ and
$$\int_0^1 \frac{\| \partial_{\epsilon} S(\epsilon) \|}{\epsilon} \, d\epsilon < \infty.$$
\item For any $\epsilon \in (0,1)$, $S(\epsilon) \in C^{\infty} (A)$ and
$$\int_0^1 \big\| \mathrm{ad}^2_A ( S(\epsilon) ) \big\| \, d\epsilon < \infty.$$
\item $S\in C^1(A)$, $\lim_{\epsilon \rightarrow 0^+} \big\| \mathrm{ad}_A ( S(\epsilon) ) - \mathrm{ad}_A (S) \big\| =0$, the map 
$\epsilon \mapsto \mathrm{ad}_A ( S(\epsilon) )$ with values in ${\mathcal B}(\mathscr H)$ is of class $C^{\infty}$ and 
$$\int_0^1 \big\| \partial_{\epsilon} \mathrm{ad}_A ( S(\epsilon) ) \big\| \, d\epsilon < \infty.$$
\end{itemize}
\end{lem}

If $H\in {\mathcal C}^{1,1}(A)$, its adjoint $H^*$, hence $\re(H)$ and $\im(H)$ also belong to ${\mathcal C}^{1,1}(A)$, see e.g. Lemma \ref{c11ReIm}. 
Applying Lemma \ref{c11approx} to both $\re(H)$ and $\im(H)$ shows that the conclusions of Lemma \ref{c11approx} are still valid if $H$ is not selfadjoint.

Given such an operator $H$, choose the function $S$ as given by Lemma \ref{c11approx} on the interval $(0,\epsilon_0)$ ($\epsilon_0 \leq 1$) and define the function $B$ by $B = i\mathrm{ad}_A(S)$. The associated functions ${\mathcal Q}^{\pm}$ are defined on $(0,1)$ by ${\mathcal Q}^{\pm} = \partial_{\epsilon} S \mp i\epsilon \partial_{\epsilon} B + i\epsilon \mathrm{ad}_A (B)$. By Lemma \ref{c11approx}, Assumptions (A1)--(A6) are verified and $q^{\pm}$ belong to $L^1((0,1))$.

Note that if $s\geq 1$, 
\begin{equation*}
\sup_{z \in \Omega^{\pm} (\Delta_0)} \| (|A|+1)^{-s} (z-H)^{-1} (|A|+1)^{-s} \| \leq \sup_{z \in \Omega^{\pm} (\Delta_0)} \| (|A|+1)^{-1} (z-H)^{-1} (|A|+1)^{-1} \|.
\end{equation*}
The proof of Theorem \ref{mourrensa} follows now directly from Remarks \ref{farfromnr}, \ref{weights} and Proposition \ref{mourrensa0}. 

\subsection{Proof of Corollary \ref{Hsmooth}}\label{proofHsmooth}

\begin{lem}\label{achia} 
Let $H\in {\mathcal B}({\mathscr H})$, $\chi \in C_0^{\infty}({\mathbb R},{\mathbb R})$ and $A$ be a selfadjoint operator acting 
on ${\mathscr H}$. If $\re(H) \in C^1 (A)$, then, for all $s\in [-1,1]$, the operator $\bra A \ket^s \chi ( \re H) \bra A\ket^{-s}$ extends as bounded 
operator on ${\mathscr H}$. 
\end{lem}
See e.g. \cite[Lemma 6.4]{royer2} for a proof.

Let us prove Corollary \ref{Hsmooth}. Without any loss of generality, we suppose that $1/2 < s \leq 1$. First, note that for any 
$z\in {\mathbb R}+i((-\infty, \sigma_- -1)\cup (\sigma_+ +1,\infty))$, $\| (z-H)^{-1}\| \leq 1/\text{dist} (z,{\mathcal N}(H))\leq 1$. Hence,
$$
\sup_{z\in {\mathbb R} +i((-\infty, \sigma_- -1)\cup (\sigma_+ +1,\infty))} \big \|\bra A\ket^{-s} \chi ( \re H) (z-H)^{-1} \chi ( \re H) 
\bra A\ket^{-s} \big \| < \infty .
$$
So, by combining Theorem \ref{mourrensa} and Lemma \ref{achia}, we get
\begin{gather*}
\sup_{z\in \Delta_1 +i(\sigma_+ ,\infty)} \big \|\bra A\ket^{-s} \chi ( \re H) (z-H)^{-1} \chi ( \re H) \bra A\ket^{-s} \big \| < \infty ,\\
\sup_{z\in \Delta_1 +i(-\infty,\sigma_-)} \big \| \bra A\ket^{-s} \chi ( \re H) (z-H)^{-1} \chi ( \re H) \bra A\ket^{-s} \big \| < \infty.
\end{gather*}
For $z \in (\R \setminus \Delta_1)+i ([\sigma_- -1,\sigma_-)\cup (\sigma_+, \sigma_+ +1])$ and any $(\varphi ,\psi)\in {\mathscr H}\times {\mathscr H}$, we can use the resolvent identity to bound $\big| \bra \chi ( \re H) \bra A\ket^{-s} \varphi, (z-H)^{-1} \chi ( \re H) \bra A\ket^{-s} \psi \ket \big|$ from above by:
\begin{align}\label{outdelta1}
\big| \bra \chi ( \re H) \bra A\ket^{-s} \varphi, & ( \re T_0(z))^{-1} \chi ( \re H) \bra A\ket^{-s} \psi \ket \big| +\\
& \big\| ( \re T_0 (z))^{-1} \chi ( \re H) \bra A\ket^{-s} \varphi \big\| \big\|  \im (T_0(z)) (z-H)^{-1} \chi ( \re H) \bra A\ket^{-s} \psi \big\|, \nonumber
\end{align}
where $C_{\chi} = \sup_{\lambda \in {\mathbb R} \setminus \Delta_1} \big\| ( \lambda - \re(H) )^{-1} \chi ( \re H ) \big\| <\infty$. 
Due to \eqref{signedqf}, for any $\psi \in {\mathscr H}$ and any $z\in \C$ such that $\im (z) \in [\sigma_- -1,\sigma_-)\cup (\sigma_+, \sigma_+ +1]$, we have
$$
\big\| \im ( T_0(z) ) \psi \big\|^2 \leq \sigma_0 \big| \bra \psi, T_0(z) \psi \ket \big|,
$$
whence
$$
\big\| \im ( T_0(z) ) (z-H)^{-1} \chi ( \re H) \bra A\ket^{-s} \psi \big\|^2 \leq \sigma_0 
\big| \bra \bra A\ket^{-s} \chi ( \re H) (z-H)^{-1} \chi ( \re H) \bra A\ket^{-s} \psi, \psi \ket \big|.
$$
Back to \eqref{outdelta1}, we deduce that
\begin{align*}
\big| \bra \chi ( \re H) &\bra A\ket^{-s} \varphi, (z-H)^{-1} \chi ( \re H) \bra A\ket^{-s} \psi \ket \big| \le \\
& C_{\chi} \|\chi \|_{\infty} \|\varphi \| \|\psi \| + C_{\chi} \sqrt{\sigma_0} \|\varphi \| \sqrt{ \big| \bra \bra A\ket^{-s} \chi ( \re H) (z-H)^{-1} \chi ( \re H) \bra A\ket^{-s} \psi,\psi \ket \big|}.
\end{align*}
This entails
\begin{eqnarray*}
\sup_{z\in ({\mathbb R}\setminus \Delta_1) +i[\sigma_- -1,\sigma_-)} \big \|\bra A\ket^{-s} \chi ( \re H) (z-H)^{-1} \chi ( \re H) \bra A\ket^{-s} \big \| &< \infty,\\
\sup_{z\in ({\mathbb R}\setminus \Delta_1) +i(\sigma_+, \sigma_+ +1]} \big \|\bra A\ket^{-s} \chi ( \re H) (z-H)^{-1} \chi ( \re H) \bra A\ket^{-s} \big \| &< \infty,
\end{eqnarray*}
which concludes the proof.

\subsection{Proof of Corollary \ref{mourrensa1}}\label{proofmourrensa1}

Since $H \in C^1(A)$, $\re(H) \in C^1(A)$. The first statement paraphrases Lemma \ref{virialnsa}. Given any relatively compact interval $\Delta_0$, $\overline{\Delta_0} \subset \Delta \setminus {\mathcal E}_{\rm p} (\re (H))$, both (M$_{\pm}$) hold on $\overline{\Delta_0}$, with $\beta_{\pm} =0$. We refer to e.g. \cite[Paragraph 7.2.2]{abmg} for details. The second statement follows then from Theorem \ref{mourrensa}.

\subsection{Proof of Theorems \ref{c11sign} and \ref{c11signweak}}\label{proofthmc11sign}

Following \cite[Section 5]{abc}, $H_0 = L \in C^{\infty} (A_0)$. Actually,
$$
i\mathrm{ad}_{A_0} (H_0) = 4H_0 -H_0^2.
$$
\paragraph{Proof of Theorem \ref{c11sign}.} If $V$ (hence $\re(V)$) belongs to $C^1 (A_0)$ (resp. ${\mathcal C}^{1,1}(A_0)$), then $H_V$ (hence $\re(H_V)$) belongs to $C^1 (A_0)$ (resp. ${\mathcal C}^{1,1}(A_0)$). In addition, if $V$ (hence $\re(V)$) and $i\mathrm{ad}_{A_0} ( \re(V) )$ belong to $\sinf(\ell^2({\mathbb Z}))$, then (M) holds with $c_{\Delta}>0$ if $\overline{\Delta} \subset (0,4)$. The conclusion follows as a particular case of Corollary \ref{mourrensa1}.

\paragraph{Proof of Theorem \ref{c11signweak}.} If $V$ belongs to ${\mathcal C}^{1,1}(A_0)$, then $H_V$ belongs to ${\mathcal C}^{1,1}(A_0)$. The conclusion follows directly from Theorem \ref{mourrensa}.


\section{Regularity classes}\label{regclass}

In this section, ${\mathscr H}$ is a Hilbert space and $A$ denotes a selfadjoint operator acting on ${\mathscr H}$, with domain ${\cal D}(A)$. 
We sum up the main properties of the regularity classes considered in this paper. We also propose some explicit criteria in the case 
${\mathscr H}= \ell^2 ({\mathbb Z})$ and $A=A_0$ defined by \eqref{A0}.

\subsection{The ${\mathcal A}(A_0)$ class}\label{se}

We refer to \cite[Chapter III]{kato} for general considerations on bounded operator-valued analytic maps.

\subsubsection{General considerations}

The classes ${\mathcal A}(A)$ and ${\mathcal A}_R(A)$, $R>0$, have been introduced in Definition \ref{dcl}. For $L\in {\mathcal A}_R(A)$, we denote the 
corresponding analytic map by
\begin{eqnarray*}
a[L] : D_R(0) & \rightarrow & {\mathcal B}({\mathscr H})\\
\theta & \mapsto & {\rm e}^{i \theta A} L {\rm e}^{-i \theta A} .
\end{eqnarray*}

The next propositions follow from direct calculations:
\begin{prop} Let $R>0$, $(L_1,L_2) \in {\mathcal A}_R(A) \times {\mathcal A}_R(A)$ and $\alpha \in {\mathbb C}$. Then,
\begin{itemize}
\item $\alpha L_1 +L_2 \in {\mathcal A}_R(A)$ and $a[\alpha L_1 +L_2]= \alpha a[L_1] + a[L_2]$,
\item $L_1L_2 \in {\mathcal A}_R(A)$ and $a[L_1 L_2]= a[L_1] a[L_2]$,
\item $I\in {\mathcal A}_R(A)$ and $a[I] \equiv I$,
\item $L_1^* \in {\mathcal A}_R(A)$ and $a[L_1^*] (\theta)= (a[L_1](\bar{\theta}))^*$, for any $\theta \in D_R(0)$.
\end{itemize}
In particular, ${\mathcal A}(A)$ is a sub-algebra of ${\cal B}({\mathscr H})$.
\end{prop}

\begin{prop} Let $R>0$ and ${\mathscr K}$ be an auxiliary Hilbert space. Let $U:{\mathscr H} \rightarrow {\mathscr K}$ be a unitary operator. Then, $L\in {\mathcal A}_R(A)$ if and only if $ULU^* \in {\mathcal A}_R(UAU^*)$. In addition, for all $\theta \in D_R(0)$, $(U(a[L])U^*) (\theta) = (a[ULU^*])(\theta)$.
\end{prop}

Next, we turn to some characterizations of the classes ${\mathcal A}_R(A)$. Let us recall the following lemma:
\begin{lem}\label{le1}
Let $\Omega \subseteq {\mathbb C} $ be an open subset and $(F_n)_{n \in {\mathbb N}}\subset {\rm Hol} ( \Omega,\cal{B}({\mathscr H}) )$. Assume that $(F_n)$ converges uniformly to a function $F_\infty$ in any compact subset included in $\Omega$. Then, $F_\infty \in {\rm Hol} ( \Omega,\cal{B}({\mathscr H}) )$.
\end{lem}
\noindent
\begin{proof}
Let $D$ be any open disk such $\overline{D} \subset \Omega$, and $T$ be a triangle in $D$. Then, Goursat's theorem implies that for any 
$n \in {\mathbb N}$, we have
$$
\int_T F_n(z) dz = 0.
$$
Since $(F_n)$ converges uniformly to $F_\infty$ in $\overline{D}$, then $F_\infty$ is continuous in $\overline{D}$ together with
$$
\int_T F_n(z) dz \underset{n \to \infty}{\longrightarrow} \int_T F_\infty(z) dz = 0.
$$
By Morera's theorem, $F_\infty \in {\rm Hol} ( D,\cal{B}({\mathscr H}) )$ and the claim follows since $D$ is arbitrary.
\end{proof}

\begin{prop}\label{pe1}
Let $\Sigma$ be a countable set and $(L_\nu)_{\nu \in \Sigma}\subset {\mathcal A}_R(A)$ for some $R > 0$. Assume that there exist $0 < R' < R$, and a 
family of finite subsets $(\Sigma_n)_{n \in {\mathbb N}}$ with 
$$
\Sigma_0 \subset \cdots \subset \Sigma_n \subset \Sigma_{n+1} \subset \cdots \subset \Sigma, \qquad \bigcup_{n \in {\mathbb N}} \Sigma_n = \Sigma,
$$
such that
\begin{equation}\label{hye1}
\sum_{n \in {\mathbb N}} \sum_{\small{\nu \in \sum_n \setminus \sum_{n-1}}} \sup_{\theta \in \overline{D_{R'}(0)}} \big\Vert {\rm e}^{i\theta A} L_\nu {\rm e}^{-i\theta A} 
\big\Vert < \infty.
\end{equation}
Then, the operator $\sum_{\nu \in \Sigma} L_\nu \in \mathcal{A}_{R'}(A)$.
\end{prop}

\noindent
\begin{proof}
By hypotheses, for any $\nu \in \Sigma$, the map $a[L_\nu] : D_{R'}(0) \longrightarrow \mathcal{B}({\mathscr H})$ belongs to ${\rm Hol} ( D_{R'}(0),\mathcal{B}({\mathscr H}) )$. Then, so are the maps
$$
a\Bigg[ \sum_{\nu \in \Sigma_n} L_\nu \Bigg] : D_{R'}(0) \longrightarrow \mathcal{B}({\mathscr H}), \quad \theta \longmapsto \sum_{\nu \in \Sigma_n} 
{\rm e}^{i\theta A} L_\nu {\rm e}^{-i\theta A},
$$
for any $n \in {\mathbb N}$. Now, hypothesis \eqref{hye1} implies that $( a\big[ \sum_{\nu \in \Sigma_n} L_\nu \big] )_{n \in {\mathbb N}}$ converges 
uniformly to $a\big[ \sum_{\nu \in \Sigma} L_\nu \big]$ on $\overline{D_{R'}(0)}$, hence on any compact subset included in $D_{R'}(0)$. The claim follows 
from Lemma \ref{le1}.
\end{proof}

A first application of Proposition \ref{pe1} is the following result:
\begin{theo}\label{te1}
Let $\Sigma = {\mathbb N}$ or ${\mathbb Z}$. Let $(\varphi_n)_{n \in \Sigma}$ and $(\psi_n)_{n \in \Sigma}$ be two sequences of analytic vectors for $A$ 
such that for any $n\in \Sigma$ the series
$$
\sum_{k = 0}^\infty \frac{\vert \theta \vert^k}{k!} \big\Vert A^k \varphi_n \big\Vert, 
\qquad \sum_{k = 0}^\infty \frac{\vert \theta \vert^k}{k!} \big\Vert A^k \psi_n \big\Vert, \qquad n \in \Sigma,
$$
converge on $D_R(0)$ for some $R > 0$. Let $(\alpha_n)_{n \in \Sigma}$ be a complex sequence.
\begin{enumerate}
\item Assume that there exists $0 < R' < R$ such that
\begin{equation}\label{hye2}
\sum_{n \in \Sigma} \vert \alpha_n \vert \sup_{\theta \in \overline{D_{R'}(0)}} \big\Vert {\rm e}^{i\theta A} \varphi_n \big\Vert 
\big\Vert {\rm e}^{i \bar \theta A} \psi_n \big\Vert < \infty.
\end{equation}
Then, the operator $\sum_{n \in \Sigma} \alpha_n \vert \varphi_n \rangle \langle \psi_n \vert$ belongs to $\mathcal{A}_{R'}(A)$ with extension given by
$$
a \Bigg[ \sum_{n \in \Sigma} \alpha_n \vert \varphi_n \rangle \langle \psi_n \vert \Bigg] (\theta) = \sum_{n \in \Sigma} \alpha_n {\rm e}^{i\theta A} \vert \varphi_n \rangle \langle \psi_n \vert {\rm e}^{-i\theta A},
$$
for all $\theta \in D_{R'}(\theta)$.
\item Assume that there exists $0 < R' < R$ such that
\begin{equation}\label{hye3}
\sum_{n \in \Sigma} \vert \alpha_n \vert \sup_{\theta \in \overline{D_{R'}(0)}} \big\Vert {\rm e}^{i \bar \theta A} \psi_n \big\Vert< \infty.
\end{equation}
Then, for any fixed $m\in \Sigma$, the operator $\sum_{n \in \Sigma} \alpha_n \vert \varphi_m \rangle \langle \psi_n \vert$ belongs to $\mathcal{A}_{R'}(A)$, with extension given by
$$
a \Bigg[ \sum_{n \in \Sigma} \alpha_n \vert \varphi_m \rangle \langle \psi_n \vert \Bigg] (\theta) = \sum_{n \in \Sigma} \alpha_n {\rm e}^{i\theta A} \vert \varphi_m \rangle \langle \psi_n \vert {\rm e}^{-i\theta A},
$$
for all $\theta \in D_{R'}(\theta)$.
\end{enumerate}
\end{theo}

\noindent
\begin{proof}
The analyticity properties of the sequences $(\varphi_n)$ and $(\psi_n)$ reads: for any $(n,m) \in \Sigma \times \Sigma$, the operator $\vert \varphi_n \rangle \langle \psi_m \vert \in {\mathcal A}_R(A)$ and the associated holomorphic map satisfies
$$
a\big[ \vert \varphi_n \rangle \langle \psi_m \vert \big] (\theta) = {\rm e}^{i\theta A} \vert \varphi_n \rangle \langle \psi_m \vert {\rm e}^{-i \theta A},
$$
for any $\theta \in D_R(0)$. In the sequel, we only give the proof of the first statement. Since
$$
\big\Vert {\rm e}^{i\theta A} \vert \varphi_n \rangle \langle \psi_n \vert {\rm e}^{-i \theta A} \big\Vert = \big\Vert \vert {\rm e}^{i\theta A} 
\varphi_n \rangle \langle {\rm e}^{i \bar \theta A} \psi_n \vert \big\Vert = \big\Vert {\rm e}^{i\theta A} \varphi_n \big\Vert \big\Vert 
{\rm e}^{i \bar \theta A} \psi_n \big\Vert,
$$ 
for all $\theta \in D_R(0)$, it follows from \eqref{hye2} that
$$
\sum_{n \in \Sigma} \vert \alpha_n \vert \sup_{\theta \in \overline{D_{R'}(0)}} \big\Vert {\rm e}^{i\theta A} \vert \varphi_n \rangle \langle \psi_n \vert
{\rm e}^{-i \theta A} \big\Vert < \infty.
$$
Now, let us introduce the family of finite subsets $(\Sigma_n)_{n \in {\mathbb N}}$ defined by $\Sigma_n = \big\{ k \in \Sigma : |k|\leq n \big\}$. 
Thus, we have 
$$
\sum_{n \in \Sigma} \vert \alpha_n \vert \sup_{\theta \in \overline{D_{R'}(0)}} \big\Vert {\rm e}^{i\theta A} \vert \varphi_n \rangle \langle \psi_n \vert
{\rm e}^{-i \theta A} \big\Vert = \sum_{n \in {\mathbb N}} \sum_{\small{k \in \sum_n \setminus \sum_{n-1}}} \vert \alpha_k \vert \sup_{\theta \in 
\overline{D_{R'}(0)}} \big\Vert {\rm e}^{i\theta A} \vert \varphi_k \rangle \langle \psi_k \vert {\rm e}^{-i \theta A} \big\Vert,
$$
and the claim follows from Proposition \ref{pe1}.
\end{proof}

Finally, we conclude this paragraph with the following observation:
\begin{prop}\label{oconnor-0}
Let $R>0$ and $B: D_R(0) \rightarrow {\mathcal B}({\mathscr H})$ be analytic. Assume that for any $(\theta, \theta') \in (-R,R)\times (-R,R)$ such that $\theta+\theta' \in (-R,R)$, it holds:
$$
e^{i\theta' A} B(\theta ) e^{-i\theta' A} = B(\theta+\theta' ) .
$$
Then, 
\begin{itemize}
\item[(a)] given any $\theta' \in (-R,R)$, for any $\theta \in D_R(0) \cap D_R (-\theta')$: $e^{i\theta' A} B(\theta ) e^{-i\theta' A} = B(\theta+\theta' )$,
\item[(b)] for any $(\theta_1, \theta_2) \in D_R (0)$ such that $\im (\theta_1) = \im (\theta_2)$: $\sigma (B(\theta_1 )) = \sigma (B(\theta_2 ))$.
\end{itemize}
If in addition, $B$ is a projection-valued analytic function and $B(0)$ has finite rank, then any vector $\varphi \in$ Ran $B(0)$ is analytic w.r.t. $A$, and the power series 
$$
\sum_{k = 0}^\infty \frac{\vert \theta \vert^k}{k!} \big\Vert A^k \varphi \big\Vert 
$$
converges for any $\theta \in D_R(0)$.
\end{prop}

\noindent
\begin{proof} 
Fix $\theta' \in (-R,R)$. We observe that the maps $B$ and $\theta \mapsto B(\theta +\theta')$ are analytic on the open connected set $D_R(0) \cap D_R (-\theta')$ and coincide by hypothesis on the segment $(-R, R)\cap (-R -\theta', R-\theta')$. So, they coincide on $D_R(0) \cap D_R (-\theta')$, which proves (a). 

Let $(\theta_1, \theta_2) \in D_R (0)$ such that $\im (\theta_1) = \im (\theta_2)$. Define $\theta_0 := (\theta_1+ \theta_2)/2$ and $h:= (\theta_2- \theta_1)/2$. We observe that $\theta_2 -\theta_0 = \theta_0 - \theta_1 =  h \in (-R,R)$. Using (a), we deduce that: $e^{ih A} B(\theta_0 ) e^{-ih A} = B(\theta_2 )$ and $e^{ih A} B(\theta_1 ) e^{-ih A} = B(\theta_0 )$. Statement (b) follows by unitary invariance of the spectrum.

The last statement is a direct application of O'Connor Lemma, \cite{RS4} p.196.
\end{proof}

\begin{rem}\label{oconnor-1} Naturally, if $L\in {\mathcal A}_R (A)$ for some $R>0$, Proposition \ref{oconnor-0} applies to the map $B= a[L]$, with $B(0)=L$.
\end{rem}

\subsubsection{Applications}\label{ssec}

In this paragraph, we consider the case ${\mathscr H} = \ell^2({\mathbb Z})$ and $A=A_0$. We recall that 
$(e_n)_{n \in {\mathbb Z}}$ denotes the canonical 
orthonormal basis of $\ell^2({\mathbb Z})$.

First, we state the following result:
\begin{lem}\label{le2}
For any $n \in {\mathbb Z}$, the series
$$
\sum_{k = 0}^\infty \frac{\vert \theta \vert^k}{k!} \big\Vert A_0^k e_n \big\Vert
$$
converge on $D_\frac{1}{2}(0)$. In particular for any $n \in {\mathbb Z}$, $e_n$ is an analytic vector for $A_0$.
\end{lem}

\noindent
\begin{proof} First, note that for any $n\in {\mathbb Z},$
$$
A_0 e_n = -i \left( n+\frac{1}{2} \right) e_{n+1} +i \left( n-\frac{1}{2} \right) e_{n-1}.
$$
We deduce the following crude combinatorial estimate: for any $k\in {\mathbb Z}_+$, $n \in {\mathbb Z},$
$$
\|A_0^k e_n \| \leq 2^k \frac{(|n|+k)!}{|n|!}.
$$
The conclusion follows. 
\end{proof}

In the sequel, we let $\varphi_\theta$ (resp. $J(\varphi_\theta)$) denote the analytic extension w.r.t. $\theta \in D_{R_0}(0)$, $R_0 \le \frac{1}{2}$, of the 
flow $\varphi_\theta$ (resp. the Jacobian $J(\varphi_\theta)$) defined by \eqref{eq6,4}. Note that by \cite[Lemma 3.3]{sig}, such extensions exist and 
furthermore $R_0$ can be chosen such that $\sqrt{J(\varphi_\theta)}$ is analytic in $D_{R_0}(0)$.

\begin{lem}\label{le3}
Let $0 < R < R_0$. Then, there exists a constant $C > 0$ such that for any $n \in {\mathbb Z}$,
\begin{equation}
\sup_{\theta \in \overline{D_R(0)}} \big\Vert {\rm e}^{i\theta A_0} e_n \big\Vert \le C {\rm e}^{\vert n \vert \sup_{(\theta,\alpha) \in \overline{D_R(0)} \times {\mathbb T}}
\vert \im(\varphi_\theta(\alpha)) \vert}.
\end{equation}
\end{lem}

\noindent
\begin{proof}
First, observe that $e_n$, $n \in {\mathbb Z}$, is an analytic vector for $A_0$ if and only if ${\rm e}^{in \cdot}$, $n \in {\mathbb Z}$, is an analytic vector for $\widehat A_0$. Thus, 
it follows from the uniqueness of the analytic extension and \eqref{eq6,4} that for any $n \in {\mathbb Z}$, 
we have
$$
\left( \e^{i\theta\widehat A_0} {\rm e}^{in \cdot} \right)(\alpha) = {\rm e}^{in \varphi_\theta(\alpha)} \sqrt{J(\varphi_\theta)(\alpha)},
\qquad (\theta,\alpha) \in D_R(0) \times {\mathbb T}.
$$
Moreover, since for $\theta \in D_R(0)$ fixed we have
$$
\big\Vert \e^{i\theta A_0} e_n \big\Vert = \big\Vert \e^{i\theta\widehat A_0} {\rm e}^{in \cdot} \big\Vert \le \sup_{\alpha \in {\mathbb T}} \sqrt{ \vert J(\varphi_\theta)(\alpha) \vert}
{\rm e}^{\vert n \vert \sup_{\alpha \in {\mathbb T}} \vert \im(\varphi_\theta(\alpha)) \vert},
$$
then the claim follows.
\end{proof}

\begin{prop}\label{pe2}
Let $\Sigma = {\mathbb N}$ or ${\mathbb Z}$ and $(\alpha_n)_{n \in \Sigma}$ be a complex sequence satisfying $\sup_{n \in \Sigma} {\rm e}^{\delta \vert n \vert} \vert \alpha_n \vert < \infty$, for some constant $\delta > 0$. Then, there exists $R_\delta > 0$ such that:
\begin{itemize}
\item[(a)] The operator $\sum_{n \in \Sigma} \alpha_n \vert e_n \rangle \langle e_n \vert$ belongs to ${\mathcal A}_{R_\delta}(A_0)$ with extension given by
$$
\theta \longmapsto \sum_{n \in \Sigma} \alpha_n {\rm e}^{i\theta A_0} \vert e_n \rangle \langle e_n \vert {\rm e}^{-i\theta A_0}.
$$
\item[(b)] For any $m \in \Sigma$ fixed, the operator $\sum_{n \in \Sigma} \alpha_n \vert e_n \rangle \langle e_m \vert$ belongs to ${\mathcal A}_{R_\delta}(A_0)$ with extension given by
$$
\theta \longmapsto \sum_{n \in \Sigma} \alpha_n {\rm e}^{i\theta A_0} \vert e_n \rangle \langle e_m \vert {\rm e}^{-i\theta A_0}.
$$
\end{itemize}
\end{prop}

\noindent
\begin{proof}
We only prove the first point. According to Lemma \ref{le2} and Theorem \ref{te1}, it is enough to show that there 
exists $0 < R_\delta < 1/2$ such that
\begin{equation}\label{hye2'}
\sum_{n \in \Sigma} \vert \alpha_n \vert \sup_{\theta \in \overline{D_{R_\delta}(0)}} \big\Vert {\rm e}^{i\theta A_0} e_n \big\Vert \big\Vert {\rm e}^{i \bar \theta A_0} 
e_n \big\Vert < \infty.
\end{equation}
Let $R_0$ be a constant as in Lemma \ref{le3}, and note that for any constant $\Gamma > 0$, there exists $0 < R_\Gamma < R_0$ such 
that we have 
\begin{equation}\label{eq:g}
\sup_{(\theta,\alpha) \in \overline{D_{R_\Gamma}(0)} \times {\mathbb T}} \vert \im(\varphi_\theta(\alpha)) \vert < \Gamma.
\end{equation}
In particular, by taking $\Gamma = \delta/4$ in \eqref{eq:g} and by using Lemma \ref{le3}, it follows that there exists $0 < R_\delta < R_0$ such that
\begin{equation}\label{eqno}
\sup_{\theta \in \overline{D_{R_\delta}(0)}} \big\Vert {\rm e}^{i\theta A_0} e_n \big\Vert \le C {\rm e}^{\frac{\delta}{4}\vert n \vert}, \qquad C > 0.
\end{equation}
Since $\sup_{n \in \Sigma} {\rm e}^{\delta \vert n \vert} \vert \alpha_n \vert < \infty$, this implies \eqref{hye2'} and the claim holds.
\end{proof}

\begin{prop}\label{pe3}
Let $B\in {\mathcal B}(\ell^{2}({\mathbb Z}))$. Assume there exists $\delta > 0$ such that
\begin{equation}\label{eq1,7++}
\sup_{(n,m) \in {\mathbb Z}^2} {\rm e}^{\delta (\vert n \vert + \vert m \vert)} \big\vert B(n,m) \big\vert < \infty ,
\end{equation}
where $B(n,m) := \bra e_n, B e_m \ket$ for $(n,m)\in {\mathbb Z}^2$. Then, there exists $R_\delta > 0$ such that $B$ belongs to ${\mathcal A}_{R_\delta}(A_0)$ with extension given by
$$
\theta \longmapsto \sum_{n \in {\mathbb Z}} \sum_{m \in {\mathbb Z}} B(n,m) \, {\rm e}^{i\theta A_0} 
\vert e_n \rangle \langle e_m \vert {\rm e}^{-i\theta A_0}.
$$
\end{prop}

\noindent
\begin{proof}
Canonically, the operator $B$ can be written as 
$$
B = \sum_{n \in {\mathbb Z}} \sum_{m \in {\mathbb Z}} B(n,m) \vert e_n \rangle \langle e_m \vert
= \sum_{m \in {\mathbb Z}} L_m,
$$
where $L_m := \sum_{n \in {\mathbb Z}} B(n,m) \vert e_n \rangle \langle e_m \vert$. Due to \eqref{eq1,7++}, it follows from Proposition \ref{pe2} (ii) that $L_m \in {\mathcal A}_{R_\delta}(A_0)$ for some $R_\delta > 0$, with extension given by
$$
\theta \longmapsto \sum_{m \in {\mathbb Z}} B(n,m) \, {\rm e}^{i\theta A_0} 
\vert e_n \rangle \langle e_m \vert {\rm e}^{-i\theta A_0}.
$$
Now, let $\Sigma_n$, $n \in {\mathbb N}$, be the finite subsets defined by $\Sigma_n = \big\{ k\in {\mathbb Z} : |k|\leq n \big\}$. We have
\begin{align*}
\sum_{n \in {\mathbb N}} \sum_{\small{m \in \sum_n \setminus \sum_{n-1}}} 
\sup_{\theta \in \overline{D_{R_\delta}(0)}} & \big\Vert {\rm e}^{i\theta A_0} L_m {\rm e}^{-i \theta A_0} \big\Vert = \sum_{m \in {\mathbb Z}} \sup_{\theta \in 
\overline{D_{R_\delta}(0)}} \big\Vert {\rm e}^{i\theta A_0} L_m {\rm e}^{-i \theta A_0} \big\Vert \\
& \le \sum_{n \in {\mathbb Z}} \sum_{m \in {\mathbb Z}} \vert B(n,m) \vert \sup_{\theta \in \overline{D_{R_\delta}(0)}} \big\Vert {\rm e}^{i\theta A_0} \vert e_n \rangle 
\langle e_m \vert {\rm e}^{-i \theta A_0} 
\big\Vert \\
& = \sum_{n \in {\mathbb Z}} \sum_{m \in {\mathbb Z}} \vert B(n,m) \vert \sup_{\theta \in \overline{D_{R_\delta}(0)}} \big\Vert {\rm e}^{i\theta A_0} e_n \big\Vert 
\big\Vert {\rm e}^{i \bar \theta A_0} e_m \big\Vert < \infty,
\end{align*}
where $R_\delta$ has been chosen as in the proof of Proposition \ref{pe2}, so that
\eqref{eqno} holds. Thus, the claim follows from Proposition \ref{pe1}.
\end{proof}

\begin{cor}\label{analyticv} Let $\varphi \in {\mathscr H}$. Suppose there exists $\beta >0$ such that: $\sup_{n\in {\mathbb Z}} e^{\beta |n|} |\bra \varphi, e_n \ket | < \infty$. Then, $\varphi$ is an analytic vector for $A_0$ and the power series 
$$
\sum_{k = 0}^\infty \frac{\vert \theta \vert^k}{k!} \big\Vert A_0^k \varphi \big\Vert 
$$
converges for any $| \theta | < R_{\beta}$, for some $R_{\beta} >0$.
\end{cor}

\noindent
\begin{proof} Write $L= |\varphi \ket \bra \varphi |$. For all $(n,m)\in {\mathbb Z}^2$, $L(n,m) := \bra e_n, L e_m \ket = \overline{\bra \varphi, e_n \ket} \bra \varphi, e_m \ket$. In particular, $$\sup_{(n,m) \in {\mathbb Z}^2} {\rm e}^{\beta (\vert n \vert + \vert m \vert)} \big\vert L(n,m) \big\vert < \infty .$$Applying Proposition \ref{pe3}, we deduce $L\in {\mathcal A}_{R_\beta}(A_0)$ for some $R_{\beta} >0$. The conclusion follows from Proposition \ref{oconnor-0} and Remark \ref{oconnor-1}.
\end{proof}

\subsection{The $C^k(A)$ and ${\mathcal C}^{s,p}(A)$ classes}\label{ecsmoo}

The classes $C^k(A)$ and ${\mathcal C}^{1,1}(A)$, $k\in {\mathbb N}$, have been introduced in Definitions \ref{ck} and \ref{c11}. Let us consider also
the following definition:

\begin{de}\label{c01}
Let ${\mathscr H}$ be a Hilbert space and $A$ be a selfadjoint operator defined on ${\mathscr H}$. An operator $B \in \mathcal{B}({\mathscr H})$ 
belongs to ${\cal C}^{0,1}(A)$ if
\begin{equation*}
\int_0^1 \big\| e^{iA\theta}B e^{-iA\theta}-B \big\| \,\frac{d\theta}{\theta} < \infty .
\end{equation*}
\end{de}
We recall that these classes are linear subspaces of ${\mathcal B}({\mathscr H})$, stable under adjunction, according to \cite[Chapter 5]{abmg}. The results exposed in this paragraph are meant to illustrate Theorem \ref{c11sign}.

We start by introducing the sequence subspaces:
\begin{equation*}
\begin{split}
{\mathcal Q}_k ({\mathbb Z}) & := \Bigg\{ x \in {\mathbb C}^{\mathbb Z} : \sum_{j=0}^k q_j(x) < \infty \Bigg\},\quad k\in \{0,1,2\},\\
c_0 (\mathbb Z) & := \Big\{ x \in {\mathbb C}^{\mathbb Z} : \lim_{|n| \rightarrow \infty} x(n) =0 \Big\}, 
\end{split}
\end{equation*}
where the maps $q_k : {\mathbb C}^{\mathbb Z} \rightarrow [0,\infty]$, $k\in \{0,1,2\}$ are defined by $q_0 (x) := \sup_{n \in {\mathbb Z}} \big| x(n) \big|$, 
$q_1(x) := \sup_{n \in {\mathbb Z}} \big| n ( x(n+1)-x(n) ) \big|$ and $q_2(x) := \sup_{n \in {\mathbb Z}} \big| n^2 ( x(n+1)-2x(n)+x(n-1) ) \big|$. 
Direct calculations yield:

\begin{lem}\label{mock} Consider the linear operator $V$ defined on the canonical orthonormal basis of $\ell^2({\mathbb Z})$ by $Ve_n =v_n e_n$, $n\in {\mathbb Z}$. Denote $v = (v_n)_{n\in {\mathbb Z}}$.
\begin{enumerate}
\item[(a)] If $v\in {\mathcal Q}_k ({\mathbb Z})$ for some $k\in \{0,1,2\}$, then $V\in C^k(A_0)$.
\item[(b)] If $v\in c_0({\mathbb Z})$, then $V$ is compact.
\end{enumerate}
\end{lem}
\begin{rem} In Lemma \ref{mock}, we have: $C^0(A_0) = {\mathcal B} (l^2({\mathbb Z}))$.
\end{rem}

\begin{lem}\label{rank1ck} Let $N\in {\mathbb N}$ and consider two finite families of vectors $(\varphi_k)_{k=1}^N$ and $(\psi_k)_{k=1}^N$ whose elements belong to ${\mathcal D}(A_0^k)$ for some $k\in {\mathbb N}$. Then, for any $(\beta_1, \ldots, \beta_N)\in {\mathbb C}^N,$
\begin{equation}
V=\sum_{k=1}^N \beta_k |\psi_k\ket \bra \varphi_k |\in C^k(A_0) \, .
\end{equation}
\end{lem}

Now, consider the sets of sequences ${\mathcal S} ({\mathbb Z})=\cup_{0 < a < b <\infty} {\mathcal S}_{a,b} ({\mathbb Z})$ and ${\mathcal M}=\cup_{0 < a < b<\infty}  {\mathcal M}_{a,b} ({\mathbb Z})$, where
\begin{equation}
\begin{split}
{\mathcal S}_{a,b} ({\mathbb Z}) & := \Bigg\{ x \in {\mathbb C}^{\mathbb Z} : \int_1^{\infty} \sup_{a r\leq n \leq b r} |x(n) |\, dr < \infty \Bigg\},\\
{\mathcal M}_{a,b} ({\mathbb Z}) & := \Bigg\{ x \in {\mathbb C}^{\mathbb Z} : \int_1^{\infty} \sup_{a r\leq n \leq b r} |x(n+1) -x(n)|\, dr < \infty \Bigg\} .
\end{split}
\end{equation}
We have:
\begin{prop}\label{moc11} Consider the linear operator $V$ defined on the canonical orthonormal basis of $\ell^2({\mathbb Z})$ by $Ve_n =v_n e_n$, $n\in {\mathbb Z}$. Denote $v = (v_n)_{n\in {\mathbb Z}}$.
\begin{enumerate}
\item[(a)] If $v \in {\mathcal S} ({\mathbb Z})$, then $V \in {\mathcal C}^{1,1}(A_0)$.
\item[(b)] If $v \in {\mathcal M} ({\mathbb Z}) \cap {\mathcal Q}_1 ({\mathbb Z})$, then $V \in C^1(A_0)$ and $\mathrm{ad}_{A_0} (V) \in {\mathcal C}^{0,1}(A_0)$. In particular, $V \in {\mathcal C}^{1,1}(A_0)$.
\end{enumerate}
\end{prop}
\noindent
\begin{proof} 
The first statement follows from \cite[Theorem 7.5.8]{abmg}. The second statement, which is based on the inclusions 5.2.19 in \cite{abmg}, was proved in \cite{sah}.
\end{proof}

Finally, let us consider the subset of vectors ${\mathcal D} = \cup_{0 < a < b <\infty} {\mathcal D}_{a,b}$, where
\begin{equation}
{\mathcal D}_{a,b} := \Bigg\{ \xi \in \ell^2({\mathbb Z}) : \int_{1}^{\infty} ( \sum_{n \in {\mathbb N}\cap [ar,br]}  | \xi_n |^2 )^{1/2}\, dr < \infty \Bigg\} .
\end{equation}
We have:
\begin{prop}\label{rank1c11} Let $N\in {\mathbb N}$ and consider two finite families of vectors $(\varphi_k)_{k=1}^N$ and $(\psi_k)_{k=1}^N$ whose elements belong to ${\mathcal D}$. Then, for any $(\beta_1, \ldots, \beta_N)\in {\mathbb C}^N,$
\begin{equation}
V=\sum_{k=1}^N \beta_k |\psi_k\ket \bra \varphi_k |\in {\mathcal C}^{1,1}(A_0) \, .
\end{equation}
\end{prop}
We refer to \cite[Lemmata 3.13 and 3.14]{abc} for the proof.

\begin{rem} Since ${\mathcal C}^{1,1}(A_0)$ is a linear subspace of ${\mathcal B} ( \ell^2({\mathbb Z}) )$, stable under adjunction, which also contains $C^2(A_0)$, it also possible to combine Lemmata \ref{mock}, \ref{rank1ck} with Propositions \ref{moc11}, \ref{rank1c11} to obtain more elaborated examples. For example,
\begin{enumerate}
\item if $v \in {\mathcal S} ({\mathbb Z}) + ( {\mathcal M} ({\mathbb Z}) \cap {\mathcal Q}_1 ({\mathbb Z}) ) + 
{\mathcal Q}_2 ({\mathbb Z})$, then the linear operator $V$ defined on the canonical orthonormal basis of $\ell^2({\mathbb Z})$ by 
$Ve_n =v_n e_n$, $n\in {\mathbb Z}$ belongs to ${\mathcal C}^{1,1}(A_0)$,
\item if $\varphi \in {\mathcal D}(A_0^2)$ and $\psi \in {\mathcal D}$ (or vice-versa), then for any $\beta \in {\mathbb C}$, $V= \beta |\psi\ket \bra \varphi |$ belongs to ${\mathcal C}^{1,1}(A_0)$.
\end{enumerate}
\end{rem}

\noindent
{\bf Acknowledgements:} O. Bourget is supported by the Chilean Fondecyt Grant 1161732. 
D. Sambou is supported by the Chilean Fondecyt Grant 3170411. The authors also thank the referees for their constructive comments.


\end{document}